\documentclass[ejs, preprint]{imsart}

\doi{10.1214/16-EJS1186}
\pubyear{2016}
\volume{10}
\firstpage{2584}
\lastpage{2640}

\usepackage{amsmath}
\usepackage{amsthm}
\usepackage{amssymb}
\usepackage{footmisc}
\usepackage[numbers]{natbib}
\usepackage{hyperref}
\usepackage{hypernat}
\usepackage{verbatim}
\usepackage{textcomp}
\usepackage{enumitem}
\usepackage{graphicx}
\usepackage{fancyhdr}

\startlocaldefs

\newcommand{\R}{{\mathbb R}}
\newcommand{\E}{{\mathbb E}}
\renewcommand{\P}{{\mathbb P}}
\newcommand{\N}{{\mathbb N}}

\newcommand{\eps}{\varepsilon}

\DeclareMathOperator{\Var}{Var}
\DeclareMathOperator{\Cov}{Cov}
\renewcommand{\det}[1]{{\operatorname{\text{det}}}{\left(#1\right)}}
\newcommand{\trace}[1]{{\operatorname{\text{trace}}}{\left(#1\right)}}

\DeclareMathOperator{\rank}{rank}
\DeclareMathOperator{\argmin}{argmin}
\DeclareMathOperator{\s}{span}
\newcommand{\lmax}[1]{{\operatorname{\lambda_{\text{max}}}}{\left(#1\right)}}
\newcommand{\lmin}[1]{{\operatorname{\lambda_{\text{min}}}}{(#1)}}

\newtheorem{theorem}{Theorem}[section]

\newtheorem{lemma}[theorem]{Lemma}
\newtheorem{corollary}[theorem]{Corollary}
\newtheorem{remark}[theorem]{Remark}
\newtheorem*{remark*}{Remark}

\newtheorem{condition*}{Condition}

\numberwithin{equation}{section}

\newcounter{rcnt}[section]

\newcommand{\rem}[1]{}

 \endlocaldefs

\begin{document}

\begin{frontmatter}

\title{The relative effects of dimensionality and multiplicity of hypotheses on the $F$-test in linear regression
}
\runtitle{Dimensionality and multiplicity of hypotheses in the $F$-test}


\author{\fnms{Lukas} \snm{Steinberger}\ead[label=e1]{lukas.steinberger@univie.ac.at} }
\address{Department of Statistics and Operations Research\\
	University of Vienna  \\
	Oskar-Morgenstern-Platz 1 \\
	1090 Vienna, Austria\\ \printead{e1}}

\runauthor{L. Steinberger}

\begin{abstract}
Recently, several authors have re-examined the power of the classical $F$-test in a non-Gaussian linear regression under a ``large-$p$, large-$n$'' framework \citep[e.g.][]{Zhong11, Wang13}. They highlight the loss of power as the number of regressors $p$ increases relative to sample size $n$. These papers essentially focus only on the overall test of the null hypothesis that all $p$ slope coefficients are equal to zero. Here, we consider the general case of testing $q$ linear hypotheses on the $p+1$-dimensional regression parameter vector that includes $p$ slope coefficients and an intercept parameter. In the case of Gaussian design, we describe the dependence of the local asymptotic power function on both the relative number of parameters $p/n$ and the relative number of hypotheses $q/n$ being tested, showing that the negative effect of dimensionality is less severe if the number of hypotheses is small. Using the recent work of \citet[][]{Sriva13} on high-dimensional sample covariance matrices we are also able to substantially generalize previous results for non-Gaussian regressors. 
\end{abstract}




\tableofcontents

\end{frontmatter}

\section{Introduction}
\label{sec:intro}

Following a suggestion of \citet{Bai96} to investigate classical statistical procedures in high-dimensional settings, \citet{Wang13} re-examine the usual $F$-test in the linear regression model under a ``large $p$, large $n$'' asymptotic framework. They derive the asymptotic power in a fairly general, non-Gaussian setting, highlighting the dependence of the local power function on the dimensionality of the problem, i.e., on the limit $\rho  = \lim p/n \in (0,1)$, where $n$ is sample size, and $p$ is the number of regressors in the model. In particular, they find that the rejection probability of the $F$-test for $H_0 : R\beta=r_0$, where $R = [0,I_q]$ and $p/n\to\rho$, $q/n\to\rho$, satisfies
\begin{align}\label{eq:WCresult}
\P(F_n>f_{q,n-p-1}^{(1-\nu)}) - \Phi\left( -\zeta_{1-\nu}  + \sqrt{n}\Delta_\beta \sqrt{\frac{1-\rho}{2\rho}}\right) \; \xrightarrow[]{} \; 0.
\end{align}
Here $F_n$ is the usual $F$-statistic, $f_{q,n-p-1}^{(1-\nu)}$ is the appropriate $F$-quantile, $\Phi$ is the cdf of the standard normal distribution, $\zeta_{1-\nu} = \Phi^{-1}(1-\nu)$ and $\Delta_\beta = (R\beta-r_0)'(R\Sigma^{-1}R')^{-1}(R\beta-r_0)/\sigma^2$ is the scaled distance from the null hypothesis.
From this approximation we see that the local asymptotic power of the $F$-test depends monotonically on the value of $\rho$ and inflates to the nominal significance level $\nu$ as $\rho$ increases to one. The result of \citet{Wang13} is consistent with the derivations of the local asymptotic power in the case of Gaussian errors, as obtained by \citet{Zhong11}. Both of these studies consider only the overall $F$-test for the null hypothesis that all, or almost all (cf. Condition~(C3) in \citet{Wang13}) of the $p$ slope coefficients are equal to zero. Also, they do not consider hypotheses involving the intercept parameter. Here, we extend this analysis and study the problem of testing $q$ general linear hypothesis (including also hypotheses on the intercept term), without the restriction that $(p-q)/n\to 0$. In this sense, we examine the effect of the dimension of the null hypothesis (i.e., the number of linear restrictions being tested) on the local asymptotic rejection probability of the $F$-test. We find that when testing the null hypothesis $H_0: R_0\gamma = r_0$, for some $q\times (p+1)$ matrix $R_0$ of rank $q\le p+1$, such that $p/n\to\rho_1$ and $q/n\to\rho_2\le\rho_1$, the rejection probability of the $F$-test satisfies
\begin{align}\label{eq:OURresult}
\P(F_n>f_{q,n-p-1}^{(1-\nu)}) - \Phi\left( -\zeta_{1-\nu}  +  \sqrt{n}\Delta_\gamma \sqrt{\frac{(1-\rho_1)(1-\rho_1+\rho_2)}{2\rho_2}}\right) \; \to \;0.
\end{align}
Now the asymptotic rejection probability depends also on the mean $\mu\in\R^p$ of the random design through $\Delta_\gamma = (R_0\gamma-r_0)'(R_0S^{-1}R_0')^{-1}(R_0\gamma-r_0)/\sigma^2$, where $\gamma = (\alpha,\beta')'$ is the vector of regression coefficients including an intercept parameter $\alpha\in\R$ and $$S = \begin{bmatrix} 1 &\mu'\\ \mu &\Sigma+\mu\mu'\end{bmatrix}.$$
This limiting expression coincides with that in \eqref{eq:WCresult} if $\rho_1=\rho_2$ and $R_0 = [0,R]$. But \eqref{eq:OURresult} refines the statement in \eqref{eq:WCresult} and shows the impact of both the relative number of regressors $\rho_1$ and the relative number of hypotheses $\rho_2$. These quantities affect the asymptotic rejection probability monotonically, which is consistent with small sample analyses in the Gaussian error case \citep[cf.][]{Ghosh73}. However, in contrast to the complicated nature of the cdf of the non-central $F$-distribution as a function in $p$, $q$ and the non-centrality parameter, our asymptotic approximation to the rejection probability depends on the quantities $\rho_1$, $\rho_2$ and $\Delta_\gamma$ only through elementary operations and the Gaussian cdf, and it is valid for a large class of error distributions. In particular, we see that even if $\rho_1$ is close to $1$, the $F$-test still has power if $\rho_2$ is sufficiently small. 
In a second step, and under slightly more restrictive assumptions on the data generating process, we also investigate the case where only a very small relative number of hypotheses $q/n$ is tested, i.e., $q$ is bounded as $n\to\infty$ and $\rho_2=0$, and the result in \eqref{eq:OURresult} no longer holds.

Our work heavily builds on the ideas of \citet{Wang13} (hereafter abbreviated as WC). The first part of the present work is concerned with reproducing their results under substantially more general assumptions. First of all, here we do not require independence between the random design and the error terms, but we assume only the usual first and second order specification of conditional moments of the errors given the design. This extension requires a slight modification of the result of \citet{Bhansali07} on the asymptotic normality of certain quadratic forms as applied by WC (cf. Lemma~\ref{lemma:Bhansali}). Furthermore, we do not assume that the $n\times p$ design matrix $X$, after standardization, consists of i.i.d. components, as is needed for the application of the famous Bai-Yin Theorem \citep{Bai93} used by WC in order to control extreme eigenvalues of large sample covariance matrices. Instead, we apply a recent result of \citet{Sriva13} which essentially requires only certain moment restrictions on the i.i.d. rows of $X$. For our extensions, we also develop a novel result on the diagonal entries of a fairly general random projection matrix that might be of interest on its own (see Lemma~\ref{lemma:ProjDiag}). It has the statistical interpretation that in a moderately high-dimensional regression the leverage values $h_i$, i.e., the diagonal entries of the projection matrix $U(U'U)^{-1}U'$, where $U=[\iota,X]$ is the design matrix including an intercept column, typically behave like $p/n$. Finally, we point out that since we also consider tests on the intercept parameter, the distribution of the $F$-statistic, in general, also depends on the mean $\mu$ of the random design vectors $x_1,\dots, x_n$. This causes certain technical complications due to non-centrality issues which are often avoided in the literature on random design regression by restricting to the case $\mu=0$. Here, we present a detailed treatment of the general case.

The paper is organized as follows. Section~\ref{sec:main} introduces the setup and notation and presents our main results in Theorem~\ref{thm:main} and Corollary~\ref{corr:main}, which provide a precise formulation of the statement in \eqref{eq:OURresult}. In Section~\ref{sec:smallq}, we specifically consider the situation where $q$ is fixed and also provide a unifying result that does not distinguish between large or small $q$. Next, in Section~\ref{sec:DiscAss} we provide a detailed discussion of our technical assumptions and explain the main differences to those imposed by WC. The results of an extensive numerical study are reported in Section~\ref{sec:sim}. Finally, Section~\ref{sec:proof} provides the basic steps in the proof of our main results. Some of the more technical arguments are deferred to the appendices.

%
%


\section{Model formulation and main results}
\label{sec:main}

We consider a random array $\{ (y_{i,n},x_{i,n}) : 1\le i\le n, n\ge 1\}$ where, for each $n\in\N$, the pairs $(y_{i,n},x_{i,n})_{i=1}^n$ are i.i.d. observations of a real valued response variable $y_{1,n}$ and $p_n$-dimensional random regressors $x_{1,n}$ with $p_n<n-1$, satisfying $\E[y_{1,n}|x_{1,n}] = \alpha_n + \beta_n'x_{1,n}$ and $\Var[y_{1,n}|x_{1,n}] = \sigma^2_n\in(0,\infty)$. Equivalently, writing $\eps_{i,n} = y_{i,n} - \E[y_{i,n}|x_{i,n}]$, the observations can be represented as
\begin{equation}
\label{eq:linmod}
y_{i,n} = \alpha_n + \beta_n'x_{i,n} + \eps_{i,n}, \hspace{1cm} i=1,\dots, n,
\end{equation}
where the $(\eps_{i,n})_{i=1}^n$ are i.i.d., satisfying $\E[\eps_{i,n}|x_{i,n}] = 0$ and $\Var[\eps_{i,n}|x_{i,n}] = \sigma_n^2$. Note that $\eps_{1,n}$ does not need to be independent of $x_{1,n}$. 
For identifiability, we also assume that $\Sigma_n := \Var[x_{1,n}]$ is positive definite and we define $\mu_n := \E[x_{1,n}]$. Furthermore, we adopt the matrix notation $Y_n = (y_{1,n},\dots,y_{n,n})'$, $X_n = [x_{1,n},\dots, x_{n,n}]'$, $\eps_n = (\eps_{1,n},\dots, \eps_{n,n})'$, $\gamma_n = (\alpha_n,\beta_n')'$ and $U_n = [\iota_n,X_n]$, where $\iota_n = (1,1,\dots,1)' \in\R^n$. For notational convenience we will drop the subscript $n$ whenever there is no risk of confusion, i.e., we write $Y=Y_n$, $X=X_n$, $\alpha=\alpha_n$, $\beta=\beta_n$, etc., keeping in mind that, unless noted otherwise, all quantities to follow depend on sample size $n$. With this, the model equations in \eqref{eq:linmod} become
\begin{equation}
Y\; =\; U\gamma\; +\; \eps.
\end{equation}

We want to test a general linear hypothesis on the coefficients $\gamma$, i.e.,
\begin{equation}\label{eq:H0}
H_0 : R_0\gamma = r_0\quad \text{vs.} \quad H_1 : R_0\gamma \ne r_0,
\end{equation}
where $R_0$ is a $q\times (p+1)$ matrix with $\rank{R_0} = q \le p+1$ and $r_0\in\R^q$. Without restriction we may assume that $R_0$ has orthonormal rows (premultiply \eqref{eq:H0} by $(R_0R_0')^{-1/2}$). We test $H_0$ by use of the $F$-statistic $F_n$ defined as 
\begin{equation}
\label{eq:Fstat}
F_n \;=\; \frac{(R_0\hat{\gamma}_n - r_0)'(R_0(U'U)^{-1}R_0')^{-1}(R_0\hat{\gamma}_n-r_0)/q}{(Y-U\hat{\gamma}_n)'(Y-U\hat{\gamma}_n)/(n-p-1)},
\end{equation} 
provided that all the appearing quantities are well defined, and $F_n = 0$, otherwise.
The $F$-statistic is then compared to the $1-\nu$ quantile of an $F$-distribution with $q$ and $n-p-1$ degrees of freedom, which we denote by $f_{q,n-p-1}^{(1-\nu)}$. Here, $\hat{\gamma}_n = (\hat{\alpha}_n,\hat{\beta}_n')'$ is the OLS estimate in the unrestricted model. We also define the usual estimator of the error variance $\hat{\sigma}_n^2 = \|Y-U\hat{\gamma}_n\|^2/(n-p-1)$, that appears in the denominator of the $F$-statistic.

In Section~\ref{sec:proof} we prove the following results, involving the scaled distance from the null hypothesis 
$$\Delta_\gamma := (R_0\gamma - r_0)'(R_0S^{-1}R_0')^{-1}(R_0\gamma-r_0)/\sigma^2,
$$ 
where 
$$
S = \begin{bmatrix} 1 &\mu'\\ \mu &\Sigma+\mu\mu'\end{bmatrix} 
= \E\left[ \begin{pmatrix}1\\x_1\end{pmatrix} \begin{pmatrix}1 &x_1'\end{pmatrix} \right]
= \E\left[ U'U/n\right].
$$ A list of further technical conditions is given below in Section~\ref{sec:Assumptions}.

\begin{theorem}
\label{thm:main}
Consider the linear, homoskedastic model \eqref{eq:linmod} and set $s_n = 2 \left(\frac{1}{q_n}+\frac{1}{n-p_n-1}\right)$ and $b_n = \sqrt{\frac{(1-(p_n+1)/n)(1-(p_n+1)/n+q_n/n)}{2q_n/n}}$. Moreover, suppose that $\limsup_n p_n/n<1$, $q_n\to\infty$ and $\Delta_\gamma = o(q_n/n)$ as $n\to\infty$. If either one of the following three cases applies then the $F$-statistic satisfies
\begin{equation}\label{eq:thm:main}
s_n^{-1/2}(F_n-1) -  \sqrt{n}\Delta_\gamma b_n \quad \xrightarrow[n\to\infty]{w} \quad\mathcal N(0,1).
\end{equation}

\begin{enumerate}
	\setlength\leftmargin{-20pt}
	\renewcommand{\theenumi}{(\roman{enumi})}
	\renewcommand{\labelenumi}{{\theenumi}} 
\item \label{thm:R=I}
The assumptions~\ref{a.design}.(\ref{a.factor},\ref{a.invertibility},\ref{a.Srivastava},\ref{a.moments1}) and \ref{a.error} on the random design and on the error distribution, are satisfied, and either $R_0=I_{p_n+1}$ for all $n\in\N$, or $R_0=[0,I_{p_n}]$ for all $n\in\N$. (In this case, either $q_n=p_n$ or $q_n=p_n+1$, and thus $(p_n-q_n)/n \to 0$ holds.)

\item \label{thm:A}
The assumptions~\ref{a.design}.(\ref{a.factor},\ref{a.invertibility},\ref{a.Srivastava},\ref{a.moments1},\ref{a.moments2}) and  \ref{a.error} on the random design and on the error distribution, are satisfied, $(p_n-q_n)/n \to 0$, and $(R_0\gamma - r_0)' R_0 S R_0' (R_0\gamma - r_0)/\sigma^2 = O(\sqrt{q_n}/n)$ holds.\footnote{Notice that this additional requirement implies and strengthens the assumption that $\Delta_\gamma = o(q_n/n)$. Simply observe that, by block matrix inversion, $R_0S^{-1}R_0' = (R_0S R_0' - R_0S R_1'(R_1 S R_1')^{-1}R_1 S R_0')^{-1}$, where $R_1$ is a $(p_n+1-q_n)\times (p_n+1)$ matrix with orthonormal rows which are also orthogonal to the rows of $R_0$. Therefore, $\Delta_\gamma \le (R_0\gamma - r_0)' R_0 S R_0' (R_0\gamma - r_0)/\sigma^2= O(\sqrt{q_n}/n)$.}

\item \label{thm:B}
Assumption~\ref{a.error} on the error distribution is satisfied and the design vectors $x_{1,n},\dots, x_{n,n}$ are i.i.d. Gaussian with mean $\mu_n\in\R^{p_n}$ and positive definite covariance matrix $\Sigma_n$.\footnote{It is easy to see that the normality assumption implies Assumption~\ref{a.design}.}
\end{enumerate}
\end{theorem}

By a simple argument involving Polya's theorem this translates into the following corollary on the rejection probability of the $F$-test.

\begin{corollary}
\label{corr:main}
If $\limsup_n p_n/n<1$ and $q_n\to\infty$, as $n\to\infty$, and the conclusion of Theorem~\ref{thm:main} holds, then the rejection probability of the $F$-test satisfies
\begin{equation*}
\P(F_n>f_{q_n,n-p_n-1}^{(1-\nu)}) - \Phi\left( -\zeta_{1-\nu}  +  \sqrt{n}\Delta_\gamma \sqrt{\frac{(1-\frac{p_n+1}{n})(1-\frac{p_n+1}{n}+\frac{q_n}{n})}{2q_n/n}}\right) \; \to \;0.
\end{equation*}
Here, $\zeta_{1-\nu} = \Phi^{-1}(1-\nu)$ is the $1-\nu$ quantile of the standard normal distribution and $\nu\in(0,1)$ does not depend on $n$.
\end{corollary}

\begin{proof}
It is easy to see, using Polay's theorem and Lemma~\ref{lemma:NCFdist}, that $\tilde{f}_n := s_n^{-1/2}(f_{q_n,n-p_n-1}^{(1-\nu)} - 1)$ satisfies
$\tilde{f}_n \to \zeta_{1-\nu}$. Now use the conclusion of Theorem~\ref{thm:main}, Polya's theorem and the Lipschitz continuity of $\Phi$ to obtain
\begin{align*}
&|\P(F_n>f_{q,n-p-1}^{(1-\nu)}) - \Phi(-\zeta_{1-\nu}+\eta_n^2)|\\
&\quad=| \P(s_n^{-1/2}(F_n-1)>\tilde{f}_n) - \Phi(-\zeta_{1-\nu} + \eta_n^2)| \\
&\quad=| \P(s_n^{-1/2}(F_n-1)-\eta_n^2\le\tilde{f}_n-\eta_n^2) - \Phi(\zeta_{1-\nu} - \eta_n^2)| \\
&\quad\le
| \P(s_n^{-1/2}(F_n-1)-\eta_n^2\le\tilde{f}_n-\eta_n^2) - \Phi(\tilde{f}_n - \eta_n^2)| \\
&\quad\quad+
| \Phi(\tilde{f}_n - \eta_n^2) - \Phi(\zeta_{1-\nu} - \eta_n^2)| \\
&\quad\le
\sup_{t\in\R}| \P(s_n^{-1/2}(F_n-1)-\eta_n^2 \le t) - \Phi(t)| + o(1) \xrightarrow[n\to\infty]{} 0,
\end{align*} 
where $\eta_n^2 := \sqrt{n}\Delta_\gamma b_n$.
\end{proof}

If $(p_n-q_n)/n\to 0$ and $0<\liminf_n q_n/n \le \limsup_n p_n/n <1$, then Corollary~\ref{corr:main} recovers the result of WC under weaker assumptions on the joint distribution of the design and the errors, and for a null hypothesis that possibly restricts also the intercept parameter $\alpha$ (cf. the assumptions of Theorem~\ref{thm:main}\ref{thm:R=I} and \ref{thm:A}). In this case the factor $b_n$ above asymptotically reduces to 
$$ b_n = \sqrt{\frac{(1-\frac{p_n+1}{n})(1-\frac{p_n+1}{n}+\frac{q_n}{n})}{2q_n/n}} \thicksim \sqrt{\frac{1-\frac{p_n+1}{n}}{2\frac{p_n+1}{n}}},$$ as in \eqref{eq:WCresult}. It highlights the dependence of the power function on the relative number of regressors $p_n/n$. However, since $(p_n-q_n)/n\to 0$, the individual roles of $p_n$ and $q_n$ can not be discerned. This shortcoming is removed here, but it comes at the price of a stronger design condition (cf. Theorem~\ref{thm:main}\ref{thm:B}). It is tempting, however, to conjecture that Assumptions~\ref{a.design} and \ref{a.error} are actually sufficient also for the general case. Corollary~\ref{corr:main} nicely shows the effect of both the dimension of the parameter space as well as the dimension of the null hypothesis, on the asymptotic power function. In particular, we see that even in a case where the relative number of regressors $p_n/n$ is large, the classical $F$-test still has power, as long as we are interested in testing only a relatively small number of hypotheses. 
However, we should make a cautionary remark at this point. In Theorem~\ref{thm:main} and Corollary~\ref{corr:main} we have assumed that $q_n\to\infty$. If the number of hypotheses $q$ being tested is too small, then the asymptotic approximation presented above will not be very accurate, in the same way the $\chi_q^2$ distribution is not very accurately approximated by the normal if $q$ is small. Therefore, in the next section we specifically study the case where $q_n$ is bounded.

In Theorem~\ref{thm:main}, we treat the special cases of $R_0=I_{p+1}$ and $R_0=[0,I_p]$ separately, because here it is considerably much easier to deal with the non-centrality term in the decomposition of the $F$-statistic (see Section~\ref{sec:NCP1}). In particular, in this case we do not need to impose further restrictions on the distance from the null $\Delta_\gamma$ other than that it is of order $o(q_n/n)$ as $n\to\infty$ and we can also work with weaker design conditions (cf. Theorem~\ref{thm:main}\ref{thm:R=I}).

\begin{remark}[On the detection boundary of the $F$-test]\normalfont 
\label{remark:detectionBoundary}
In the classical setting, where $q$ and $p$ are fixed, while $n$ goes to infinity, it is well known that the detection boundary of the $F$-test is $n^{-1/2}$. This means that a violation of the null hypothesis $H_0: R_0\gamma = r_0$ that is of the order $\|R_0\gamma - r_0\|_2 \asymp n^{-1/2}$ leads to non-trivial asymptotic power, while a slower order yields asymptotic power equal to the size of the test and a larger order yields asymptotic power equal to one.
However, in general, when $q = q_n$ and $p=p_n$ are allowed to grow with sample size $n$, the detection boundary of the $F$-test is no longer $n^{-1/2}$ but rather $q_n^{1/4}/n^{1/2}$. To see this, first we ignore the influence of the nuisance parameters and set $\mu=0$, $\Sigma=I_p$ and $\sigma^2 = 1$, so that $S = I_{p+1}$ and $\Delta_\gamma = \|R_0\gamma-r_0\|_2^2$. From Corollary~\ref{corr:main}, we see that in order to obtain non-trivial asymptotic power, i.e., asymptotic power in the open interval $(\nu, 1)$, the non-centrality term 
$$
\sqrt{n}\Delta_\gamma b_n \;=\; \left(\frac{n^{1/2}}{q_n^{1/4}} \|R_0\gamma - r_0\|_2\right)^2 \sqrt{\frac{1}{2}\left(1-\frac{p_n+1}{n}\right)\left(1-\frac{p_n+1}{n}+\frac{q_n}{n}\right)},
$$ 
has to stay away from $0$ and $\infty$. If we exclude the pathological case $\limsup_n p_n/n= 1$, then this requirement is met if the violation of the null hypothesis is of the order $\|R_0\gamma - r_0\|_2 \asymp q_n^{1/4}/n^{1/2}$.
\end{remark}

\begin{remark}[Gaussian errors and fixed design]\normalfont
\label{remark:Gaussian}
In the classical case where the error $\eps$ follows a spherical normal distribution which is independent of the design, the $F$-statistic \eqref{eq:Fstat} follows a non-central $F$-distribution with $q_n$ and $n-p_n-1$ degrees of freedom and non-centrality parameter $\varsigma_n^2 = n\nabla_n$, conditional on $X$, where $\nabla_n = (R_0\gamma - r_0)'(R_0(U'U/n)^{-1}R_0')^{-1}(R_0\gamma-r_0)/\sigma^2$ \citep[cf.][page 41]{Rao95}. Nevertheless, even in this traditional case, only basic monotonicity results are available for the power function $\P(F_n>f_{q_n,n-p_n-1}^{(1-\nu)}|X)$ as a function of $q_n$, $n-p_n-1$ and $\varsigma_n^2$ \citep[e.g.,][]{Ghosh73, Aubel03}. In Section~\ref{sec:NCP} we investigate $\nabla_n$ as $n\to\infty$, such that $\limsup_n p_n/n <1$. Our results provide approximations for the average (or unconditional) rejection probability, i.e., for $\E[\P(F_n>f_{q_n,n-p_n-1}^{(1-\nu)}|X)]$, which are given by the Gaussian cdf applied to an elementary function in $p_n/n$, $q_n/n$ and $\Delta_\gamma= (R_0\gamma - r_0)'(R_0 S^{-1}R_0')^{-1}(R_0\gamma-r_0)/\sigma^2$ and which are therefore easy to interpret (cf. Corollary~\ref{corr:main}). 
\end{remark}

\begin{remark}[On omitted variable misspecification]\normalfont
\label{remark:Misspec}
One major motivation for us to extend the result of WC to scenarios where there is some dependence between the design and the errors, and also among the components of the standardized design vectors themselves (see Section~\ref{sec:DiscAss} for details), was to treat simple sub-models of high-dimensional linear models. These sub-models typically exhibit misspecification due to omitted regressor variables. Consider an i.i.d. sample $(y_i, z_i)_{i=1}^n$ from the high-dimensional linear model
$$
y_i \;=\; \theta'z_i\;+\;u_i,
$$ 
where the $z_i$ are random $d$-vectors with $d\gg n$ that are independent of the $u_i$, which satisfy $\E[u_i]=0$ and $\E[u_i^2]\in(0,\infty)$. Moreover, assume that the marginal distribution of the regressors $z_i$ can be represented as $z_i = \Omega^{1/2}\tilde{z}_i$, where $\Omega^{1/2}$ is symmetric and positive definite and $\tilde{z}_i$ has a Lebesgue density $f_{\tilde{z}}$ which is such that the components of $\tilde{z}_i$ are independent with zero means, unit variances and bounded $8$-th moments, so that in particular $\Cov[z_i]=\Omega$. 

Now suppose we want to use only a small number $p$ of the $d$ available regressors, with $p<n$, so that classical regression methods are feasible within this subset regression problem. These working regressors can be described as $$x_i = M'z_i,$$ where $M$ is a $d\times p$ matrix of full rank $p<d$. For instance, $M$ could be a selection matrix so that $x_i$ consists of a certain choice of $p$ components of $z_i$. In such a situation, the sample $(y_i,x_i)_{i=1}^n$ need not follow a linear homoskedastic model as in \eqref{eq:linmod}, because, in general, the conditional expectation $\E[y_i|x_i]$ is not linear in $x_i$ and the conditional variance $\Var[y_i|x_i]$ is not constant if the pair $(y_i,x_i)$ is not jointly Gaussian. However, one can always write 
$$
y_i \;=\; \tilde{\beta}'x_i \;+\;\xi_i,
$$
where $\tilde{\beta} = \argmin_b \E[(y_i-b'x_i)^2]$ and $\xi_i= y_i - \tilde{\beta}'x_i$. Here, the parameter $\tilde{\beta}$ corresponds to the best linear predictor of $y_i$ given $x_i$, and it may be of interest to test whether $H_0: \tilde{\beta} = 0$, i.e., whether the selected regressors $x_i = M'z_i$ have any value for linearly predicting the response variables $y_i$. 

Clearly, in the present setting the errors $\xi_i$ are not independent of the design vectors $x_i$ and, in general, $\E[\xi_i|x_i]\ne 0$ and $\Var[\xi_i|x_i]\ne \Var[\xi_i]$. Now consider the unobservable corrected sample $(y_i^*,x_i)_{i=1}^n$, where 
$$
y_i^* \;=\; \tilde{\beta}'x_i \;+\; \sqrt{\frac{\Var[\xi_i]}{\Var[\xi_i|x_i]}} (\xi_i-\E[\xi_i|x_i]).
$$
The corrected sample clearly follows a linear homoskedastic model as in \eqref{eq:linmod}, because $\E[y_i^*|x_i] = \tilde{\beta}'x_i$ and $\Var[y_i^*|x_i] = \Var[\xi_i]$. Moreover, the corrected sample satisfies Assumption~\ref{a.design} in view of Lemma~\ref{lemma:product} applied with $\Gamma = M'\Omega^{1/2}$, $\mu=0$ and $m_n=d$, and because the design matrix $X=[x_1,\dots, x_n]'$ has a Lebesgue density on $\R^{n\times p}$. Of course, in general, the actual sample $(y_i,x_i)_{i=1}^n$ may be very different from the corrected sample $(y_i^*,x_i)_{i=1}^n$, but the results of \citet{Steinb14} suggest that if $d\gg p$, then $\E[\xi_i|x_i]\approx 0$ and $\Var[\xi_i|x_i]\approx\Var[\xi_i]$, at least for a large collection of selection matrices $M$. Therefore, the observed sample and the corrected sample should be very similar if $d$ is large relative to $p$, and one might expect that also for a sufficiently regular statistic $T$, we have $T( (y_i,x_i)_{i=1}^n) \approx T( (y_i^*,x_i)_{i=1}^n)$. We suspect that the results of \citet{Steinb14} can also be used to establish the validity of Assumption~\ref{a.error} for the corrected sample, with $e_{i,n} := \sqrt{\Var[\xi_i]/\Var[\xi_i|x_i]}$ and $\tilde{\eps}_{i,n} := \xi_i-\E[\xi_i|x_i]$. Thus, we expect the $F$-test for $H_0:\tilde{\beta}=0$ based on the sample $(y_i,x_i)_{i=1}^n$ to be asymptotically valid and even possess similar power as the $F$-test in a correctly specified model, for most choices of $M$, provided that $d=d_n$ and $p=p_n$ tend to infinity along with $n$, such that $d_n\gg p_n$. The details of this line of reasoning will be further developed elsewhere.
\end{remark}

\subsection{Technical conditions}
\label{sec:Assumptions}

Throughout this paper, the reader will encounter several different norms. For vectors $v\in\R^k$ we write $\|v\| = (\sum_{i=1}^k v_i^2)^{1/2}$ for the usual Euclidean norm, whereas for matrices $M\in\R^{k\times \ell}$ we distinguish between the spectral norm $\|M\|_S = (\lmax{M'M})^{1/2}$ and the Frobenius norm $\|M\|_F = (\trace{M'M})^{1/2}$. We write $P_M$ for the matrix of orthogonal projection onto the column span of $M$. If $M$ satisfies $\rank{M} = \ell \le k$, then $P_M = M(M'M)^{-1}M'$. We also make use of the stochastic Landau notation. For a sequence of real random variables $z_n$, we say that $z_n=O_\P(1)$ if the sequence is bounded in probability, i.e., if $\sup_{n\in\N}\P(|z_n| > \delta) \to 0$ as $\delta\to\infty$, and we say that $z_n=o_\P(1)$ if $z_n\to0$ in probability. For a non-stochastic real sequence $a_n\ne 0$ we write $z_n = O_\P(a_n)$ if $z_n/a_n = O_\P(1)$ and $z_n = o_\P(a_n)$ if $z_n/a_n = o_\P(1)$.

The following is a list of technical assumptions needed in the proof of Theorem~\ref{thm:main}.
\begin{enumerate}
        \setlength\leftmargin{-20pt}
\renewcommand{\theenumi}{(A\arabic{enumi})}
\renewcommand{\labelenumi}{\textbf{\theenumi}} 

\item \label{a.design} 
	\begin{enumerate}
	\renewcommand{\theenumii}{\alph{enumii}}
	\renewcommand\labelenumii{(\theenumii)}
	\makeatletter \renewcommand\p@enumii{} \makeatother

	\item \label{a.factor}
		The design vectors $x_{i,n}$ are linearly generated as follows: $$ x_{i,n} = \mu_n + \Gamma_n z_{i,n},$$ where $\Gamma_n$ is a $p_n \times m_n$ matrix with $m_n\ge p_n$, such that $\Gamma_n\Gamma_n' = \Sigma_n$. The random $m_n$-vectors $z_{1,n}, \dots, z_{n,n}$ are i.i.d. and satisfy $\E[z_{1,n}]=0$, $\E[z_{1,n}z_{1,n}'] = I_{m_n}$.
		
	\item \label{a.invertibility} 
		The $(n-1)\times(p_n+1)$ matrix $U_{n,-1} = [0,I_{n-1}] U_n$ has rank $p_n+1$ with probability one. (In particular, $\P(\det{U_n'U_n}=0)=0$.)
	
	\item \label{a.Srivastava}
		For every $n\in\N$, the random $m_n$-vector $z_{1,n}$ from Assumption~\ref{a.design}.(\ref{a.factor}) also has the following property. There exist universal positive constants $c$ and $C$, not depending on $n$, such that for every orthogonal projection $P$ in $\R^{m_n}$ and for every $t>C\rank{P}$, we have $\P(\|Pz_{1,n}\|^2 > t) \le Ct^{-1-c}$.
	
	\item \label{a.moments1}
		Let $z_{1,n}$ be the random $m_n$-vector from Assumption~\ref{a.design}.(\ref{a.factor}). For $\mathcal L_{k,n} := \sup_{\|v\|=1} (\E|v'z_{1,n}|^k)^{1/k}$ we have $\mathcal L_{4,n} =O(1)$ as $n\to\infty$. For every symmetric matrix $M_n\in\R^{m_n\times m_n}$ we have $\Var[z_{1,n}'M_nz_{1,n}] = O(\trace{M_n^2}) + (\trace{M_n})^2o(1)$, as $n\to\infty$.
		
	\item \label{a.moments2}	
		In addition to \ref{a.design}.(\ref{a.moments1}), we also have $\mathcal L_{8,n} =O(1)$ as $n\to\infty$ and for any projection matrix $P_n$ in $\R^{m_n}$, $(\E[(z_{1,n}'P_nz_{1,n})^4])^{1/4} = O(\|P_n\|_F^2)$ as $n\to\infty$.

	\end{enumerate}
	
\item \label{a.error}
The error terms $\eps_{i,n}$ can be written as $\eps_{i,n} = e_{i,n} \tilde{\eps}_{i,n}$, where $e_{i,n}$ is $x_{i,n}$-measurable and such that $\max_{i=1,\dots,n}e_{i,n} = O_\P(1)$, and the $\tilde{\eps}_{i,n}$ have the following properties. There exists a universal constant $\kappa>0$, not depending on $n$, such that $\E[(\E[|\tilde{\eps}_{1,n}/\sigma_n|^4|x_{1,n}])^{1+\kappa}]=O(1)$ as $n\to\infty$, and $\max_i \E[(\tilde{\eps}_{i,n}/\sigma_n)^4|x_{i,n}] = o_\P(\sqrt{q_n})$.

\end{enumerate}

This set of assumptions is weaker than the analogous conditions imposed by WC to treat the case $(p_n-q_n)/n\to0$. A detailed discussion and comparison of the differences between our treatment and the one of WC is deferred to Section~\ref{sec:DiscAss}.


\section{Testing only a small number of hypotheses}
\label{sec:smallq}

In Theorem~\ref{thm:main} above, we needed the assumption that $q_n\to\infty$ as $n\to\infty$ in order to achieve asymptotic normality of the $F$-statistic. Since the asymptotic distribution of the $F$-statistic is well understood in the Gaussian error and fixed design case (cf. Remark~\ref{remark:Gaussian} and Lemma~\ref{lemma:NCFdist}), we expect the asymptotic distribution of the $F$-statistic to be $\chi_q^2$ rather than Gaussian if $q_n = q$ is fixed. The following result establishes this asymptotic $\chi^2$ distribution for the $F$-statistic when the error distribution of the model \ref{eq:linmod} is arbitrary up to bounded fourth moments. For this result we need somewhat stronger assumptions than those of Theorem~\ref{thm:main}. The proof is deferred to Section~\ref{sec:proof}.

\begin{theorem}
\label{thm:qfixed}
Consider the linear, homoskedastic model \eqref{eq:linmod} and assume that the design vectors $x_{1,n},\dots, x_{n,n}$ are i.i.d. Gaussian with mean $\mu_n\in\R^{p_n}$ and positive definite covariance matrix $\Sigma_n$, and that the design $X_n = [x_1,\dots, x_n]'$ is independent of the errors $\eps_n = (\eps_{1,n},\dots, \eps_{n,n})'$ which satisfy $\E[(\eps_{1,n}/\sigma_n)^4] = O(1)$. Moreover, suppose that $\limsup_n p_n/n<1$, $\Delta_\gamma = o(q_n/n)$ and that the null hypothesis does not involve a restriction on the intercept parameter (i.e., the first column of $R_0$ is equal to zero). If $q_n = q\in\N$ does not depend on $n$, then
\begin{align}
s_n^{-1/2}(F_n-1) -  \sqrt{n}\Delta_\gamma b_n \quad \xrightarrow[n\to\infty]{w} \quad (2q)^{-1/2}\chi_q^2 - \sqrt{q/2},
\end{align}
where $s_n$ and $b_n$ are as in Theorem~\ref{thm:main}.
\end{theorem}

Together with the asymptotic normality of the $F$-statistic in the case $q_n\to\infty$ (cf. Theorem~\ref{thm:main}), we can establish a non-central $F$ approximation for the $F$-statistic for any number of tested hypotheses $q_n$.

\begin{corollary}\label{corr:NCFapprox}
Consider the linear, homoskedastic model \eqref{eq:linmod} and assume that the design vectors $x_{1,n},\dots, x_{n,n}$ are i.i.d. Gaussian with mean $\mu_n\in\R^{p_n}$ and positive definite covariance matrix $\Sigma_n$, and that the design $X_n = [x_1,\dots, x_n]'$ is independent of the errors $\eps_n= (\eps_{1,n},\dots, \eps_{n,n})'$ which satisfy $\E[(\eps_{1,n}/\sigma_n)^4] = O(1)$. Moreover, suppose that $\limsup_n p_n/n<1$, $\Delta_\gamma = o(q_n/n)$ and that the null hypothesis does not involve a restriction on the intercept parameter (i.e., the first column of $R_0$ is equal to zero). Then, 
$$
\sup_{t\in\R} \left| \P(F_n\le t) - \P(F_{q_n,n-p_n-1}(\lambda_n)\le t)\right| \quad\xrightarrow[n\to\infty]{}\quad 0,
$$
where $F_{q_n,n-p_n-1}(\lambda_n)$ denotes a random variable following the non-central $F$ distribution with $q_n$ and $n-p_n-1$ degrees of freedom and non-centrality parameter $\lambda_n = \Delta_\gamma (n-p_n-1+q_n)$.
\end{corollary}

\begin{proof}
Suppose the claim does not hold. Then there exists a subsequence $n'$, such that the supremum converges to a positive number along $n'$. Then, by compactness of the closed unit interval and the extended real line, we can extract a further subsequence $n''$, such that $p_{n''}/n''\to \rho_1\in[0,1)$, $q_{n''}/n''\to \rho_2\in[0,\rho_1]$, and $q_{n''}\to q\in[1,\infty]$, as $n''\to\infty$. If $q=\infty$, then we are in the setting of Theorem~\ref{thm:main}\ref{thm:B} and we obtain asymptotic normality of $s_n^{-1/2}(F_n-1) -  \sqrt{n}\Delta_\gamma b_n$. Since the limiting distribution function is continuous, we get uniform convergence of the corresponding distribution functions, in view of Polya's theorem. Since by Lemma~\ref{lemma:NCFdist}, $s_n^{-1/2}(F_{q_n,n-p_n-1}(\lambda_n)-1) -  \sqrt{n}\Delta_\gamma b_n$ is also asymptotically standard normally distributed, we get a contradiction in that case.  If $q<\infty$, then $q_{n''}=q$ for all large $n''$, because $q_{n''}$ is integer valued. Thus Theorem~\ref{thm:qfixed} applies and shows that $s_n^{-1/2}(F_n-1) -  \sqrt{n}\Delta_\gamma b_n$ converges weakly to $(2q)^{-1/2}\chi_q^2 - \sqrt{q/2}$. Since the limiting distribution function is again continuous, and since Lemma~\ref{lemma:NCFdist} shows that $s_n^{-1/2}(F_{q_n,n-p_n-1}(\lambda_n)-1) -  \sqrt{n}\Delta_\gamma b_n$ has the same asymptotic distribution, we also get a contradiction in this case, upon using the same argument as before, involving Polya's theorem.
\end{proof}

Corollary~\ref{corr:NCFapprox} provides a unified treatment of the asymptotic behavior of the $F$-statistic without distinguishing between the cases $q_n = q$ and $q_n\to\infty$. However, this does not immediately lead to a neat formula for the local asymptotic power function of the $F$-test because of the complicated nature of the cdf of the non-central $F$-distribution. 

\begin{remark}[Fixed vs. random design]\normalfont
\label{rem:FixedVsRandom}
It is instructive to compare the non-central $F$ approximation of Corollary~\ref{corr:NCFapprox} to the distribution of the $F$-statistic in the Gaussian error and fixed design case (cf. Remark~\ref{remark:Gaussian}). In the former case the non-centrality parameter is given by $n\Delta_\gamma (n-p_n-1+q_n)/n$, while in the latter case it is $n\nabla_n$. Now if the design $X$ is random with i.i.d. rows and $S = \E[U'U/n]$, $U=[\iota, X]$, it turns out that $\Delta_\gamma := (R_0\gamma - r_0)'(R_0S^{-1}R_0')^{-1}(R_0\gamma-r_0)/\sigma^2$ is not a good approximation for $\nabla_n :=(R_0\gamma - r_0)'(R_0(U'U/n)^{-1}R_0')^{-1}(R_0\gamma-r_0)/\sigma^2$ if $p_n/n$ is not close to $q_n/n$, even if $n$ is very large. In fact, we need a correction factor of $\E[\chi_{n-p_n-1+q_n}^2/n] = (n-p_n-1+q_n)/n$ in order to account for the additional randomness coming from the design $X$ (cf. Lemma~\ref{lemma:invWish}). This correction factor, however, is close to one if $q_n/n\approx p_n/n$. If the number $q_n$ of hypotheses to be tested is much smaller than the number of parameters $p_n$, this correction is quite significant. Thus, the limiting distribution of the $F$-statistic under random design with $\E[U'U/n]=S$ is, in general, not equal to the distribution of the $F$-statistic under Gaussian errors and fixed design $X$ satisfying $U'U/n = S$. This distinction only occurs for $q_n\ll p_n$. In particular, this issue disappears completely if $p_n/n\approx 0$.
\end{remark}

\begin{remark}[On confidence sets for $R_0\gamma$]\normalfont
Corollary~\ref{corr:NCFapprox} can immediately be used to construct asymptotically valid confidence sets for $R_0\gamma$. Simply define
$$
CI_\nu \;=\; \left\{ r\in\R^q : (R_0\hat{\gamma}_n - r)'(R_0(U'U)^{-1}R_0')^{-1}(R_0\hat{\gamma}_n-r) \le f_{q,n-p-1}^{(1-\nu)} \hat{\sigma}_n^2q\right\},
$$
and note that $$\P(R_0\gamma \notin CI_\nu) = \P(F_n>f_{q,n-p-1}^{(1-\nu)}) \to \P(F_{q, n-p-1}(0)>f_{q,n-p-1}^{(1-\nu)}) = \nu,$$ where $F_n$ is the $F$-statistic under the null hypothesis $r_0 = R_0\gamma$.
\end{remark}


\section{Discussion of the technical assumptions}
\label{sec:DiscAss}

We pause for a moment to discuss the meaning of our Assumptions~\ref{a.design} and \ref{a.error}, as well as the other conditions used in Theorem~\ref{thm:main}, and we comment on the main differences to the conditions imposed in WC. 

First of all, Assumption~\ref{a.design}.(\ref{a.factor}) of linear generation of the design from possibly much higher dimensional random vectors also appears in WC who take it as a modification from \citet{Bai96}. We point out that this is a straight forward generalization of the case $m=p$, where moment restrictions have to be imposed directly on the design vectors $x_i$ (note that the components of $x_1$ may not be independent, even after standardization, whereas $x_1$ can still be linearly generated from a vector $z_1$ whose components are independent). Moreover, this assumption also allows for the interpretation that there is actually a much higher dimensional set of explanatory variables $z_i$ available whose dimensionality $m$ (possibly $m\gg n$) has already been reduced to $p<n$. See also Remark~\ref{remark:Misspec}.

Together with \ref{a.design}.(\ref{a.factor}), our Conditions~\ref{a.design}.(\ref{a.moments1},\ref{a.moments2}) replace and considerably relax Assumption~\ref{c.design} in WC, which reads as follows.
\begin{enumerate}
        \setlength\leftmargin{-20pt}
\renewcommand{\theenumi}{(C\arabic{enumi})}
\renewcommand{\labelenumi}{\textbf{\theenumi}} 

\item \label{c.design} $x_i$ is linearly generated by a m-variate random vector $z_i = (z_{i1}, . . . , z_{im})'$ so that $x_i = \Gamma z_i + \mu$, where $\Gamma$ is a $p \times m$ matrix for some $m \ge p$ such that $\Gamma\Gamma' = \Sigma$, each $z_{il}$ has finite $8$-th moment, $\E[z_i] = 0$, $\Var[z_i] = I_m$, $E[z_{ik}^4] = 3 + \Delta$ and for any $\sum_{j=1}^d \ell_j \le 8$, $\E[z_{1i_1}^{\ell_1}z_{1i_2}^{\ell_2}\cdots z_{1i_d}^{\ell_d}] = \E[z_{1i_1}^{\ell_1}]\E[z_{1i_2}^{\ell_2}]\cdots\E[z_{1i_d}^{\ell_d}]$, where $\Delta$ is some finite constant.\footnote{Note that this formulation, as it stands, is self-contradictory. Clearly, one has to assume that the indices $i_1,\dots, i_d$ are distinct, or otherwise $(C1)$ implies $1=\E[z_{11}^2] = \E[z_{11}^1 z_{11}^1] = \E[z_{11}]\E[z_{11}] = 0$.}
\end{enumerate}
In fact, in addition to \ref{c.design}, WC also need the $8$-th moments of $z_{ik}$ to be uniformly bounded so that none of them goes off to infinity as $n$ (and $m = m_n$) increases. The factorization requirement of the $8$-th mixed moments in \ref{c.design} is a straight forward relaxation of an independence assumption. However, just like the independence assumption, it rules out many spherical distributions (cf. Lemma~\ref{lemma:spherical}\ref{lemma:sph-A}). Therefore, moment conditions like \ref{a.design}.(\ref{a.moments1},\ref{a.moments2}) are much more natural to accommodate both product and spherical distributions. In fact, Condition~\ref{c.design}, together with uniform boundedness of $\E[z_{ik}^8]$, is strictly stronger than our Assumptions~\ref{a.design}.(\ref{a.factor},\ref{a.moments1},\ref{a.moments2}) (cf. Lemma~\ref{lemma:spherical}\ref{lemma:sph-C} and Lemma~\ref{lemma:product}\ref{lemma:product-B}).

Our Assumption~\ref{a.design}.(\ref{a.invertibility}) is important to guarantee that the $F$-statistic is equal to the expression on the right-hand-side of \eqref{eq:Fstat}, at least with asymptotic probability one, which is used in WC implicitly. The reason that we not only require almost sure invertibility of $U'U$ but also of the design matrix based on $n-1$ observations is only of a technical nature and plays an important role in the proof of Lemma~\ref{lemma:ProjDiag} (cf. the end of Subsection~\ref{sec:normality}), which is based on leave-one-out ideas. This lemma replaces the strong assumption of WC that there exists a global constant $c_1>0$, such that the smallest eigenvalue of the sample covariance matrix satisfies $\lmin{\tilde{X}'\tilde{X}/n} \ge c_1$, almost surely, where $\tilde{X} = (X-\iota\mu')\Sigma^{-1/2}$ is the design matrix based on the standardized regressors (cf. page 147 in WC).\footnote{Notice that this assumption rules out, for example, Gaussian design \citep[see, e.g.,][Theorem 2.1]{Edelman91}.}

Finally, Assumption~\ref{a.design}.(\ref{a.Srivastava}) is taken directly from \citet{Sriva13} to control the extreme eigenvalues of large sample covariance matrices, and different sets of sufficient conditions for \ref{a.design}.(\ref{a.Srivastava}) can be found in that reference. In WC, control of extreme eigenvalues is accomplished by use of the celebrated Bai-Yin Theorem of \citet{Bai93} (cf. Lemma~2 in WC). However, this comes at the price of the implicit assumption that the standardized design vector $\Sigma^{-1/2}(x_1-\mu)$ has independent components.\footnote{This is particularly inconvenient if one is interested in the case where the random vectors $z_1,\dots, z_n$ in (C1) (and \ref{a.design}.(\ref{a.factor})) have independent components. If both $z_1$ and $\Sigma^{-1/2}(x_1-\mu) = (\Gamma\Gamma')^{-1/2}\Gamma z_1$ have independent components and at least two rows of $(\Gamma\Gamma')^{-1/2}\Gamma$ have only non-zero entries (this can be relaxed even further), then, by the Darmois-Skitovich Theorem \citep[cf.][Theorem 5.3.1.]{Bryc95}, $z_1$ must already be Gaussian.}

Altogether, our design condition \ref{a.design} includes linear functions of both, product distributions with uniformly bounded $8$-th marginal moments and a large class of spherically symmetric distributions (cf. Lemma~\ref{lemma:spherical} and Lemma~\ref{lemma:product} in Appendix~\ref{sec:AppDiscAss}).

\medskip

The Assumption~\ref{a.error} on the error distribution extends the fourth moment condition (C2) in WC, which simply states that $\E[(\eps_{1,n}/\sigma_n)^4] = O(1)$, as $n\to\infty$.\footnote{In WC it is implicitly assumed that $\liminf_n\sigma_n^2>0$.} If the errors are independent of the design and $q_n\to\infty$ (as is the case in WC), then (C2) and \ref{a.error} are equivalent. However, Condition~\ref{a.error} is suited to also allow for some amount of dependence between the errors and the design. This dependence is ruled out in WC, because they use results of asymptotic normality of quadratic forms from \citet{Bhansali07} that apply only in the independence case (see Lemma~\ref{lemma:Bhansali} and the discussion at the beginning of Subsection~\ref{sec:normality}). We note that an $(8+\kappa)$-th moment condition like, e.g., $\sup_n\E[(\tilde{\eps}_1/\sigma_n)^{8+\kappa}]\le K$, together with $\max_i e_i = O_\P(1)$, is sufficient for \ref{a.error}, provided that $\liminf_n q_n/n>0$.

\medskip

The additional requirement of Theorem~\ref{thm:main}, that $\limsup_n p_n/n<1$ and $q_n\to\infty$, simply describes the regime of the relative number of parameters and hypotheses we are interested in. The corresponding assumption (C3) in WC and also parts~\ref{thm:R=I} and \ref{thm:A} of Theorem~\ref{thm:main} additionally require that $(p_n-q_n)/n\to 0$. This is a more serious restriction which is convenient in the present strategy of proof to show that the non-centrality term in the $F$-statistic under the local alternative degenerates to the correct value (cf. Section~\ref{sec:NCP2}). This, however, means that asymptotically we are only dealing with hypotheses where almost all of the $p$ parameters are restricted, since $q_n/p_n\to 1$ in this regime. It is therefore important to extend the analysis of the rejection probability of the $F$-test also to the regime where $q_n\ll p_n$ in order to asses the different contributions of the overall dimensionality and the multiplicity of hypothesis testing to the asymptotic rejection probability. This is what we do in Theorem~\ref{thm:main}\ref{thm:B} and in Section~\ref{sec:smallq} in the Gaussian design setting. The requirement that $q_n\to\infty$ as $n\to\infty$ is essential in Theorem~\ref{thm:main} in order to obtain a Gaussian limit. This assumption is dropped in Theorem~\ref{thm:qfixed} and Corollary~\ref{corr:NCFapprox}.

\medskip

Finally, the assumption in Theorem~\ref{thm:main} that $\Delta_\gamma = o(q_n/n)$, is rather natural in the case where $\liminf_n q_n/n>0$, where it simply reduces to $\Delta_\gamma = o(1)$. In this case, it says that we study the asymptotic rejection probability only in a shrinking neighborhood of the null hypothesis. If $\liminf_n q_n/n>0$, we also do not need to specify a rate at which $\Delta_\gamma$ approaches zero. Note, however, that part~\ref{thm:A} of Theorem~\ref{thm:main} actually does require a specific rate of contraction which is, again, only needed for technical reasons in establishing the asymptotic behavior of the non-centrality term (cf. Section~\ref{sec:NCP}). The corresponding Assumption~(C4) in WC is rather dubious and seems to arise from a miscalculation when dealing with said non-centrality term. In fact, they also need the $O(n^{-1/2})$ rate of $(R_0\gamma - r_0)' R_0 S R_0' (R_0\gamma - r_0)/\sigma_n^2$ imposed by our Theorem~\ref{thm:main}\ref{thm:A} and nothing more.\footnote{See the first display on page 146 in WC, where also the matrix $X_2X_2^{T}$ needs to be standardized. Also, there is a scaling factor of $\sqrt{n}$ missing in that argument, which is necessary to bring the non-centrality term to the same scale as the noise term.} 

In the case where $\liminf_n q_n/n=0$, we need the rate $\Delta_\gamma = o(q_n/n)$ in order to ensure that the mixed term in the expansion of the $F$-statistic is asymptotically negligible (cf. the discussion following display \eqref{eq:eta-conv}). Note that in the extreme case where $q_n=q$ is fixed, the assumption $\Delta_\gamma = o(q_n/n)$ appears to be somewhat restrictive because it rules out $n\Delta_\gamma \asymp 1$, as required for non-trivial local asymptotic power (cf. Remark~\ref{remark:detectionBoundary}). In the classical case where $q_n=q$ (and $p_n=p$) is fixed, the classical approach based on the asymptotic normality of the OLS estimator $\hat{\gamma}_n$ allows for a much wider range of local alternatives than our present strategy. Of course, the asymptotic normality of the whole vector $\hat{\gamma}_n$ breaks down if the dimensions $p_n$ and $q_n$ are too large relative to sample size $n$ (see, e.g., \citet{Portnoy86, Portnoy88}), which is why we here use a different strategy involving the assumption $\Delta_\gamma = o(q_n/n)$. Judging from the simulations of Section~\ref{sec:sim} below, it seems as if some bound on $\Delta_\gamma$ that is proportional to $q_n$ is essential for the validity of the normal approximation to the power function, because this approximation turns out to be accurate for a larger range of $\Delta_\gamma$ if $q$ gets larger.


\section{Numerical results}
\label{sec:sim}

In order to better understand the quality of the theoretical approximations to the power function of the $F$-test derived above, we conducted an extensive simulation study. We roughly follow the simulation setup of WC but our focus is more on the role of the number of tested hypotheses $q$. The linear model we considered was
$$
y_i \;=\; \beta'x_i \;+\; \eps_i,
$$
where the $\eps_i$ were i.i.d. with mean zero and variance one, independent of the design. We tried different error distributions but found little effect on the power function so we report only the results for $\mathcal N(0,1)$ errors and for $t(5)/\sqrt{5/3}$ errors.\footnote{Of course, one can always completely change the picture by, e.g., taking an error distribution that does not have a finite variance. In that case none of our theoretical approximations is of much use.} The $p$-dimensional design vectors $x_i$ were generated as i.i.d. realizations of a moving average process
$$
x_{ij} \;=\; \sum_{t=1}^T \alpha_t z_{i(j+t-1)},
$$
where $z_i = (z_{i1},\dots, z_{i(p+T-1)})'$ was either a $(p+T-1)$-dimensional standard normal vector or generated with i.i.d. $\Gamma(1,1) -1$ entries\footnote{The $\Gamma(1,1)$-distribution is just the standard exponential distribution, but we keep the $\Gamma$ notation in the plots to facilitate comparability with WC.} and $T=10$. The coefficients $\alpha_t$ were generated as i.i.d. uniform from $(0,1)$ only once and then held fixed across all the simulations to follow.

We tested only null hypotheses of the form $$H_0 : \beta_1 = \dots = \beta_q = 0,$$ at level $\nu=0.05$, so that the distribution of the $F$-statistic does neither depend on the mean of the design vectors nor on the intercept parameter, and hence it is no loss of generality that we have omitted both. The signal $\beta$ was generated such that half of the tested coefficients where equal to one and all the other coefficients were equal to zero. The signal was then scaled appropriately to produce a range of alternatives at which the power function was evaluated numerically.

Since WC have already extensively studied the effect of the dimensionality $p$ on the power of the $F$-test, we here focus more on the effect of the number of tested restrictions $q$. For our first set of Montecarlo experiments we fixed $n=100$ and $p=60$ and looked at the cases $q=4$ and $q=50$. Figure~\ref{fig:Power} shows the simulated power of the $F$-test (solid lines) as a function of the scaled distance from the null hypothesis $\Delta_\gamma$, where for each value of $\Delta_\gamma$, $10.000$ Montecarlo samples were generated. 
\begin{figure}
\includegraphics[width=\textwidth]{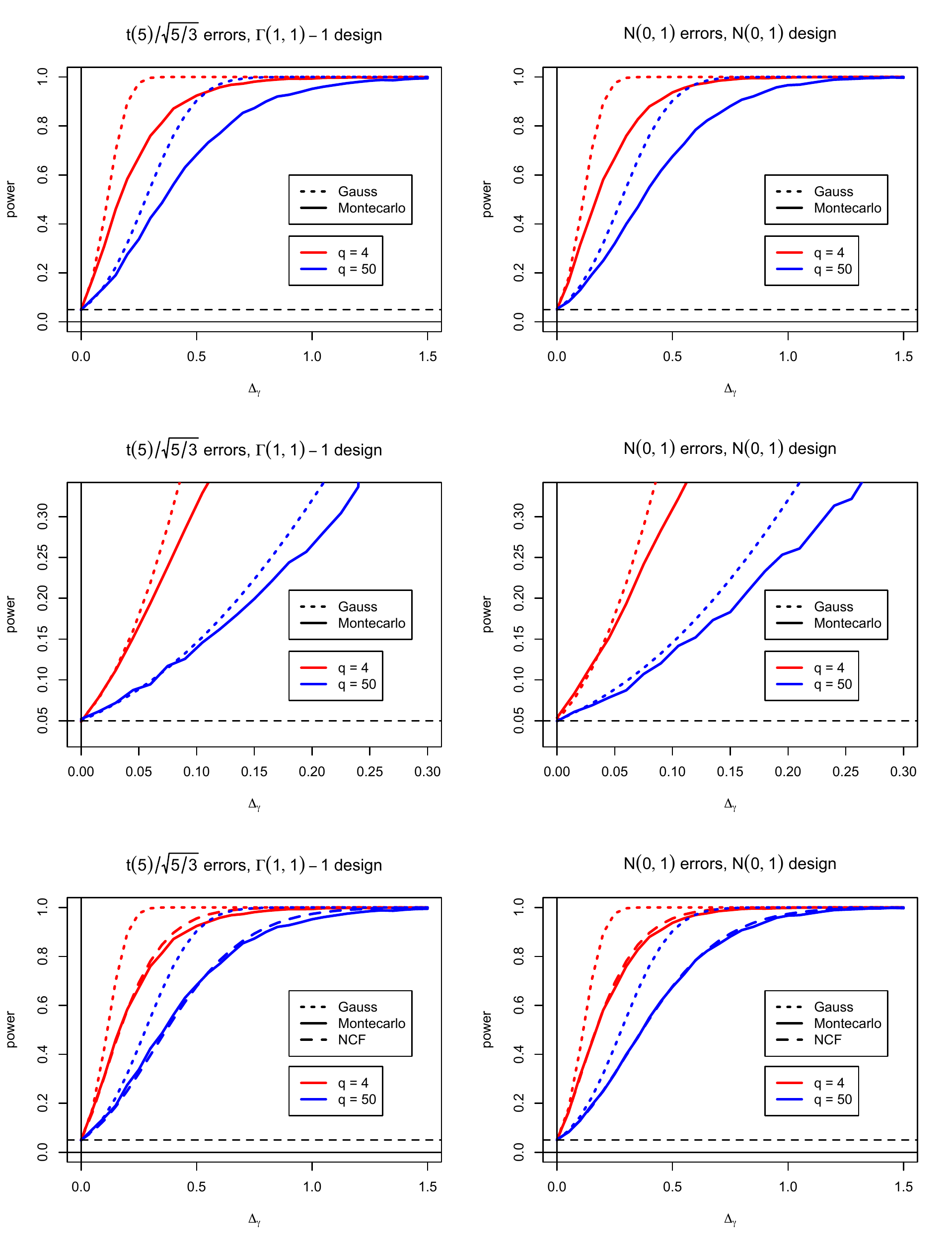} 
\caption{Simulated power functions of the $F$-test (Montecarlo) at level $\nu=0.05$, compared to the Gaussian approximation of Corollary~\ref{corr:main} (Gauss) and the non-central $F$ approximation of Corollary~\ref{corr:NCFapprox} (NCF).}
\label{fig:Power}
\end{figure}
The left column shows the results for $t(5)/\sqrt{5/3}$ distributed errors and a design that was generated from a moving average process with $\Gamma(1,1)-1$ innovations. The right column was generated with i.i.d. standard Gaussian errors and a Gaussian moving average design. As a first observation we note that the influence of non-Gaussianity on the simulated power (solid lines) is almost negligible at the present sample size of $n=100$. Now Figure~\ref{fig:Power} should be inspected from top to bottom. In the first row we clearly see the gain of power as the number of hypotheses $q$ decreases from $q=50$ to $q=4$. We also compare the simulated power to the Gaussian approximation from our asymptotic result of Theorem~\ref{thm:main} (dotted lines). The picture is qualitatively the same as in WC, who considered $q=p-2$ and who concluded ``that there is a good conformity between the empirical power and the theoretical power of the [\dots] $F$-test [\dots]'' \citep[][p. 142]{Wang13}. It seems hard to evaluate the quality of approximation directly from this picture in absolute terms. Moreover, the Gaussian approximation does not seem to become much more accurate when $q$ increases, contrary to what was suggested by Theorem~\ref{thm:main}. In fact, however, Theorem~\ref{thm:main} says that we should look at the Gaussian approximation only locally, for values of $\Delta_\gamma$ that are of a smaller order than $q/n$. Indeed, if we look at the power function in a smaller neighborhood around the null (cf. the second row of Figure~\ref{fig:Power}) we see a much better agreement between the simulated true power and the Gaussian approximation. We also see that when $q/n$ is larger, the approximation is accurate on a larger interval around the null, as predicted by the theory. However, a global Gaussian approximation to the power of the $F$-test seems to be too much to ask for. Finally, the last row of Figure~\ref{fig:Power} is identical to the first row, except that we have added the theoretical approximation based on the non-central $F$-distribution as in Corollary~\ref{corr:NCFapprox} (NCF). Compared to the Gaussian approximation, the non-central $F$ approximation is remarkably accurate over the entire range of alternatives considered, not just locally, and for both large and small values of $q$. 

To investigate also the quality of our approximations in small samples, we have repeated the simulations above with $n=30$ and $p=20$. We present only one instance of this second round of simulations in Figure~\ref{fig:Power2} to discuss the main differences to the case where $n=100$. We still find that the non-central $F$ approximation is much better than the Gaussian approximation, but clearly also the quality of the former deteriorates considerably compared to the case $n=100$. However, the local behavior of the power function is still picked up quite well even in the small sample scenario. Moreover, it is remarkable how similar the picture with $t(5)$-errors is to the picture with normal errors already for $n=30$. This suggests that the dominant reason for the imperfect approximation by the non-central $F$-distribution is the randomness of the design rather than the non-Gaussianity of the errors (cf. Remark~\ref{rem:FixedVsRandom}).
\begin{figure}
\includegraphics[width=\textwidth]{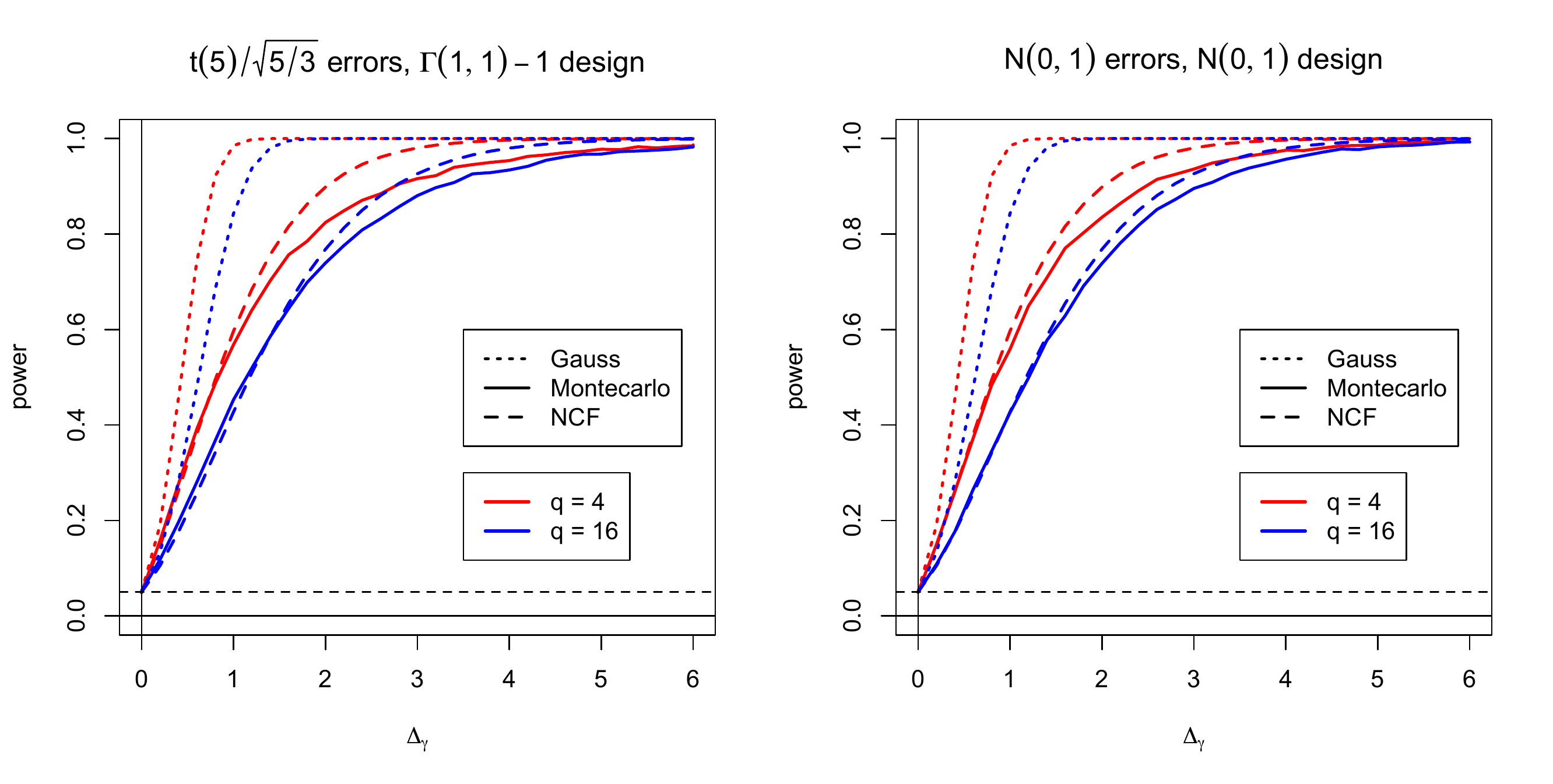} 
\caption{Simulated power functions of the $F$-test (Montecarlo) at level $\nu=0.05$,Gaussian approximation (Gauss) and the non-central $F$ approximation (NCF) for sample size $n=30$ and $p=20$ regressor variables.}
\label{fig:Power2}
\end{figure}

Finally, as in WC, we also investigate the distribution of the $F$-statistic under the null hypothesis $H_0 : \beta_1 = \dots = \beta_q = 0$. For different choices of $n$, $p$ and $q$, we have generated $10.000$ Montecarlo samples of $s_n^{-1/2}(F_n-1)$ as before, but with vanishing true signal $\beta=0$, and for $t(5)/\sqrt{5/3}$-errors and regressors generated from $\Gamma(1,1)-1$. To investigate also the impact of a non-symmetric error distribution we have repeated all the simulations also with the distribution of the errors and the design interchanged.\footnote{We do not report any results for Gaussian errors here because in that case the $F$-statistic follows exactly a central $F_{q,n-p-1}$-distribution under the null hypothesis, irrespective of the design (cf. Remark~\ref{remark:Gaussian}).} The plots in Figure~\ref{fig:Densities} were generated by applying the R-function `density' \citep{R} to the samples of standardized $F$-statistics with default settings. As before, the dashed lines correspond to the (appropriately scaled and centered) asymptotic non-central $F$-approximation of Corollary~\ref{corr:NCFapprox} and the black dotted line is the standard normal density. The overall picture is the same as for the power function. The non-central $F$ approximation is remarkably accurate even for moderate sample sizes like $n=50$. As predicted by the theory, for large $n$ and large $q$ the null distribution is well approximated by the normal, whereas for small $q$ the normal approximation fails. We also note that, again, the approximation accuracy appears to be rather insensitive to changes of the error and design distributions.

\begin{figure}
\includegraphics[width=\textwidth]{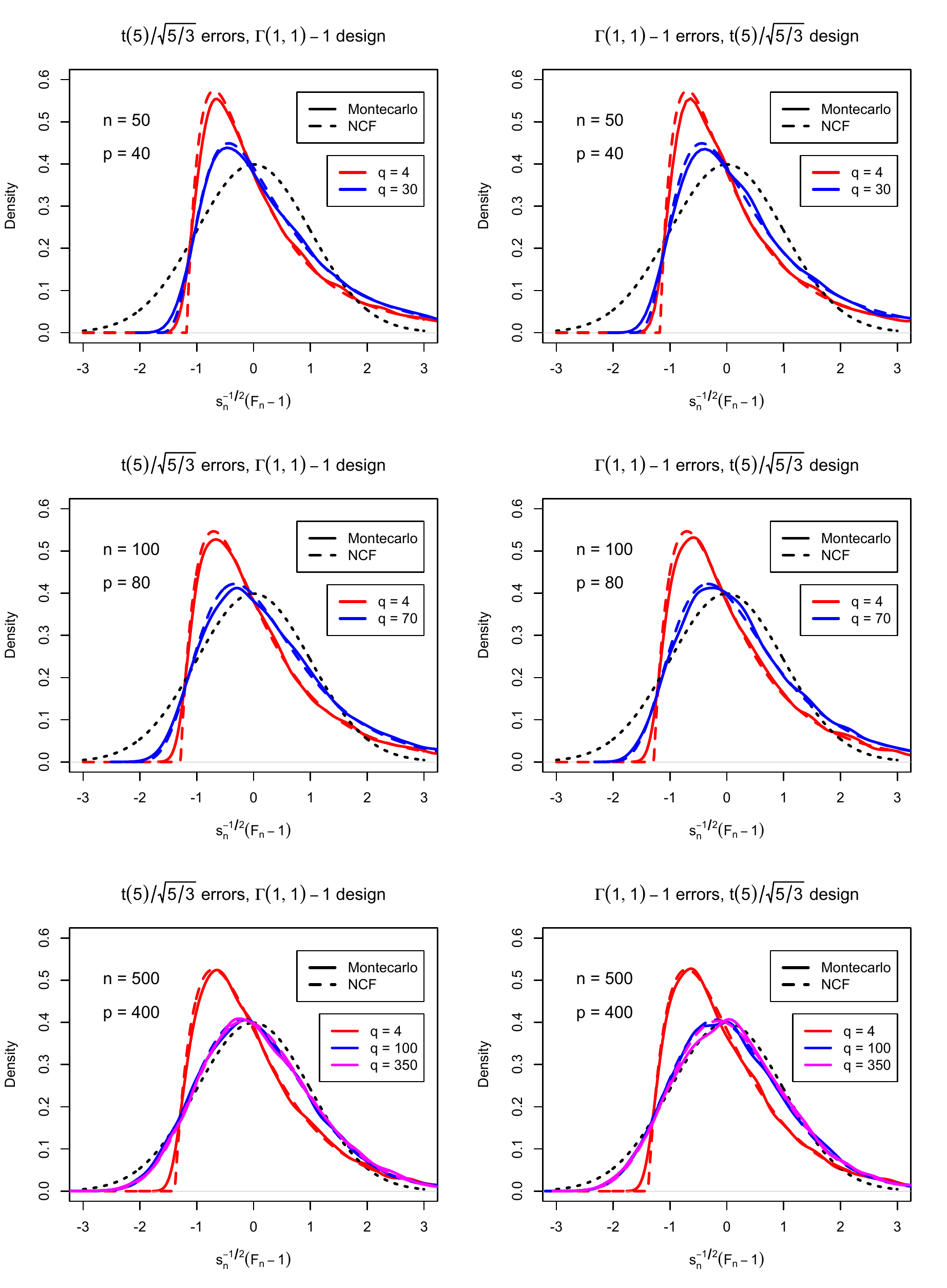} 
\caption{Kernel density estimates of simulated standardized $F$-statistics under the null hypothesis (solid lines), appropriately standardized non-central $F$-approximation of Corrolary~\ref{corr:NCFapprox} (dashed lines) and the standard normal density (black dotted lines) for different choices of $n$, $p$ and $q$, and different design and error distributions.}
\label{fig:Densities}
\end{figure}


\section{Proofs of main results}
\label{sec:proof}
 
In this section we give a high-level description of the proofs of both Theorem~\ref{thm:main} and Theorem~\ref{thm:qfixed}. The more technical parts of the argument are collected in the appendices. The following outline section pertains to both proofs and uses only assumptions that are invoked by both theorems. 
Note that because of compactness and, in particular, $\limsup_n p_n/n<1$, it is no restriction to assume that $p_n/n$ and $q_n/n$ are convergent sequences with limits $\rho_1\in[0,1)$ and $\rho_2\in[0,\rho_1]$, respectively.

 \subsection{Outline}
 \label{sec:outline}
 
In addition to the general model assumptions of Section~\ref{sec:main}, the present outline section~\ref{sec:outline} only uses the conditions $\Delta_\gamma = o(q_n/n)$, $\P(\det{U'U}=0)=0$ and $\sqrt{n}(\hat{\sigma}_n^2/\sigma_n^2-1) = O_\P(1)$. All of these are satisfied under the assumptions of any part of Theorem~\ref{thm:main} as well as under the assumptions of Theorem~\ref{thm:qfixed}, in view of Lemma~\ref{lemma:sigma}. 

The first part of Section~\ref{sec:outline} closely follows the classical approach for the decomposition of the $F$-statistic as described, e.g., in \citet[Chapter 3.7]{Rao95}. These arguments are kept to a minimum but are included nonetheless to make the notation more intelligible.

Recall the $F$-statistc $F_n$ as defined in \eqref{eq:Fstat}. For the following preliminary consideration, we work only on the event $C_n = \{\omega: \hat{\sigma}_n^2(\omega) >0, \det{U'U(\omega)} \ne 0\}$, where $\hat{\sigma}_n^2 = \|Y-U\hat{\gamma}_n\|^2/(n-p-1)$. On this event, $F_n$ is given by 
\begin{align}\label{eq:Fn1}
F_n \;=\; \frac{(R_0\hat{\gamma}_n - r_0)'(R_0(U'U)^{-1}R_0')^{-1}(R_0\hat{\gamma}_n-r_0)/q}{\sigma_n^2} \frac{\sigma_n^2}{\hat{\sigma}_n^2}.
\end{align}
Setting $\delta_\gamma = (R_0\gamma - r_0)/\sigma_n$, we have 
\begin{align*}
(R_0\hat{\gamma}_n - r_0)/\sigma_n = (R_0(U'U)^{-1}U'Y - r_0)/\sigma_n = R_0(U'U)^{-1}U'(\eps/\sigma_n) + \delta_\gamma,
\end{align*}
and thus, the first fraction in \eqref{eq:Fn1} reads
\begin{align*}
&(\eps/\sigma_n)'U(U'U)^{-1}R_0'(R_0(U'U)^{-1}R_0')^{-1}R_0(U'U)^{-1}U'(\eps/\sigma_n)/q \\
&+\; 2(\eps/\sigma_n)'U(U'U)^{-1}R_0'(R_0(U'U)^{-1}R_0')^{-1}\delta_\gamma/q \\
&+\; \delta_\gamma'(R_0(U'U)^{-1}R_0')^{-1}\delta_\gamma/q.
\end{align*}
Next, if $q<p+1$, choose a $(p+1-q)\times p$ matrix $R_1$, whose rows form an orthonormal basis for the orthogonal complement of the rows of  $R_0$. Recall that $R_0$ was chosen such that $R_0R_0' = I_q$. Hence, $T := [R_0',R_1']'$ is a $(p+1) \times (p+1)$ orthogonal matrix. Partitioning $U T' = [U_0,U_1]$, where $U_0 = U R_0'$ and $U_1 = U R_1'$, and using block matrix inversion, we see that
\begin{align*}
R_0(U'U)^{-1}R_0' &= 
[I_q,0]T(U'U)^{-1}T'[I_q,0]' \\
&= [I_q,0](TU'UT')^{-1}[I_q,0]' \\
&= (U_0'(I_n-P_{U_1})U_0)^{-1}.
\end{align*}
Similarly, we get
\begin{align*}
R_0(U'U)^{-1}U' &= 
[I_q,0]T(U'U)^{-1}T'TU' \\
&= (U_0'(I_n-P_{U_1})U_0)^{-1}U_0'(I_n-P_{U_1}).
\end{align*}
Now, by writing $W = (I_n-P_{U_1})U_0$, on $C_n$, we can simplify the $F$-statistic to read
\begin{align*}
&(\hat{\sigma}_n^2/\sigma_n^2) F_n = (\eps/\sigma_n)'P_W(\eps/\sigma_n)/q + 
2(\eps/\sigma_n)'W\delta_\gamma/q 
+ \delta_\gamma'W'W\delta_\gamma/q.
\end{align*}
The above representation remains correct also in the case where $q=p+1$ provided that the matrix $U_1$ is removed wherever it appears, i.e., $W=U_0$ in this case.
The correct centering and scaling of $F_n$ is $s_n^{-1/2}(F_n-1)$, for $s_n = 2(1/q + 1/(n-p-1))$ (cf. Lemma~\ref{lemma:NCFdist}). After noting that $P_W = P_U - P_{U_1}$ and abbreviating $M_n = (P_U - P_{U_1})/q - (I_n-P_U)/(n-p-1)$, we obtain
\begin{align}\label{eq:Fn-1a}
s_n^{-1/2}(F_n-1) \;&=\; s_n^{-1/2} (\eps/\sigma_n)'M_n(\eps/\sigma_n) \frac{\sigma_n^2}{\hat{\sigma}_n^2}\\
&\;\quad +\; s_n^{-1/2}\left( 
2(\eps/\sigma_n)'W\delta_\gamma/q
+ \delta_\gamma' W'W\delta_\gamma/q
\right) \frac{\sigma_n^2}{\hat{\sigma}_n^2},\label{eq:Fn-1b}
\end{align}
on the event $C_n$. Now, to get rid of the restriction to $C_n$, define $G_n$ by
\begin{align}\label{eq:Gn}
G_n = s_n^{-1/2} (\eps/\sigma_n)'M_n(\eps/\sigma_n) 
	+ 2s_n^{-1/2} (\eps/\sigma_n)'W\delta_\gamma/q + s_n^{-1/2} \delta_\gamma' W'W\delta_\gamma/q,
\end{align}
and note that this is well defined everywhere. It is now elementary to verify that we can study the asymptotic behavior of $G_n$ instead of $s_n^{-1/2}(F_n-1)$. Simply note that if $G_n - \eta_n^2$ converges weakly to some limiting distribution $\mathcal L$, for an appropriate centering sequence $\eta_n^2$ with $\eta_n^2=o(\sqrt{n})$, then, on $C_n$,
\begin{align}\label{eq:FnVSGn}
s_n^{-1/2}(F_n-1) - \eta_n^2 \;=\;
(G_n - \eta_n^2)(\sigma_n^2/\hat{\sigma}_n^2) +  \eta_n^2 (\sigma_n^2/\hat{\sigma}_n^2 - 1) 
\;\xrightarrow[n\to\infty]{w}\;\mathcal L,
\end{align}
and $\P(C_n)\to 1$ as $n\to\infty$, because $\P(\det{U'U}=0)=0$, $\sqrt{n}(\hat{\sigma}_n^2/\sigma_n^2 - 1) = O_\P(1)$ and $\P(\hat{\sigma}_n^2=0)\le \P(|\hat{\sigma}_n^2/\sigma_n^2 - 1| > 1/2) \to 0$, as claimed at the beginning of this section.

In what follows, we will establish that the first term on the right of the equal sign in \eqref{eq:Gn}, which we denote by $Q_n := s_n^{-1/2}\eps'M_n\eps/\sigma_n^2$, satisfies $Q_n \to \mathcal L$, weakly, for an appropriate limit distribution $\mathcal L$. The last summand in \eqref{eq:Gn} can be abbreviated to $s_n^{-1/2}n\nabla_n/q$, where $\nabla_n = \delta_\gamma'W'W\delta_\gamma/n = \delta_\gamma'(R_0(U'U/n)^{-1}R_0')^{-1}\delta_\gamma$ (as in Remark~\ref{remark:Gaussian}). It will play the role of a non-centrality term and it will be shown to be asymptotically non-random. Note that if we can also show $s_n^{-1/2}n\nabla_n/q = o_\P(\sqrt{q})$, then the mixed term in \eqref{eq:Gn} satisfies $s_n^{-1/2}(\eps/\sigma_n)'W\delta_\gamma/q = o_\P(1)$. Indeed, the conditional mean of the latter expression given $X$ is equal to zero, and its conditional variance is equal to $s_n^{-1}n\nabla_n/q^2 = (s_n^{-1/2}n\nabla_n/q)/ (\sqrt{q}\sqrt{s_n q}) = o_\P(1)$, provided that $s_n^{-1/2}n\nabla_n/q = o_\P(\sqrt{q})$.

Suppose, for now, that we have already established both, the weak convergence
\begin{equation}\label{eq:normal-conv}
Q_n \xrightarrow[n\to\infty]{w} \mathcal L,
\end{equation}
and also the fact that 
\begin{equation}\label{eq:eta-conv}
s_n^{-1/2}n\nabla_n/q - \sqrt{n}\Delta_\gamma b_n \quad \xrightarrow[n\to\infty]{i.p.}\quad 0,
\end{equation}
where $b_n$ is as in Theorem~\ref{thm:main}. Then, because $\Delta_\gamma = o(q/n)$, we have $\eta_n^2 := \sqrt{n}\Delta_\gamma b_n = o(\sqrt{q}) = o(\sqrt{n})$, as required for the argument in \eqref{eq:FnVSGn}. It also follows that $s_n^{-1/2}n\nabla_n/q = o(\sqrt{q}) + o_\P(1) = o_\P(\sqrt{q})$, so we have asymptotic negligibility of the mixed term in \eqref{eq:Gn} by the argument in the previous paragraph. Altogether, we arrive at 
$$G_n - \eta_n^2\; =\; Q_n  + o_\P(1) \;\xrightarrow[n\to\infty]{w} \;\mathcal L,$$
which establishes the conclusion of Theorem~\ref{thm:main} and Theorem~\ref{thm:qfixed}, for an appropriate choice of $\mathcal L$, provided that \eqref{eq:normal-conv} and \eqref{eq:eta-conv} hold. 

For Theorem~\ref{thm:main}, we will prove the weak convergence \eqref{eq:normal-conv} with $\mathcal L = \mathcal N(0,1)$, under the general Assumptions~\ref{a.design}.(\ref{a.factor},\ref{a.invertibility},\ref{a.Srivastava},\ref{a.moments1}) and \ref{a.error} in Section~\ref{sec:normality}, and the convergence in \eqref{eq:eta-conv} under each of the sets of assumptions of Theorem~\ref{thm:main}\ref{thm:R=I}, \ref{thm:main}\ref{thm:A} and \ref{thm:main}\ref{thm:B}, respectively (cf. Section~\ref{sec:NCP}).

For Theorem~\ref{thm:qfixed}, we will prove the  convergence \eqref{eq:normal-conv} with $\mathcal L = (2q)^{-1/2}\chi_q^2 - \sqrt{q/2}$ in Section~\ref{sec:chisquared}, and the convergence in \eqref{eq:eta-conv} is established by the same argument as in the case of Theorem~\ref{thm:main}\ref{thm:B} in Section~\ref{sec:NCP}, which does not require $q_n\to\infty$ nor Assumption~\ref{a.error}, so that it goes through also in the setting of Theorem~\ref{thm:qfixed}.


\subsection{Asymptotic normality of the noise term}
\label{sec:normality}

This section establishes the weak convergence in \eqref{eq:normal-conv} with $\mathcal L = \mathcal N(0,1)$. For this claim we only use the Assumptions~\ref{a.design}.(\ref{a.factor},\ref{a.invertibility},\ref{a.Srivastava},\ref{a.moments1}), \ref{a.error}, as well as $p_n/n\to \rho_1 \in[0,1)$, $q_n/n\to\rho_2\in[0,\rho_1]$ and $q_n\to\infty$. The following lemma is a variation of Theorem~2.1 in \citet{Bhansali07} on the asymptotic normality of quadratic forms for the case where the matrix and the enclosing vectors may exhibit a certain dependence between each other. Its proof is deferred to Appendix~\ref{sec:Appendixnormality}.

\begin{lemma}
\label{lemma:Bhansali}
Let $(\Omega, \mathcal F, \P)$ be the common probability space on which all the random quantities below are defined. For every $n\in\N$, let $\mathcal G_n\subseteq \mathcal F$ be a sub-sigma algebra, let $A_n = (a_{ij,n})_{i,j=1}^n$ be a real random symmetric $n\times n$ matrix that is $\mathcal G_n$ measurable and such that $A_n(\omega)\ne0$, $\forall \omega\in\Omega$. Let $Z_{1,n},\dots,Z_{n,n}$ be real random variables that are conditionally independent, given $\mathcal G_n$, and such that for $i\le n$, almost surely, $\E[Z_{i,n}|\mathcal G_n] = 0$, $\E[Z_{i,n}^2|\mathcal G_n] = 1$ and $\E[|Z_{i,n}|^4|\mathcal G_n] <\infty$. Moreover, assume that, as $n\to \infty$,
\begin{align*}
\frac{\|A_n\|_S^2}{\|A_n\|_F^2} \max_{j=1,\dots,n} \E[Z_{j,n}^4|\mathcal G_n]  &\quad\xrightarrow[n\to\infty]{i.p.} \quad0, \\
\frac{\max_j(A_n^2)_{jj}}{\|A_n\|_F^2} \left(\max_{j=1,\dots,n} \E[Z_{j,n}^4|\mathcal G_n]\right)^2&\quad\xrightarrow[n\to\infty]{i.p.} \quad 0, \\
 \text{and}\hspace{1cm} \sum_{j=1}^n\frac{a_{jj,n}^2\E[Z_{j,n}^4|\mathcal G_n]}{\|A_n\|_F^2} &\quad\xrightarrow[n\to\infty]{i.p.} \quad0.
\end{align*}
Then, for $Z_n = (Z_{1,n},\dots,Z_{n,n})'$, we have
\begin{align*}
\frac{Z_n'A_nZ_n - \E[Z_n'A_nZ_n|\mathcal G_n]}{\sqrt{2}\|A_n\|_F} \xrightarrow[n\to\infty]{w} \mathcal N(0,1).
\end{align*}
\end{lemma}

\begin{remark}\normalfont
The proof of Lemma~\ref{lemma:Bhansali} essentially follows the rationale of \citet{Bhansali07} with the obvious modification that all the moments of $Z_{i,n}$ have to be replaced by conditional moments. Note that if the $Z_{1,n},\dots,Z_{n,n}$ are the first $n$ elements of a sequence of i.i.d. random variables and $A_n$ is non-random, as in \citet{Bhansali07}, then the assumptions of Lemma~\ref{lemma:Bhansali} reduce to those imposed by Theorem~2.1(iii) in that reference, except for the additional requirement that $\E[Z_1^4]<\infty$, as needed here. By the method of \citet{Bhansali07} we can not get rid of this additional requirement because their truncation argument does not apply in the case of dependence between $A_n$ and $Z_n$.
\end{remark}


With Lemma~\ref{lemma:Bhansali} at hand, we can proof the asymptotic normality of 
$$
Q_n \quad=\quad s_n^{-1/2}(\eps/\sigma_n)'M_n(\eps/\sigma_n).
$$ Under the linear model \eqref{eq:linmod}, the components $\eps_1,\dots,\eps_n$ of $\eps$ are conditionally independent given the design $X$, with $\E[\eps_i/\sigma_n|X] = 0$, $\E[(\eps_i/\sigma_n)^2|X] = 1$ and $\E[(\eps_i/\sigma_n)^4|X] < \infty$, almost surely, in view of Assumption~\ref{a.error}. Moreover, the random matrix $M_n = (P_U - P_{U_1})/q - (I_n-P_U)/(n-p-1)$ is $\sigma(X)$-measurable and satisfies $\trace{M_n} = 0$ and
\begin{align*}
\|M_n\|_F^2 &= \trace{M_n^2} = \trace{[(P_U - P_{U_1})/q^2 + (I_n-P_U)/(n-p-1)^2]} \\
&= 1/q + 1/(n-p-1),
\end{align*}
with probability one, in view of \ref{a.design}.(\ref{a.invertibility}). Also, $\|M_n\|_F^2\ge 1/(n-p-1)\ge 1/n$, everywhere, because $\trace{P_U}\le p+1$. With this and in view of $\E[\eps'M_n\eps|X]/\sigma_n^2 = \trace{M_n} = 0$, almost surely, we see that
\begin{align*}
Q_n = s_n^{-1/2}\frac{\eps'M_n\eps}{\sigma_n^2} = 
\frac{(\eps/\sigma_n)'M_n(\eps/\sigma_n) - \E[(\eps/\sigma_n)'M_n(\eps/\sigma_n)|X]}{\sqrt{2}\|M_n\|_F},
\end{align*}
at least on a set of probability one, and it remains to verify the convergence conditions of Lemma~\ref{lemma:Bhansali}. For the first one, note that $\|M_n\|_S^2 \le (1/q + 1/(n-p-1))^2$ and hence, almost surely,
\begin{align*}
\frac{\|M_n\|_S^2}{\|M_n\|_F^2} \max_{j=1,\dots,n} \E[(\eps_j/\sigma_n)^4|X]  
&\le
\left(1 + \frac{q}{n-p-1}\right) \frac{\max_{j=1,\dots,n} \E[(\eps_j/\sigma_n)^4|X] }{q}.
\end{align*}
But clearly $\max_j\E[(\eps_j/\sigma_n)^4|X]/q \le O_\P(1)\,\max_j\E[(\tilde{\eps}_j/\sigma_n)^4|x_j]/q = o_\P(q^{-1/2})$ under Assumption~\ref{a.error}. Therefore, the upper bound in the previous display converges to zero in probability. For the second condition, since the diagonal entries of a projection matrix are between $0$ and $1$, we see that $(M_n^2)_{jj} \le 1/q^2 + 1/(n-p-1)^2$, and thus
\begin{align*}
&\frac{\max_j(M_n^2)_{jj}}{\|M_n\|_F^2} \left(\max_{j=1,\dots,n} \E[(\eps_j/\sigma_n)^4|X]\right)^2 \\
&\quad\le
\frac{1 + q^2/(n-p-1)^2}{1 + q/(n-p-1)} (\max_i e_i)^8\frac{(\max_j \E[(\tilde{\eps}_j/\sigma_n)^4|x_j])^2}{q},
\end{align*}
which converges to zero in probability under Assumption~\ref{a.error} and $q_n\to\infty$. Establishing the validity of the last condition is slightly more involved. Since $\|M_n\|_F^2$ is of order $1/q$, we have to show that 
$$q\sum_{j=1}^n m_{jj}^2 \E[(\eps_j/\sigma_n)^4|X] \le (\max_i e_i)^4 q \sum_{j=1}^n m_{jj}^2 \E[(\tilde{\eps}_j/\sigma_n)^4|X] = o_\P(1),$$ 
where $M_n = (m_{ij})_{i,j=1}^n$. By Assumption~\ref{a.error}, $(\max_i e_i)^4=O_\P(1)$. Now, take expectation and use H\"older's inequality with $a,b>1$ such that $1/a + 1/b = 1$ to obtain 
\begin{align}
&\E\left(q\sum_{j=1}^n m_{jj}^2\E[(\tilde{\eps}_j/\sigma_n)^4|X]\right) \notag\\
&\quad\le 
\left(\E[(\E[(\tilde{\eps}_1/\sigma_n)^{4}|x_1])^b]\right)^{1/b}\; q\sum_{j=1}^n (\E[m_{jj}^{2a}])^{1/a}.\label{eq:checkCond3}
\end{align}
Now choose $b = 1+\kappa$ and invoke Assumption~\ref{a.error} to show that the $b$-th moment of the conditional expectation in \eqref{eq:checkCond3} is $O(1)$. To establish that 
\begin{align}\label{eq:lastCondition}
q\sum_{j=1}^n (\E[m_{jj}^{2a}])^{1/a} \quad\xrightarrow[n\to\infty]{i.p.}\quad 0,
\end{align}
we distinguish between the cases $\rho_2>0$ and $\rho_2=0$.

If $\rho_2>0$, we write the diagonal elements of $M_n = (P_U - P_{U_1})/q - (I_n - P_U)/(n-p-1)$ as
\begin{align*}
m_{jj} &= \frac{(P_U)_{jj} - (P_{U_1})_{jj}}{q} - \frac{1-(P_U)_{jj}}{n-p-1}\\
&= \frac{(P_U)_{jj}-\frac{p+1}{n} + \frac{p+1-q}{n} - (P_{U_1})_{jj}}{q}
+ \frac{(P_U)_{jj} - \frac{p+1}{n}}{n-p-1},
\end{align*}
and note that $(P_U)_{jj}-(p+1)/n \in[-1,1]$ and $(P_{U_1})_{jj}-(p+1-q)/n \in[-1,1]$, in order to get the bound
\begin{align*}
&\left( \E[m_{jj}^{2a}]\right)^{1/a}\\
&\quad= \frac{1}{n^2} \left( 
\E\left[
	\left\{
		\left((P_U)_{jj} - \frac{p+1}{n}\right)\left(\frac{n}{q} + \frac{n}{n-p-1}\right) \right.\right.\right.+\\
&\hspace{4cm} \left.\left.\left.		
		\left( \frac{p+1-q}{n} - (P_{U_1})_{jj}\right)\frac{n}{q}
	\right\}^{2a}
\right]
\right)^{1/a}\\
&\quad\le
\frac{1}{n^2}
\left(
2^{2a-1} \left\{
	\left( \frac{n}{q} + \frac{n}{n-p-1}\right)^{2a} \E\left[\left( (P_U)_{jj} - \frac{p+1}{n} \right)^{2a}\right] \right.\right.+\\ 
& \hspace{4cm}\left.\left.	
	\left(\frac{n}{q} \right)^{2a} \E\left[\left((P_{U_1})_{jj} - \frac{p+1-q}{n} \right)^{2a}\right]
\right\}
\right)^{1/a},
\end{align*}
where we have used the inequality $(c+d)^{2a} \le 2^{2a-1}(c^{2a}+d^{2a})$ for $c,d \in\R$.
Hence, if we can show that the $(P_U)_{jj}$, for $j=1,\dots,n$, and also the $(P_{U_1})_{jj}$, for $j=1,\dots, n$, are identically distributed, then 
\begin{align*}
&n\sum_{j=1}^n(\E[m_{jj}^{2a}])^{1/a} \quad\le\quad 2^{(2a-1)/a}O(1)\times  \\
&\quad\quad\times\left(\E\left[|(P_U)_{11}-(p+1)/n|^{2a}\right]+\E\left[|(P_{U_1})_{11}-(p+1-q)/n|^{2a}\right]\right)^{1/a},
\end{align*}
because $n/q\to 1/\rho_2$. Since $a=(1+\kappa)/\kappa$ is fixed, it then remains to show that $|(P_U)_{11}-(p+1)/n|\to0$ and $|(P_{U_1})_{11}-(p+1-q)/n|\to 0$, in probability. The desired properties of the diagonal entries of $P_U$ and $P_{U_1}$ are now established by the following lemma, which applies under the Assumptions~\ref{a.design}.(\ref{a.factor},\ref{a.invertibility},\ref{a.moments1}), and whose proof is deferred to Appendix~\ref{sec:Appendixnormality}.


\begin{lemma}\label{lemma:ProjDiag}
For every $n\in\N$, let $x_{1,n},\dots, x_{n,n}$ be i.i.d. random $p_n$-vectors that satisfy Assumptions~\ref{a.design}.(\ref{a.factor},\ref{a.invertibility}) with $\mu_n\in \R^{p_n}$ and positive definite covariance matrix $\Sigma_n$. Moreover, suppose that the random vector $z_{1,n}$ from Assumption~\ref{a.design}.(\ref{a.factor}) also satisfies $\Var[z_{1,n}'M_nz_{1,n}] = O(\trace{M_n^2}) + (\trace{M_n})^2o(1)$, as $n\to\infty$, for every symmetric $m_n\times m_n$ matrix $M_n$. Let $R_n$ be a non-random $(p_n+1)\times k_n$ matrix such that $\rank{R_n} = k_n \le p_n+1$ and define $X_n = [x_{1,n},\dots, x_{n,n}]'$ and $W_n = [\iota,X_n]R_n$, where $\iota = (1,\dots, 1)'\in\R^n$. Furthermore, let $h_{1,n},\dots, h_{n,n}$ denote the diagonal entries of the projection matrix $P_{W_n}$. Then, the $(h_{j,n})_{j=1}^n$ are exchangeable random variables and $$|h_{1,n} - k_n/n| \quad\xrightarrow[n\to\infty]{}\quad 0,$$ in probability.
\end{lemma}

\begin{remark}\normalfont
Note that for $R_n = I_{p_n+1}$, the $h_{1},\dots,h_{n}$ in Lemma~\ref{lemma:ProjDiag} are the leverage values of the regression with design matrix $U = [\iota,X]$. 
\end{remark}


Altogether, we see that in the case $\rho_2>0$, \eqref{eq:lastCondition} holds and the weak convergence
$$
Q_n = s_n^{-1/2} \frac{\eps'M_n\eps}{\sigma_n^2} \xrightarrow[n\to\infty]{w} \mathcal N(0,1),
$$
follows, as required in \eqref{eq:normal-conv}.

To treat the case $\rho_2=0$, we recall from Section~\ref{sec:outline} that $P_U - P_{U_1} = P_W = U(U'U)^{-1}R_0'(R_0(U'U)^{-1}R_0')^{-1}R_0(U'U)^{-1}U'$, almost surely, by Assumption~\ref{a.design}.(\ref{a.invertibility}), and write the diagonal elements of $M_n = P_W/q - (I_n - P_U)/(n-p-1)$ as before as
\begin{align*}
m_{jj} &= \frac{(P_W)_{jj} -\frac{q}{n}}{q}
+ \frac{(P_U)_{jj} - \frac{p+1}{n}}{n-p-1}.
\end{align*}
Note that the $(P_W)_{jj} = (1,x_j')(U'U)^{-1}R_0'(R_0(U'U)^{-1}R_0')^{-1}R_0(U'U)^{-1}(1,x_j')'$, for $j=1,\dots, n$, are exchangeable random variables, because $x_1,\dots, x_n$ are i.i.d. and $U'U = \sum_{j=1}^n (1,x_j')'(1,x_j')$ is a function in $x_1,\dots, x_n$ that is invariant under permutations of its arguments.
Therefore,
\begin{align*}
q\sum_{j=1}^n (\E[m_{jj}^{2a}])^{1/a} 
\;&\le\;
q\sum_{j=1}^n \left( \E\left[ \left( \frac{|(P_W)_{jj} - \frac{q}{n}|}{q} + \frac{1}{n-p-1}\right)^{2a}\right] \right)^{1/a}\\
&=\;
\left( \E\left[ \left( \sqrt{\frac{n}{q}}\left|(P_W)_{11} - \frac{q}{n}\right|  + \frac{\sqrt{nq}}{n-p-1}\right)^{2a}\right] \right)^{1/a}.
\end{align*}
By boundedness of the diagonal entries of a projection matrix, it remains to show that $H_n := \sqrt{n/q}((P_W)_{11}-q/n) \to 0$, in probability, as $n\to\infty$. By exchangeability, and Assumption~\ref{a.design}.(\ref{a.invertibility}), it follows that $q = \trace{P_W} = \E[\trace{P_W}] = \sum_{j=1}^n\E[(P_W)_{jj}] = n\E[(P_W)_{11}]$, almost surely, and thus $\E[H_n] = 0$. Moreover, $H_n \ge -\sqrt{q/n}$, and for $\eps>0$, $0=\E[H_n] \ge \E[H_n\mathbf{1}_{\{H_n> \eps\}}] - \sqrt{q/n}$, which implies that for $n$ large (such that $\sqrt{q/n}<\eps$), $\eps\P(|H_n|>\eps) \le \eps\P(H_n>\eps) + \eps\P(H_n< - \eps) \le \E[H_n\mathbf{1}_{\{H_n>\eps\}}] + 0 \le \sqrt{q/n}$. This establishes the asymptotic normality required in \eqref{eq:normal-conv} also in the case $\rho_2=0$.


\subsection{Asymptotic $\chi^2$-distribution of the noise term}
\label{sec:chisquared}

In this section we show that under the assumptions of Theorem~\ref{thm:qfixed}, \eqref{eq:normal-conv} holds with $\mathcal L = (2q)^{-1/2}\chi_q^2 - \sqrt{q/2}$. First note that 
$$
Q_n \;=\; s_n^{-1/2}\frac{\eps'M_n\eps}{\sigma_n^2} \;=\; 
s_n^{-1/2} \frac{\eps'P_W\eps}{\sigma_n^2q} - s_n^{-1/2}\frac{\hat{\sigma}_n^2}{\sigma_n^2}.
$$
By Lemma~\ref{lemma:sigma}, we have $ s_n^{-1/2}\hat{\sigma}_n^2/\sigma_n^2\to \sqrt{q/2}$, in probability. Since we do not test the intercept parameter $\alpha$, we can write $R_0 = [0,T_0]$, for some $q\times p$ matrix $T_0$, and 
$$
R_1=\begin{pmatrix}1&0\cdots0\\ \begin{matrix} 0\\ \vdots\\0\end{matrix} &T_1 \end{pmatrix}
$$
for some $(p-q)\times p$ matrix $T_1$, such that $\bar{T} := [T_0',T_1']'$ is $p\times p$ orthogonal. If $q=p$, then $T_0=I_p$ and $R_1=(1,0,\dots, 0)$. With this notation, we get $U = [\iota,X]$, $U_0 = UR_0' = XT_0'$ and $U_1 = UR_1' = [\iota, XT_1']$. From this we see that the distribution of $W=(I_n-P_{U_1})U_0 = (I_n-P_{[\iota, (I_n-P_\iota)XT_1']})XT_0'$ does not depend on $\mu$, and without loss of generality we may assume that $\mu=0$. Moreover, by standard properties of orthogonal projections,
\begin{align*}
P_W \;=\; P_U - P_{U_1} \;&=\; P_X + P_{(I_n-P_X)\iota} - \left(P_{XT_1'} + P_{(I_n-P_{XT_1'})\iota}\right)\\
&=\; P_{X\bar{T}'} - P_{XT_1'}  \;+\;  P_{(I_n-P_X)\iota} -P_{(I_n-P_{XT_1'})\iota}\\
&=\; P_{(I_n-P_{XT_1'})XT_0'} \;+\;  P_{(I_n-P_X)\iota} -P_{(I_n-P_{XT_1'})\iota}.
\end{align*}
Now, abbreviate $A = (I_n-P_{XT_1'})XT_0'$ and $B = P_{(I_n-P_X)\iota} -P_{(I_n-P_{XT_1'})\iota}$, and note that $P_A$ is uniformly distributed on the Grassmann manifold of $n\times n$ projection matrices of rank $q$, because for an $n\times n$ orthogonal matrix $O$ we have $OX = (OX\Sigma^{-1/2}) \Sigma^{1/2} \thicksim X$, $OP_AO' =OA(A'O'OA)^{-1}A'O'$, and
\begin{align*}
OA \;&=\; (I_n- OXT_1'(T_1X'O'OXT_1')^{-1}T_1X'O')OXT_0' \\
&\thicksim\; (I_n- XT_1'(T_1X'XT_1')^{-1}T_1X')XT_0' \;=\; A.
\end{align*}
Moreover, $\trace{B} = 0$, almost surely, and thus $\E[\eps'B\eps/\sigma_n^2|X] = 0$ and by standard calculations using independence (cf. the proof of Lemma~\ref{lemma:sigma}) and the fact that $P_X = P_{X\bar{T}'} = P_{[XT_0',XT_1']}$, 
\begin{align*}
\Var[\eps'&B\eps/\sigma_n^2|X] = 2\trace{B^2} + \sum_{i=1}^n(\E[(\eps_i/\sigma_n)^4] - 3) B_{ii}^2 \\
&\le (2+\E[(\eps_1/\sigma_n)^4]) \trace{B^2} \\
&=\; O(1) \left(2 - 2\frac{\iota'(I_n-P_X)(I_n-P_{XT_1'})\iota \iota'(I_n-P_{XT_1'})(I_n-P_X)\iota}{\iota'(I_n-P_X)\iota\iota'(I_n-P_{XT_1'})\iota}\right)\\
&=\; O(1) \left(1- \frac{\iota'(I_n-P_{X\bar{T}'})\iota}{\iota'(I_n-P_{XT_1'})\iota}\right).
\end{align*}
In view of Lemma~\ref{lemma:NormalProjMat} with $\mu=0$ and using the first two moments of the $\chi^2$ distribution, we see that the expressions $\iota'(I_n-P_{X\bar{T}'})\iota/n$ and $\iota'(I_n-P_{XT_1'})\iota/n$ both converge to $1-\rho_1$, in probability, and thus the whole expression on the last line of the previous display converges to zero in probability. Altogether, we see that 
$$
Q_n \;=\; (s_nq^2)^{-1/2} \frac{\eps'P_A\eps}{\sigma_n^2} -\sqrt{q/2} + o_\P(1).
$$
Because $P_A$ is uniformly distributed on the Grassmann manifold, it can be stochastically represented as $P_A \thicksim CC'$, where $C$ is a random $n\times q$ matrix that is uniformly distributed on the Stiefel manifold of order $n\times q$, i.e., $V$ has orthonormal columns and its distribution is both left and right invariant under the action of the appropriate orthogonal group. Since $C$ and $\eps$ are independent, the so called Diaconis-Freedman effect as described in \citet[][Theorem 2.1]{DuembCZ12} entails that the conditional distribution of $C'\eps/\sigma_n$ given $C$, converges weakly in probability to a $q$-dimensional standard normal distribution because $\|\eps/\sigma_n\|^2/n \to 1$ in probability and $\eps'\bar{\eps}/(n\sigma_n^2) \to 0$ in probability, where $\bar{\eps}$ is an independent copy of $\eps$. Weak convergence in probability implies convergence of the conditional characteristic functions of $C'\eps/\sigma_n$ given $C$, in probability, which, by boundedness, implies convergence of the unconditional characteristic functions. Consequently, we obtain the weak convergence 
$$
Q_n \;\thicksim\; (s_nq^2)^{-1/2} \|C'\eps/\sigma_n\|^2 -\sqrt{q/2} + o_\P(1) \;\xrightarrow[n\to\infty]{w} \; (2q)^{-1/2}\chi_q^2 - \sqrt{q/2}.
$$
This establishes \eqref{eq:normal-conv} with $\mathcal L$ as claimed.


\subsection{Asymptotic behavior of the non-centrality term}
\label{sec:NCP}

Finally, we have to establish the convergence in \eqref{eq:eta-conv} in the three cases of Theorem~\ref{thm:main}\ref{thm:R=I}, \ref{thm:main}\ref{thm:A} and \ref{thm:main}\ref{thm:B} as well as under the assumptions of Theorem~\ref{thm:qfixed}. We begin by a representation of $s_n^{-1/2}n\nabla_n/q$ that pertains to all of these cases. In this preliminary consideration we only use Assumpitons~\ref{a.design}.(\ref{a.factor}) and (\ref{a.invertibility}) which are assumed to hold in each of the cases under investigation. Recall the conventions and definitions of Section~\ref{sec:outline}, in particular, $U = [\iota, X]$, $T=[R_0',R_1']'$, $U_0 = UR_0'$, $U_1 = UR_1'$, $W=(I_n - P_{U_1})U_0$, $\delta_\gamma=(R_0\gamma-r_0)/\sigma_n$, $\Delta_\gamma = \delta_\gamma'(R_0S^{-1}R_0')^{-1}\delta_\gamma$, and $\nabla_n = \delta_\gamma'W'W\delta_\gamma/n = \delta_\gamma'(R_0(U'U/n)^{-1}R_0')^{-1}\delta_\gamma$. Write $\hat{\mu}_n = \sum_{i=1}^n x_i/n = X'\iota/n$ and $\hat{\Sigma}_n = X'X/n - \hat{\mu}_n\hat{\mu}_n' = X'(I_n-P_\iota)X/n$ and partition the $(p+1)\times (p+1)$ orthogonal matrix $T$ as
$$
T = \begin{pmatrix}R_0\\ R_1 \end{pmatrix} = \begin{pmatrix}t_0 &T_0\\ t_1 &T_1 \end{pmatrix},
$$
where $t_0\in\R^{q}$ and $T_0\in\R^{q\times p}$.
Since
$$
(U'U/n)^{-1} = \begin{pmatrix} 1 & \hat{\mu}_n'\\ \hat{\mu}_n & \hat{\Sigma}_n+\hat{\mu}_n\hat{\mu}_n'\end{pmatrix}^{-1} = \begin{pmatrix} 1+\hat{\mu}_n'\hat{\Sigma}_n^{-1}\hat{\mu}_n & -\hat{\mu}_n'\hat{\Sigma}_n^{-1}\\ -\hat{\Sigma}_n^{-1}\hat{\mu}_n &\hat{\Sigma}_n^{-1}\end{pmatrix},
$$
almost surely, by \ref{a.design}.(\ref{a.invertibility}), we have
\begin{align}
s_n^{-1/2}n\nabla_n/q & = \begin{cases}
\sqrt{\frac{n}{q}}(s_n q)^{-1/2} \sqrt{n}\delta_\gamma'(U_0'U_0/n)\delta_\gamma, 	&\text{if } q = p+1,\\
\sqrt{\frac{n}{q}}(s_n q)^{-1/2} \sqrt{n}\delta_\gamma'U_0'(I_n - P_{U_1})U_0\delta_\gamma/n, &\text{if } q\le p, \label{eq:Cases}\\
\sqrt{\frac{n}{q}}(s_n q)^{-1/2} \sqrt{n}\delta_\gamma'(T_0\hat{\Sigma}_n^{-1}T_0')^{-1}\delta_\gamma, 	&\text{if } t_0=0.
\end{cases}
\end{align}
Notice that the last two cases are not mutually exclusive, but the case $t_0=0$ is a sub-case of the case $q\le p$. The representation of $\nabla_n$ in the case $t_0=0$ will come in handy.
With the notation
\begin{align*}
&\Omega = \begin{bmatrix}
\Omega_{00} & \Omega_{01}\\
\Omega_{10} & \Omega_{11}
\end{bmatrix}
=
\begin{bmatrix}
R_0S R_0' & R_0S R_1'\\
R_1S R_0' & R_1S R_1'
\end{bmatrix}
= T S T', \quad\text{where}\\
&S = \begin{pmatrix} 1 &\mu'\\\mu &\Sigma+\mu\mu'\end{pmatrix} = \E\left[\begin{pmatrix} 1\\ x_1\end{pmatrix}\begin{pmatrix} 1 &x_1'\end{pmatrix}\right] = \E[U'U/n],
\end{align*}
under \ref{a.design}.(\ref{a.factor}), and by the simple block matrix inversion argument $R_0S^{-1}R_0' = [I_q,0](TST')^{-1}[I_q,0]' = (\Omega_{00} - \Omega_{01}\Omega_{11}^{-1}\Omega_{10})^{-1}$, involving the orthogonality of $T$, we analogously get
\begin{equation}\label{eq:CasesDelta}
\Delta_\gamma = \begin{cases}
 \delta_\gamma'\Omega_{00}\delta_\gamma, 				&\text{if } q= p+1,\\
 \delta_\gamma'(\Omega_{00} - \Omega_{01}\Omega_{11}^{-1}\Omega_{10})\delta_\gamma, 	&\text{if } q\le p,\\
 \delta_\gamma'(T_0\Sigma^{-1}T_0')^{-1}\delta_\gamma,		&\text{if } t_0=0.
\end{cases}
\end{equation}

\subsubsection{The case of Theorem~\ref{thm:main}\ref{thm:R=I}}
\label{sec:NCP1}

Under the assumptions of Theorem~\ref{thm:main}\ref{thm:R=I}, choose $b_n$ as in the theorem, which can be written as $b_n= \sqrt{(1-(p+1)/n)(1-(p+1)/n+q/n)/(2q/n)} = \sqrt{n/q} (s_nq)^{-1/2}(n-p-1+q)/n$. 
To establish the convergence in \eqref{eq:eta-conv} in the case of $q=p+1$, first note that now $b_n = \sqrt{n/q} (s_nq)^{-1/2}$, and consider
\begin{align}\label{eq:U0U0}
s_n^{-1/2}n\nabla_n/q - \sqrt{n}\Delta_\gamma b_n \quad=\quad b_n \sqrt{n} \delta_\gamma'\left(U_0'U_0/n - \Omega_{00}\right)\delta_\gamma,
\end{align}
which has mean zero. For the variance, we observe that $\Var[\sqrt{n} \delta_\gamma'(U_0'U_0/n)\delta_\gamma] = n^{-1}\sum_{i=1}^n \Var[\delta_\gamma'R_0(1,x_i')'(1,x_i')R_0'\delta_\gamma] \le \E[|\delta_\gamma'R_0(1,x_1')'|^4] = O(|\delta_\gamma' \Omega_{00} \delta_\gamma|^2) = O(\Delta_\gamma^2)$, in view of Lemma~\ref{lemma:moments}\ref{l:mom:lin} and Assumption~\ref{a.design}.(\ref{a.moments1}), and $b_n\Delta_\gamma = O(1) \sqrt{n/q}\Delta_\gamma = o(1)$, by assumption. This clearly covers also the case $R_0 = I_{p+1}$. If $R_0 = [0,I_p]$, then $t_0=0$, $T_0=I_p$, $q=p$ and the difference in \eqref{eq:eta-conv} reads
\begin{align*}
s_n^{-1/2}n\nabla_n/q - \sqrt{n}\Delta_\gamma b_n \quad=\quad \sqrt{\frac{n}{q}}(s_nq)^{-1/2} \sqrt{n} \delta_\gamma'\left(\hat{\Sigma}_n-\Sigma\frac{n-1}{n}\right)\delta_\gamma.
\end{align*}
The mean of this expression is, again, equal to zero. For its variance we find that $\Var[ \sqrt{n} \delta_\gamma'\hat{\Sigma}_n\delta_\gamma] = O(|\delta_\gamma'\Sigma\delta_\gamma|^2)$, in view of Lemma~\ref{lemma:moments}\ref{l:mom:cov} together with Assumption~\ref{a.design}.(\ref{a.moments1}), and $\sqrt{n/q}(s_nq)^{-1/2}\delta_\gamma'\Sigma\delta_\gamma = O(1)\sqrt{n/q}\Delta_\gamma\to 0$, by assumption. This finishes the proof of Theorem~\ref{thm:main}\ref{thm:R=I}.

\subsubsection{The case of Theorem~\ref{thm:main}\ref{thm:A}}
\label{sec:NCP2}

For part~\ref{thm:A} we only need to consider the case where $q\le p$, as the case $q = p+1$ has already been treated above (simply restrict to the subsequence $n'$ such that $q_{n'} \le p_{n'}$). We establish the convergence in \eqref{eq:eta-conv} for $b_n = \sqrt{n/q}(s_nq)^{-1/2}$ rather than $b_n$ as in the Theorem. It should be clear, however, that this is no restriction. [Indeed, if \eqref{eq:eta-conv} holds with $b_n=\sqrt{n/q}(s_nq)^{-1/2}$ and $\tilde{b}_n = b_n(n-p-1+q)/n$ is as in the theorem, then $\sqrt{n}\Delta_\gamma b_n - \sqrt{n}\Delta_\gamma\tilde{b}_n = (b_n-\tilde{b}_n)\sqrt{n}\Delta_\gamma = ((p+1)/n-q/n)b_n \sqrt{n} O(\sqrt{q}/n)$, by the additional assumption of Theorem~\ref{thm:main}\ref{thm:A}. Since in the present case $(p-q)/n\to 0$, the previous expression converges to zero as $n\to\infty$.] Now, since $q\le p$, the quantity of interest reads
\begin{align}
s_n^{-1/2}n&\nabla_n/q - \sqrt{n}\Delta_\gamma b_n \quad=\quad
b_n \sqrt{n} \delta_\gamma'\left(U_0'U_0/n - \Omega_{00} \right)\delta_\gamma\notag\\
&+ b_n\sqrt{n}\delta_\gamma'\left(\Omega_{01}\Omega_{11}^{-1}\Omega_{10} - \frac{U_0'U_1}{n}\left(\frac{U_1'U_1}{n}\right)^{-1}\frac{U_1'U_0}{n}\right)\delta_\gamma.\label{eq:partii}
\end{align}
The first term on the right-hand-side has already been studied in \eqref{eq:U0U0}, and the same argument applies, except that now $\delta_\gamma' \Omega_{00} \delta_\gamma \ne \Delta_\gamma$, in general. But $b_n\delta_\gamma' \Omega_{00} \delta_\gamma= O(1) \sqrt{n/q}\delta_\gamma' R_0S R_0' \delta_\gamma = O(n^{-1/2}) \to 0$ by the additional assumption of Theorem~\ref{thm:main}\ref{thm:A}. For the remaining term in \eqref{eq:partii}, as in WC, we begin by approximating $U_1'U_1/n$ by $\Omega_{11}$. This can only be successful because here we are dealing with a sample covariance matrix of dimension $p+1-q$, based on $n$ independent observations and we assume that $(p+1-q)/n \to 0$. We abbreviate $\tilde{U}_0 = U_0\Omega_{00}^{-1/2}$ and $\tilde{U}_1 = U_1\Omega_{11}^{-1/2}$ and consider the absolute difference
\begin{align}
&\left|
	\sqrt{n}\delta_\gamma'
	\left[ 
		\frac{U_0'U_1}{n}\left(\frac{U_1'U_1}{n}\right)^{-1}\frac{U_1'U_0}{n} - 
		\frac{U_0'U_1}{n}\Omega_{11}^{-1}\frac{U_1'U_0}{n}
	\right]\delta_\gamma
\right|\notag\\
&\quad=
\left|
	\sqrt{n}\delta_\gamma'\Omega_{00}^{1/2}
	\frac{\tilde{U}_0'\tilde{U}_1}{n}\left[ \left(\frac{\tilde{U}_1'\tilde{U}_1}{n}\right)^{-1} - I_{p+1-q}\right]\frac{\tilde{U}_1'\tilde{U}_0}{n} 
	\Omega_{00}^{1/2}\delta_\gamma
\right|\notag\\
&\quad\le
\left\| \left(\frac{\tilde{U}_1'\tilde{U}_1}{n}\right)^{-1} - I_{p+1-q}\right\|_S
\left\| \frac{\tilde{U}_1\tilde{U}_1'}{n} \right\|_S
\left\|\frac{\tilde{U}_0'\tilde{U}_0}{n}\right\|_S
\sqrt{n}\delta_\gamma'\Omega_{00}\delta_\gamma.\label{eq:Inverse}
\end{align}
Now, Lemma~\ref{lemma:moments}\ref{l:mom:consistency} with $k_n = p + 1 - q$ shows that $\|\tilde{U}_1'\tilde{U}_1/n - I_{p+1-q}\|_S \to 0$ in probability, since $k_n/n\to 0$, and it also establishes the boundedness in probability of $\| \tilde{U}_1\tilde{U}_1'/n \|_S
\|\tilde{U}_0'\tilde{U}_0/n\|_S$. The assumptions of this lemma are clearly satisfied under \ref{a.design}.(\ref{a.factor},\ref{a.Srivastava},\ref{a.moments1}). Now the convergence in spectral norm implies the convergence of the extreme eigenvalues of $\tilde{U}_1'\tilde{U}_1/n$ to $1$, and thus, also the extreme eigenvalues of the inverse converge to $1$, which means that $\| (\tilde{U}_1'\tilde{U}_1/n)^{-1} - I_{p+1-q}\|_S \to 0$, in probability. Since $b_n\sqrt{n}\delta_\gamma'\Omega_{00}\delta_\gamma = O(1)$, by the additional assumption of Theorem~\ref{thm:main}\ref{thm:A}, the upper bound in \eqref{eq:Inverse} converges to zero in probability. Thus, we have shown that we can replace $U_1'U_1/n$ in \eqref{eq:partii} by $\Omega_{11}$, without changing the limit.

To finish part~\ref{thm:A} it remains to show that $b_nB_n$ converges to zero in probability, where
\begin{equation*}
B_n := \sqrt{n}\delta_\gamma'\left[ 
\frac{U_0'U_1}{n}\Omega_{11}^{-1}\frac{U_1'U_0}{n} - \Omega_{01}\Omega_{11}^{-1}\Omega_{10}
\right]\delta_\gamma.
\end{equation*} 
To evaluate its expectation, write
\begin{align*}
B_n = \frac{1}{n^2} \sum_{i,j=1}^n\sqrt{n}\delta_\gamma'&\left[
R_0
(1, x_i')'(1, x_i')R_1'\Omega_{11}^{-1}R_1
(1, x_j')'(1, x_j')R_0'
- \Omega_{01}\Omega_{11}^{-1}\Omega_{10}
\right]\delta_\gamma.
\end{align*}
Since $\E[R_0(1, x_i')'(1, x_i')R_1'] = R_0SR_1' = \Omega_{01}$, all summands in $\E[B_n]$ with distinct indices $i\ne j$ disappear. Using parts~\ref{l:mom:lin} and \ref{l:mom:quad1} of Lemma~\ref{lemma:moments}, which apply in view of Assumptions~\ref{a.design}.(\ref{a.factor},\ref{a.moments1}), we arrive at
\begin{align*}
|\E[B_n]| &= 
\Big|\frac{1}{n} \sqrt{n}\delta_\gamma'\E[
R_0
(1, x_1')'(1, x_1')R_1'\Omega_{11}^{-1}R_1
(1, x_1')'(1, x_1')R_0']\delta_\gamma\\
&\hspace{2cm}
- \frac{1}{n} \sqrt{n}\delta_\gamma'\Omega_{01}\Omega_{11}^{-1}\Omega_{10}
\delta_\gamma\Big|\\
&\le
n^{-1/2}
\sqrt{\E[|\delta_\gamma'R_0(1,x_1')'|^4]\E[|(1, x_1')R_1'\Omega_{11}^{-1}R_1(1, x_1')'|^2]}\\
&\quad+
n^{-1/2}\delta_\gamma'\Omega_{01}\Omega_{11}^{-1}\Omega_{10}\delta_\gamma\\
&=
n^{-1/2}\sqrt{O(|\delta_\gamma'R_0SR_0'\delta_\gamma|^2)O((p+1-q)^2)} + n^{-1/2}\delta_\gamma'\Omega_{01}\Omega_{11}^{-1}\Omega_{10}\delta_\gamma\\
&=
O(\sqrt{n}\delta_\gamma'\Omega_{00}\delta_\gamma \,(p+1-q)/n) + n^{-1/2}\delta_\gamma'\Omega_{01}\Omega_{11}^{-1}\Omega_{10}\delta_\gamma.
\end{align*} 
After multiplying by $b_n$, the upper bound still converges to zero, because, first, $b_n\sqrt{n}\delta_\gamma'\Omega_{00}\delta_\gamma = O(1)$ and $(p+1-q)/n \to 0$, by the additional assumptions of Theorem~\ref{thm:main}\ref{thm:A}, and because $b_nn^{-1/2}\delta_\gamma'\Omega_{01}\Omega_{11}^{-1}\Omega_{10}\delta_\gamma \le b_nn^{-1/2}\delta_\gamma'\Omega_{00}\delta_\gamma = O(n^{-1})$, where the inequality holds in view of the second case in \eqref{eq:CasesDelta}.

In order to show that the distribution of $B_n$ also concentrates around its mean, we make use of the Efron-Stein inequality.\footnote{An explicit argument for the concentration aspect of $B_n$ is missing in WC.} We use the abbreviations $D = \delta_\gamma'\Omega_{01}\Omega_{11}^{-1}\Omega_{10}\delta_\gamma$, $L_i = \delta_\gamma'R_0(1, x_i')'$, $Q_{ij} = (1,x_i')R_1'\Omega_{11}^{-1}R_1(1, x_j')'$ and define the functions $g:\R^{p\times n}\to \R$ and $g_k:\R^{p\times (n-1)}\to \R$, for $k=1,\dots, n$, by $g(x_1,\dots,x_n) = n^{-{3/2}}\sum_{i,j=1}^n L_iQ_{ij}L_j -\sqrt{n}D = B_n$ and 
$$
g_k(x_1,\dots, x_{k-1},x_{k+1},\dots, x_n) = n^{-{3/2}}\sum_{\substack{i,j=1\\i\ne k, j\ne k}}^n L_iQ_{ij}L_j -\sqrt{n}D.
$$
By the Efron-Stein inequality \citep[Theorem 9]{Lugosi09},
\begin{align}\label{eq:Efron-Stein}
\Var[B_n] \le  \sum_{k=1}^n \E[(g(x_1,\dots, x_n) - g_k(x_1,\dots,x_{k-1},x_{k+1},\dots, x_n))^2].
\end{align}
Now, for $k\in\{1,\dots, n\}$, $g(x_1,\dots, x_n)$ can be expressed as
$$
n^{-3/2}\left[
		\sum_{\substack{i,j=1\\ i\ne k, j\ne k}}^n L_iQ_{ij}L_j
		\; +\; \sum_{\substack{i=1\\ i\ne k}}^n L_iQ_{ik}L_k \;+\; \sum_{\substack{i=1\\ i\ne k}}^n L_kQ_{ki}L_i
		\;+\; L_kQ_{kk}L_k
	\right] - \sqrt{n}D.
$$
Using the fact that $Q_{ij} = Q_{ij}' = Q_{ji}$, the differences $g-g_k$ in \eqref{eq:Efron-Stein}, are equal to
\begin{align}\label{eq:EfronDiff}
n^{-3/2}\left[
		2\sum_{\substack{i=1\\ i\ne k}}^n L_iQ_{ik}L_k \;+\; L_kQ_{kk}L_k
		\right].
\end{align}
We need to bound the expectation of the squared expression. To this end, we calculate the expectation of $L_iQ_{ik}L_kL_jQ_{jk}L_k$ for arbitrary indices $i,j,k$, as well as in the special case where $i\ne k$, $j\ne k$ and $i\ne j$. Observe that $Q_{ik}$ is the inner product of $\Omega_{11}^{-1/2}R_1(1,x_i')'$ and $\Omega_{11}^{-1/2}R_1(1,x_k')'$ and therefore, by Cauchy-Schwarz inequalities in both Euclidean and $L_p$ space, satisfies $\E[|Q_{ik}|^4]\le \E[\|\Omega_{11}^{-1/2}R_1(1,x_i')'\|^4\|\Omega_{11}^{-1/2}R_1(1,x_k')'\|^4] \le \E[\|\Omega_{11}^{-1/2}R_1(1,x_1')'\|^8] = \E[Q_{11}^4]$. Moreover, parts~\ref{l:mom:lin} and \ref{l:mom:quad2} of Lemma~\ref{lemma:moments}, whose assumptions are implied by the conditions~\ref{a.design}.(\ref{a.factor},\ref{a.moments1},\ref{a.moments2}), establish the facts $\E[L_1^4] = O(|\delta_\gamma'\Omega_{00}\delta_\gamma|^2)$, $\E[L_1^8] = O(|\delta_\gamma'\Omega_{00}\delta_\gamma|^4)$ and $\E[Q_{11}^4] = O(|p+1-q|^4)$, provided that at least one of the assumptions $(a)$ or $(b)$ of Lemma~\ref{lemma:moments}\ref{l:mom:quad2} holds. But this follows from Lemma~\ref{lemma:Tmatrix}, because if $t_0=0$, then from the representations of $\nabla_n$ and $\Delta_\gamma$ in \eqref{eq:Cases} and \eqref{eq:CasesDelta}, we see that the distribution of the quantity of interest does not depend on $\mu$ and we may restrict to $\mu=0$, whereas, if $t_0\ne 0$, Lemma~\ref{lemma:Tmatrix} shows that $T_1$ has full rank.\footnote{It should be noted that the case $t_0=0$ corresponds to a null-hypothesis that does not involve a restriction on the intercept parameter $\alpha$, i.e., $H_0:R_0\gamma = r_0$ can be expressed as $H_0: T_0\beta = r_0$, in this case.} With this, in general, we obtain
\begin{align*}
|\E&[L_iQ_{ik}L_kL_jQ_{jk}L_k]| = |\E[L_iL_jL_k^2Q_{ik}Q_{jk}]| 
\le 
\sqrt{\E[L_i^2L_j^2L_k^4] \E[Q_{ik}^2Q_{jk}^2]}\\
&\le 
(\E[L_1^8])^{1/2} (\E[Q_{ik}^4])^{1/4} (\E[Q_{jk}^4])^{1/4}
\le 
(\E[L_1^8])^{1/2} (\E[Q_{11}^4])^{1/2} \\
&= 
O(|\delta_\gamma'\Omega_{00}\delta_\gamma|^2) \,O(|p+1-q|^2)
= O(q|p+1-q|^2/n^2).
\end{align*}
If $i\ne k$, $j\ne k$ and $i\ne j$, using the abbreviations $v = R_0'\delta_\gamma$ and $M=R_1'\Omega_{11}^{-1}R_1$, we get the smaller bound
\begin{align*}
|\E&[L_iQ_{ik}L_kL_jQ_{jk}L_k]| = |\E[\E[L_iQ_{ik}|x_k,x_j]L_jQ_{jk}L_k^2]| \\
&= |\E[v'SM(1,x_k')'\E[L_jQ_{jk}|x_k]L_k^2]|
=
|\E[(v'SM(1,x_k')')^2L_k^2]|\\
&\le
\sqrt{\E[(v'SM(1,x_k')')^4]\E[L_k^4]}
=
\sqrt{O(|v'SMSMSv|^2)O(|\delta_\gamma'\Omega_{00}\delta_\gamma|^2)}\\
&=
O(\delta_\gamma'\Omega_{01}\Omega_{11}^{-1}\Omega_{10}\delta_\gamma) O(\delta_\gamma'\Omega_{00}\delta_\gamma)
\le
O(|\delta_\gamma'\Omega_{00}\delta_\gamma|^2) = O(q/n^2),
\end{align*}
where we have used Lemma~\ref{lemma:moments} and $\delta_\gamma'\Omega_{01}\Omega_{11}^{-1}\Omega_{10}\delta_\gamma\le \delta_\gamma'\Omega_{00}\delta_\gamma$ again. It is now easy to bound the expectations in \eqref{eq:Efron-Stein}. When squaring the expression in \eqref{eq:EfronDiff} we first note the leading factor $n^{-3}$. Next, we expand the square of the bracket term in \eqref{eq:EfronDiff} and take expectation. From the previous considerations we see that those summands in the resulting sum involving $L_kQ_{kk}L_k$ are of order $O(q|p+1-q|^2/n^2)$, and there are $O(n)$ of them. Together with the leading factor $n^{-3}$ and the summation in \eqref{eq:Efron-Stein} we arrive at a total contribution of $O(q|p+1-q|^2/n^3)$ from all those summands involving $L_kQ_{kk}L_k$. This expression has to be multiplied by $b_n^2 = O(n/q)$ to yield $O(|p+1-q|^2/n^2) = o(1)$. The remaining terms are of the form $|\E[L_iQ_{ik}L_kL_jQ_{jk}L_k]|$ with $i\ne k$ and $j\ne k$. Of those, there are a number of $O(n)$ summands where $i=j$, but they are again of order $O(q|p+1-q|^2/n^2)$ and therefore, as in the case before, their total contribution to \eqref{eq:Efron-Stein} is asymptotically negligible, even after multiplying by $b_n^2$. Finally, there is a number of $O(n^2)$ remaining summands as above, but with $i\ne k$, $j\ne k$ and $i\ne j$. Therefore, by the refined bound above, they are of order $O(q/n^2)$, so that their total contribution to the variance bound in \eqref{eq:Efron-Stein} is $O(q/n^2)$. Together with the factor $b_n^2$ we arrive at an additional term of order $O(1/n) = o(1)$. Hence, we see that the variance of $b_nB_n$ goes to zero as $n\to\infty$ and the proof of Theorem~\ref{thm:main}\ref{thm:A} is finished.

\subsubsection{The cases of Theorem~\ref{thm:main}\ref{thm:B} and Theorem~\ref{thm:qfixed}}

In this section we establish \eqref{eq:eta-conv} under the assumptions that the design vectors $x_1,\dots, x_n$ are Gaussian, $p_n/n\to\rho_1\in[0,1)$, $q_n/n\to\rho_2\in[0,\rho_1]$ (cf. the beginning of Section~\ref{sec:proof}) and $\Delta_\gamma = o(q_n/n)$.

We restrict to $q\le p$, because this is assumed in Theorem~\ref{thm:qfixed} and the case $q = p+1$ of Theorem~\ref{thm:main}\ref{thm:B} has already been treated above in Section~\ref{sec:NCP1}, in much higher generality. We use different arguments for the two cases $t_0=0$ and $t_0\ne 0$. If $t_0=0$, then, by \eqref{eq:Cases} and \eqref{eq:CasesDelta}, we see that the quantity of interest is given by
\begin{align*}
&s_n^{-1/2}n\nabla_n/q - \sqrt{n}\Delta_\gamma b_n \\
&\quad=
\sqrt{\frac{n}{q}} (s_nq)^{-1/2}\sqrt{n}
\left(\delta_\gamma'(T_0\hat{\Sigma}_n^{-1}T_0')^{-1}\delta_\gamma
-
\delta_\gamma'(T_0\Sigma^{-1}T_0')^{-1}\delta_\gamma \frac{n-(p+1)+q}{n}\right).
\end{align*}
By Lemma~\ref{lemma:invWish}, this expression has mean zero and variance equal to 
$$
\frac{n}{s_nq^2} (\delta_\gamma'(T_0\Sigma^{-1}T_0')^{-1}\delta_\gamma)^2 \frac{2(n-(p+1)+q)}{n}
= \frac{n}{q}\Delta_\gamma^2 O(1) = o(1).
$$
For the case $t_0\ne 0$ we recall $T= (R_0',R_1')'$ and introduce the matrix $\Sigma_T$ by
$$
\Sigma_T = \Var[T(1,x_1')'] = T\begin{pmatrix} 0 &0 \\ 0 &\Sigma \end{pmatrix}T' = 
\begin{pmatrix} T_0\Sigma T_0' &T_0\Sigma T_1'\\ T_1\Sigma T_0' &T_1\Sigma T_1'\end{pmatrix}
=\begin{pmatrix} \Sigma_{00} &\Sigma_{01}\\ \Sigma_{10} &\Sigma_{11}\end{pmatrix},
$$
and note that by Lemma~\ref{lemma:Tmatrix} the sub matrix $\Sigma_{11} = \Var[R_1(1,x_1')']$ of order $(p+1-q)$ is regular. In the present case the difference in \eqref{eq:eta-conv} is given by 
\begin{align}
&s_n^{-1/2}n\nabla_n/q - \sqrt{n}\Delta_\gamma b_n \notag\\
&\quad= 
(s_nq)^{-1/2}\sqrt{\frac{n}{q}}\times\label{eq:Gausst0NOT0}\\
&\quad\quad\quad\sqrt{n}\delta_\gamma'\left(
\frac{U_0'(I_n-P_{U_1})U_0}{n} - (\Omega_{00} - \Omega_{01}\Omega_{11}^{-1}\Omega_{10})\frac{n-(p+1)+q}{n}
\right)\delta_\gamma.\notag
\end{align}
The distribution of this quantity is slightly more complicated than that of the corresponding object in the case $t_0=0$, because now we have to deal with non-centrality issues due to $\mu\ne 0$. 
We take a closer look at the random part. The joint distribution of the first row of $U_0$ and the first row of $U_1$ is $T(1,x_1')' \thicksim \mathcal N(T(1,\mu')', \Sigma_T)$. Therefore, the conditional distribution of $R_0(1,x_1')'$ given $R_1(1,x_1')'$ is given by $\mathcal N(\tilde{\mu} + \Sigma_{01}\Sigma_{11}^{-1}R_1(1,x_1')', \Sigma_{00\cdot1})$, where $\tilde{\mu} = R_0(1,\mu')'-\Sigma_{01}\Sigma_{11}^{-1}R_1(1,\mu')'$ and $\Sigma_{00\cdot1}=\Sigma_{00}-\Sigma_{01}\Sigma_{11}^{-1}\Sigma_{10}$. Hence, conditional on $U_1$, the rows of $U_{0\cdot1} := U_0 - U_1\Sigma_{11}^{-1}\Sigma_{10}$ are i.i.d. $\mathcal N(\tilde{\mu},\Sigma_{00\cdot1})$ and since this distribution is free of $U_1$, this also means that $U_1$ and $U_{0\cdot1}$ are independent. Also, we clearly have $U_0'(I_n-P_{U_1})U_0 = U_{0\cdot1}'(I_n-P_{U_1})U_{0\cdot1}$. So if $V$ is a random $n\times q$ matrix independent of $U_1$, whose rows are i.i.d. $\mathcal N(0,\Sigma_{00\cdot1})$, then $U_0'(I_n-P_{U_1})U_0$ has the same distribution as $(V+\iota\tilde{\mu}')'(I_n-P_{U_1})(V+\iota\tilde{\mu}') = V'(I_n-P_{U_1})V + \tilde{\mu}\iota'(I_n-P_{U_1})V + V'(I_n-P_{U_1})\iota\tilde{\mu}' + \tilde{\mu}\tilde{\mu}'\iota'(I_n-P_{U_1})\iota$. For the Schur complement of $\Omega_{11}$ in $\Omega$ we use the representation given by Lemma~\ref{lemma:SchurComp}. Plugging this back into \eqref{eq:Gausst0NOT0} and removing the leading $(s_nq)^{-1/2}$ term that converges to a positive constant, it remains to study the limiting behavior of
\begin{align}
&\sqrt{n}\delta_\gamma'\left(\frac{V'(I_n-P_{U_1})V}{n} - \Sigma_{00\cdot1} \frac{n-(p+1)+q}{n}\right)\delta_\gamma\label{eq:Wishart}\\
&\quad+
2\sqrt{n}\delta_\gamma'\tilde{\mu}\frac{\iota'(I_n-P_{U_1})V\delta_\gamma}{n}\label{eq:MixedTerm}\\
&\quad+
\frac{(\delta_\gamma'\tilde{\mu})^2}{1+\nu}\sqrt{n}\left( \frac{\iota'(I_n-P_{U_1})\iota(1+\nu)}{n}\frac{n}{n-(p+1)+q} - 1\right)
\frac{n-(p+1)+q}{n},\label{eq:nonCentralChi}
\end{align}
multiplied by $\sqrt{n/q}$, where $\nu = (1,\mu')R_1'\Sigma_{11}^{-1}R_1(1,\mu')'$ is defined as in Lemma~\ref{lemma:SchurComp}. From that lemma we also see that $ \delta_\gamma'\Sigma_{00\cdot1}\delta_\gamma + (\delta_\gamma'\tilde{\mu})^2/(1+\nu) = \delta_\gamma'\Omega_{00\cdot1}\delta_\gamma = \Delta_\gamma =o(q/n)$. Since $\Sigma_{00\cdot1}$ is the Schur complement of the positive definite matrix $\Sigma_{11}$ within the positive semidefinite matrix $\Sigma_T$, it follows that $\Sigma_{00\cdot1}$ is itself positive semidefinite (consider the minimizer of the quadratic form $u\mapsto (v',u')'\Sigma_T(v',u')'$ in those variables $u\in\R^{p+1-q}$ corresponding to the block $\Sigma_{11}$). Consequently, $\delta_\gamma'\Sigma_{00\cdot1}\delta_\gamma\ge 0$ and both $\delta_\gamma'\Sigma_{00\cdot1}\delta_\gamma$ and $(\delta_\gamma'\tilde{\mu})^2/(1+\nu)$ are bounded by $\delta_\gamma'\Omega_{00\cdot1}\delta_\gamma = \Delta_\gamma = o(q/n)$. Now, we first show that the quantity in \eqref{eq:nonCentralChi} converges to zero in probability. By Lemma~\ref{lemma:NormalProjMat} with $\lambda_n = n\nu$ and $k=p+1-q$, we have
\begin{align*}
&\sqrt{n}\left(\frac{\iota'(I_n-P_{U_1})\iota(1+\nu)}{n}\frac{n}{n-(p+1)+q} - 1\right)\\
&\quad=
\sqrt{n}\left(\frac{\xi/(n-(p+1)+q)}{(\xi + \zeta)/(n(1+\nu))} - 1\right)\\
&\quad
= \frac{\sqrt{n}\left(\xi/(n-(p+1)+q) - 1\right) + \sqrt{n}\left( 1 - (\xi + \zeta)/(n(1+\nu))\right)}
		{(\xi + \zeta)/(n(1+\nu))},
\end{align*}
where $\xi\thicksim \chi_{n-(p+1)+q}^2$ independent of $\zeta\thicksim \chi_{p+1-q}^2(n\nu)$. The term in the denominator has mean $(n + n\nu)/(n(1+\nu)) = 1$ and variance $2(n + 2n\nu)/(n^2(1+\nu)^2) = O(1/n)$, by independence, and thus, converges to $1$ in probability. From the form of these moments we also conclude that the term $\sqrt{n}\left( 1 - (\xi + \zeta)/(n(1+\nu))\right)$ is $O_\P(1)$. Moreover, it is easy to see that also $\sqrt{n}\left(\xi/(n-(p+1)+q) - 1\right)$ is $O_\P(1)$, which entails that the entire expression in the previous display is of order $O_\P(1)$, and hence, the expression in \eqref{eq:nonCentralChi} converges to zero in probability even after multiplying by $\sqrt{n/q}$. Next, the mixed term in \eqref{eq:MixedTerm} is easily seen to have conditional distribution $\mathcal N(0, 4\delta_\gamma'\Sigma_{00\cdot1}\delta_\gamma\iota'(I_n-P_{U_1})\iota(\delta_\gamma'\tilde{\mu})^2/n)$ given $U_1$. By the previous considerations, the conditional variance is equal to 
\begin{align*}
4\delta_\gamma'\Sigma_{00\cdot1}\delta_\gamma\iota'(I_n-P_{U_1})\iota(\delta_\gamma'\tilde{\mu})^2/n 
\;=\; O(\Delta_\gamma^2) \frac{\iota'(I_n-P_{U_1})\iota(1+\nu)}{n} \;=\; O_\P(\Delta_\gamma^2),
\end{align*}
which converges to zero in probability, even after multiplying by $n/q$, and this implies convergence to zero of the scaled mixed term itself. Finally, for the expression in \eqref{eq:Wishart}, we note that conditional on $U_1$, $V'(I_n-P_{U_1})V$ has a Wishart distribution with scale matrix $\Sigma_{00\cdot1}$ and $n-(p+1)+q$ degrees of freedom. Therefore, \eqref{eq:Wishart} has mean zero and $\delta_\gamma'V'(I_n - P_{U_1})V\delta_\gamma \thicksim \delta_\gamma'\Sigma_{00\cdot1}\delta_\gamma\chi_{n-(p+1)+q}^2$ \citep[cf.][Theorem 3.4.2]{Mardia95}, which entails that the variance of \eqref{eq:Wishart}, multiplied by $n/q$, is $(n/q)(\delta_\gamma'\Sigma_{00\cdot1}\delta_\gamma)^2 2(n-(p+1)+q)/n = o(q/n)$, since $\delta_\gamma'\Sigma_{00\cdot1}\delta_\gamma\le \Delta_\gamma = o(q/n)$. This finishes the proofs of Theorem~\ref{thm:main} and Theorem~\ref{thm:qfixed}. \hfill{$\Box$}

\section*{Acknowledgements}

I want to thank Hannes Leeb and the three referees for many valuable suggestions that helped to considerably improve the paper. This research was supported by the Austrian Science Fund (FWF): [P 28233-N32].


\begin{appendix}

\section{Auxiliary results of Section~\ref{sec:DiscAss}}
\label{sec:AppDiscAss}

\begin{lemma}
\label{lemma:spherical}
For $m\in\N$, let $Z=(Z_1,\dots, Z_m)'$ be a spherically symmetric random vector in $\R^m$ such that $\E[ZZ'] = I_m$, let $V\thicksim \mathcal N(0,I_m)$ and let $z_1,\dots, z_n$ be i.i.d. copies of $Z$. For $p\le m$, let $\Gamma$ be a $p\times m$ matrix of full rank $p$, let $\mu\in\R^p$ and define $x_i = \Gamma z_i + \mu$.
\begin{enumerate}
	\renewcommand{\theenumi}{(\roman{enumi})}
	\renewcommand{\labelenumi}{{\theenumi}}
\item \label{lemma:sph-A} Fix $r\in\N$. If $\E[\|Z\|^{2r}]<\infty$ and for every choice of non-negative integers $\ell_j$ with $\sum_{j=1}^m \ell_j\le 2r$, we have $\E[\prod_{j=1}^m Z_j^{\ell_j}] = \prod_{j=1}^m \E[Z_j^{\ell_j}]$, then $\E[Z_1^{2l}] = \E[V_1^{2l}]$ and $\E[\|Z\|^{2l}] = \E[\|V\|^{2l}]$ for every $l=1,\dots, r$.\footnote{Due to symmetry, we always have $\E[Z_1^l]=0=\E[V_1^l]$ if $l$ is odd and the former moment exists.}

\item \label{lemma:sph-B} If $2\le p\le n-2$ and $Z$ also satisfies $\P(\|Z\|=0)=0$, then the random vectors $x_1,\dots, x_n$ satisfy Assumptions~\ref{a.design}.(\ref{a.factor},\ref{a.invertibility}). Moreover, if also $\E[\|Z\|^8]<\infty$, $\E[\|Z\|^4]/\E[\|V\|^4] \to 1$ and $\E[\|Z\|^8]/\E[\|V\|^8] = O(1)$, as $m\to\infty$, then also Assumptions~\ref{a.design}.(\ref{a.moments1},\ref{a.moments2}) hold. 

\item \label{lemma:sph-C} If $Z$ follows the uniform distribution on the ball (of appropriate radius $\sqrt{m+2}$, to ensure $\E[ZZ']=I_m$) and $2 \le p\le n-2$, then the $x_i$, for $i=1,\dots, n$, satisfy the full Assumption~\ref{a.design}, but not Assumption~\ref{c.design}.

\end{enumerate}
\end{lemma}
\begin{proof}
We make use of the well known fact that any spherical distribution can be represented as $Z = b \|Z\|$, where $b$ and $\|Z\|$ are independent, and $b$ is uniformly distributed on the unit $m$-sphere $\mathcal S^{m-1}$ \citep[cf.][]{Camb81}. 

For part~\ref{lemma:sph-A}, set $\ell = \sum_{j=1}^m\ell_j$ and let $e_i\in\R^m$ denote the $i$-th element of the standard basis in $\R^m$ and note that
\begin{align*}
\E\left[\prod_{j=1}^m Z_j^{\ell_j}\right] 
&= \E\left[\prod_{j=1}^m(e_j'Z)^{\ell_j}\right] 
= \E\left[\|Z\|^\ell \prod_{j=1}^m (e_j'b)^{\ell_j}\right]\\
&= \E\left[\|Z\|^\ell\right] \E\left[ \prod_{j=1}^m (e_j'b)^{\ell_j}\right].
\end{align*}
Of course, the same argument can be carried through for the spherical vector $V\thicksim \mathcal N(0,I_m)$, so that we have
\begin{equation}\label{eq:Eprod}
\frac{\E\left[\prod_{j=1}^m Z_j^{\ell_j}\right] }{\E\left[\prod_{j=1}^m V_j^{\ell_j}\right] } = \frac{\E\left[\|Z\|^\ell\right]}{\E\left[\|V\|^\ell\right]},
\end{equation}
provided that all the $\ell_j$ are even, so that $\E[\prod_{j=1}^m V_j^{\ell_j}] \ne 0$.
Now choose the $\ell_j$ to be either equal to $2$ or $0$, such that $\ell$ is any even number from $2$ to $2r$. Therefore, since $\E[Z_1^2] = 1 = \E[V_1^2]$ and by our factorization assumption, the left-hand-side of \eqref{eq:Eprod} is equal to one, so that we have established the equality of even moments of $\|Z\|$ and $\|V\|$. To see that also the even moments of $Z_1$ and $V_1$ coincide, simply choose $\ell_1 = \ell = 2l$, for some $l\in\{1,\dots,k\}$ and $\ell_j = 0$, if $j\ne1$. 

For part~\ref{lemma:sph-B}, to establish Assumption~\ref{a.design}.(\ref{a.invertibility}), first note that the column span of $U_{-1} = [\iota, [z_2,\dots,z_n]'\Gamma' + \iota\mu']$ does not depend on $\mu\in\R^p$, where $\iota = (1,\dots,1)'\in\R^{n-1}$. So we may assume without restriction that $\mu=0$. Moreover, 
$[\iota, [z_2,\dots,z_n]'\Gamma']$ and
$$
[\iota, [z_2,\dots,z_n]'\Gamma']
\begin{bmatrix}
1 &0\\
0 &\Sigma^{-1/2}
\end{bmatrix}
$$
also have the same column span, such that it suffices to determine the rank of $[\iota,[\tilde{x}_2,\dots,\tilde{x}_n]']$, where $\tilde{x}_i = (\Gamma\Gamma')^{-1/2}\Gamma z_i$ is spherically symmetric with $\P(\|\tilde{x}_i\|=0)=0$.
Next, we claim that the matrix $M_k = [\tilde{x}_2,\dots,\tilde{x}_{k+1}]'$ has full rank $p$, almost surely, provided that $k\ge p$. To see this, simply write $M_k = D_1D_2^{-1}\Lambda$, almost surely, where $D_1$ is $k\times k$ diagonal with entries $\|\tilde{x}_2\|,\dots,\|\tilde{x}_{k+1}\|$, $D_2$ is $k\times k$ diagonal with i.i.d. $\chi_p$ entries and $\Lambda$ is $k\times p$ and has i.i.d. $\mathcal N(0,1)$ entries. Now it is easy to see that $D_1D_2^{-1}$ is almost surely of full rank, and thus, $\P(\rank(M_k)=p) = \P(\det{\Lambda'\Lambda}\ne0) = 1$, where the last equality follows from the well known fact that the zero-set of a non-constant polynomial is a Lebesgue null-set. It remains to show that the event $A = \{\iota\in\s(M_{n-1})\}$ has probability zero. Define $B=\{\rank(M_p)=p\}$, $\iota_k = (1,\dots,1)'\in\R^k$ and the function $v: \R^{p\times p} \to \R^p$ by $v(M) = M^{-1}\iota_p$, if $\det{M} \ne 0$, and $v(M)=0$, else. Note that $\P(B)=1$ and $v$ is Borel measurable. Since $p+2\le n$, we see that $A\cap B$ is a subset of the event where both $\iota_p = M_pv(M_p)$ and $1=\tilde{x}_{p+2}'v(M_p)$, the probability of which is clearly bounded by 
\begin{align*}
&\P(1=\tilde{x}_{p+2}'v(M_p)) 
= \P\left(\|\tilde{x}_{p+2}\|^{-1} = (\tilde{x}_{p+2}'/\|\tilde{x}_{p+2}\|)v(M_p), \|\tilde{x}_{p+2}\|\ne 0\right)\\
&\quad= \E\left[
\P\left(\|\tilde{x}_{p+2}\|^{-1} = (\tilde{x}_{p+2}'/\|\tilde{x}_{p+2}\|)v(M_p), \|\tilde{x}_{p+2}\|\ne 0 \,\big|\, M_p, \|\tilde{x}_{p+2}\|\right)
\right].
\end{align*}
But the conditional probability in the previous display is equal to zero, almost surely, because $v(M_p)$, $\tilde{x}_{p+2}/\|\tilde{x}_{p+2}\|$ and $\|\tilde{x}_{p+2}\|$ are independent and $\tilde{x}_{p+2}/\|\tilde{x}_{p+2}\|$ is uniformly distributed on the unit sphere in $\R^p$, and therefore its inner product with any fixed vector has a Lebesgue density on $\R$ provided that $p\ge2$.
For Assumptions~\ref{a.design}.(\ref{a.moments1},\ref{a.moments2}), recall the moments of the $\chi^2$-distribution with $m$-degrees of freedom $\E[\|V\|^{2k}] = \prod_{j=0}^{k-1}(m+2j)$ \citep[cf.][]{Johnson94I}. The same reasoning as in part~\ref{lemma:sph-A} yields
\begin{align*}
\E[|v'Z|^8] = \E[ |v'b|^8] \E[\|Z\|^8] = \frac{\E[\|Z\|^8]}{\E[\|V\|^8]} \E[|v'V|^8] = O(1),
\end{align*}
uniformly in $v\in\mathcal S^{m-1}$. Similarly, for a symmetric matrix $M\in\R^{m\times m}$,
\begin{align*}
\E[(Z'MZ)^2] = \E[(b'Mb)^2] \E[\|Z\|^4] = \frac{\E[\|Z\|^4]}{\E[\|V\|^4]} \E[(V'MV)^2],
\end{align*}
and one easily calculates $\E[(V'MV)^2] = (\trace{M})^2 + 2\trace{M^2}$. Therefore,
\begin{align*}
\Var[Z'MZ] = (\trace{M})^2\left( \frac{\E[\|Z\|^4]}{\E[\|V\|^4]} - 1\right) + 2 \frac{\E[\|Z\|^4]}{\E[\|V\|^4]} \trace{M^2}.
\end{align*}
Finally, for a projection matrix $P\in\R^{m\times m}$, $V'PV$ follows a $\chi^2$-distribution with $\rank{P}=\|P\|_F^2$ degrees of freedom, and thus
\begin{align*}
\left(\E[(Z'PZ)^4]\right)^{1/4} 
&= \left(\frac{\E[\|Z\|^8]}{\E[\|V\|^8]}\right)^{1/4} \left(\E[(V'PV)^4]\right)^{1/4}\\
&= \left(\frac{\E[\|Z\|^8]}{\E[\|V\|^8]}\right)^{1/4} \left(\prod_{j=0}^3(\|P\|_F^2 + 2j)\right)^{1/4}
=  O(\|P\|_F^2).
\end{align*}

To establish part~\ref{lemma:sph-C}, we first verify the conditions of part~\ref{lemma:sph-B}. The finiteness of the $8$-th moment of the radial component and $\P(\|Z\|=0)=0$ are immediate. It is also elementary to calculate the higher non-central moments $\E[\|Z\|^{2k}] = (m+2)^{k}m/(m+2k) $. [Use, for example, the formula for the volume of the $m$-ball of radius $r>0$ to obtain $\P(\|Z\|\le x) = (x/\sqrt{m+2})^m$, for $x\in[0,\sqrt{m+2}]$.] Comparing this to the moments of the $\chi^2_m$ distribution $\E[\|V\|^{2k}] = \prod_{j=0}^{k-1} (m+2j)$ for $k=2,4$, we see that for $m\to\infty$ the moment ratios behave as desired. Therefore, Assumptions~\ref{a.design}.(\ref{a.factor},\ref{a.invertibility},\ref{a.moments1},\ref{a.moments2}) hold in this case. Finally, the validity of Assumption~\ref{a.design}.(\ref{a.Srivastava}) follows from \citet[Section 1.4]{Sriva13}. But Condition~(C1) can not be satisfied in view of part~\ref{lemma:sph-A} (with $r=4$) and the fact that $\E[\|Z\|^4] \ne \E[\|V\|^4]$.
\end{proof}

\begin{lemma}\label{lemma:product}
Let $Z=(Z_1,\dots, Z_m)'$ be a random $m$-vector with $\E[Z]=0$ and $\E[ZZ'] = I_m$ and let $z_1,\dots, z_n$ be i.i.d. copies of $Z$. For $p\le m$, let $\Gamma$ be a $p\times m$ matrix of full rank $p$, let $\mu\in\R^p$ and define $x_i = \Gamma z_i + \mu$.

\begin{enumerate}
	\renewcommand{\theenumi}{(\roman{enumi})}
	\renewcommand{\labelenumi}{{\theenumi}}
\item \label{lemma:product-A} If $Z$ has independent components, whose $8$-th moments are uniformly bounded, then the $x_i$ satisfy Assumptions~\ref{a.design}.(\ref{a.factor},\ref{a.Srivastava},\ref{a.moments1},\ref{a.moments2}). 

\item \label{lemma:product-B} If $z_1,\dots, z_n$ are as in Condition~\ref{c.design} and, in addition, the components of $Z$ have $8$-th moments that are uniformly bounded, then the $x_i$ satisfy Assumptions~\ref{a.design}.(\ref{a.factor},\ref{a.moments1},\ref{a.moments2}).  

\end{enumerate}
\end{lemma}

\begin{proof}
To establish part~\ref{lemma:product-A}, we use the results of \citet{Whittle60}. Theorem~2 in that reference shows that for a unit vector $v = (v_1,\dots, v_m)'\in\R^m$, 
$$
\E[|v'Z|^8] \le C \left(\sum_j v_j^2 (\E[|Z_{j}|^8])^{1/4}\right)^4 \le C\max_j \E[|Z_{j}|^8],
$$
for some numerical constant $C>0$, and thus, $\mathcal L_8 =O(1)$ as $m\to \infty$, in view of uniform boundedness of $\E[|Z_{j}|^8]$. Next, for a symmetric matrix $M\in\R^{m\times m}$, the same theorem yields
$$
\Var[Z'MZ] \le C \sum_{j,k} M_{jk}^2 \sqrt{\E[|Z_{j}|^4]\E[|Z_{k}|^4]}
\le C \max_j \E[|Z_{j}|^4] \trace{M^2},
$$
and, for a projection matrix $P\in\R^{m\times m}$, 
\begin{align*}
&\left|
	(\E[(Z'PZ)^4])^{1/4} - \E[Z'PZ]
\right| 
\le
\left( \E[(Z'PZ - \E[Z'PZ])^4]\right)^{1/4}\\
&\quad\quad\le
C^{1/4} \left( \sum_{j,k}P_{jk}^2 \left(\E[|Z_{j}|^8]\E[|Z_{k}|^8]\right)^{1/4} \right)^{1/2}\\
&\quad\quad\le
\left(C\max_j \E[|Z_{j}|^8] \right)^{1/4} \|P\|_F,
\end{align*}
where the first inequality is the reverse triangle inequality for the $L^4$-norm. Now the previous chain of inequalities implies that
\begin{align*}
(\E[(Z'PZ)^4])^{1/4} \le \E[Z'PZ] + D \|P\|_F \le \|P\|_F^2 + D \|P\|_F^2 = O(\|P\|_F^2),
\end{align*}
since $\E[Z'PZ] = \trace{P} = \trace{P^2} = \|P\|_F^2 = \rank{P}$ is integer, and where $D>0$ is an appropriate constant, not depending on $m$. The validity of Assumption~\ref{a.design}.(\ref{a.Srivastava}) follows from the arguments in Section~1.4 in \citet{Sriva13}.

For part~\ref{lemma:product-B}, simply note that under the factorization assumption in \ref{c.design} all the moments occurring in Conditions~\ref{a.design}.(\ref{a.moments1},\ref{a.moments2}) are identical to those calculated under independence of the components of $Z$. Therefore, the result follows from part~\ref{lemma:product-A}.
\end{proof}


\section{Proofs of auxiliary results of Section~\ref{sec:normality}}
\label{sec:Appendixnormality}
\begin{proof}[Proof of Lemma~\ref{lemma:Bhansali}]
For ease of notation we drop the subscript $n$ that indexes the position of the matrix $A_n$ in the array, i.e., we write $A = A_n$ and denote by $a_{ij}$ the $ij$-th entry of that matrix. Similarly, we write $Z = (Z_1,\dots,Z_n)'$, where $Z_i = Z_{i,n}$. Now, expand
\begin{align*}
Z'AZ - \E[Z'AZ|\mathcal G_n] &\quad=\quad \sum_{i\ne j}^n Z_iZ_ja_{ij} + \sum_{j=1}^na_{jj}(Z_j^2-1)\\
&\quad=\quad \sum_{j=1}^n 
	2Z_j \sum_{i=1}^{j-1} Z_i a_{ij} \quad +\quad \sum_{j=1}^na_{jj}(Z_j^2-1)\\
&\quad = \quad \bar{T}_n \quad+\quad T_n^*, 	
\end{align*}
where we adopt the convention that empty sums are equal to zero. We show that $\bar{T}_n/(\sqrt{2}\|A\|_F) \xrightarrow[]{w} \mathcal N(0,1)$ and $T_n^*/(\sqrt{2}\|A\|_F) \xrightarrow[]{i.p.}0$, as $n\to\infty$. 

The desired convergence of $T_n^*$ follows from the straight forward calculation
\begin{align*}
\E\left[ \left(\frac{T_n^*}{\sqrt{2}\|A\|_F}\right)^2 \Bigg| \mathcal G_n\right] 
&= \sum_{j=1}^n\frac{a_{jj}^2 \E[(Z_j^2-1)^2|\mathcal G_n]}{2\|A\|_F^2}\\
&= \sum_{j=1}^n\frac{a_{jj}^2 (\E[Z_j^4|\mathcal G_n]-1)}{2\|A\|_F^2} \\
&\le \frac{1}{2}\sum_{j=1}^n\frac{a_{jj}^2 \E[Z_j^4|\mathcal G_n]}{\|A\|_F^2},
\end{align*}
and by assumption.

To see the weak convergence of $\bar{T}_n$, for $j=1,\dots, n$, define 
$$
V_{n,j} \quad=\quad \sqrt{2}Z_j \sum_{i=1}^{j-1} Z_ia_{ij}/\|A\|_F,
$$
$\mathcal F_{n,0} = \mathcal G_n$ and $\mathcal F_{n,j} = \sigma\left( \mathcal G_n, Z_i : i\le j\right)$, by which we mean the smallest sigma algebra for which $Z_1,\dots, Z_j$ are measurable and which also contains $\mathcal G_n$. Note that for each $n,j\in\N$, $\mathcal F_{n,j-1} \subseteq \mathcal F_{n,j} \subseteq \mathcal F$, $\|A_n\|_F=\sqrt{\trace{A^2}}$ is $\mathcal F_{n,0}$ measurable and $V_{n,j}$ is $\mathcal F_{n,j}$ measurable. Moreover, we have
\begin{align*}
\frac{\bar{T}_n}{\sqrt{2}\|A\|_F} = \sum_{j=1}^n V_{n,j}.
\end{align*}
Now, by the central limit theorem for dependent random variables \citep[see][and notice the discussion in \citet{Helland82} following eq. (2.7)]{Dvoretzky72, Helland82}, it remains to verify that 
\begin{align}
&\sum_{j=1}^n \E[V_{n,j}|\mathcal F_{n,j-1}] \xrightarrow[n\to\infty]{i.p.}0,\label{eq:CLTmean}\\
&\sum_{j=1}^n \Var[V_{n,j}|\mathcal F_{n,j-1}] \xrightarrow[n\to\infty]{i.p.}1, \quad\text{and}\label{eq:CLTvar}\\
&\sum_{j=1}^n \E[V_{n,j}^2\mathbf{1}_{|V_{n,j}|>\delta}|\mathcal F_{n,j-1}] \xrightarrow[n\to\infty]{i.p.}0 \quad \text{for all $\delta>0$,}\label{eq:CLTbound}
\end{align}
as in equations (2.5)-(2.7) in \citet{Helland82}.
The convergence in \eqref{eq:CLTmean} is trivial, since $\E[V_{n,j}|\mathcal F_{n,j-1}] = 0$, in view of the conditional independence of the $Z_i$ given $\mathcal G_n$.

For \eqref{eq:CLTvar}, abbreviate $T_n = \sum_{j=1}^n \Var[V_{n,j}|\mathcal F_{n,j-1}]  = \sum_{j=1}^n \E[V_{n,j}^2|\mathcal F_{n,j-1}]$ and use conditional independence again to obtain
\begin{align*}
\E[V_{n,j}^2|\mathcal F_{n,j-1}] = 
2\|A\|_F^{-2}  \left(\sum_{i=1}^{j-1} Z_ia_{ij}\right)^2.
\end{align*}
Expanding the squared sum gives
\begin{align*}
 \left(\sum_{i=1}^{j-1} Z_ia_{ij}\right)^2 = \sum_{i,k}^{j-1} Z_iZ_ka_{ij}a_{kj} 
 = \sum_{i=1}^{j-1}Z_i^2a_{ij}^2 + 2 \sum_{i< k}^{j-1} Z_iZ_ka_{ij}a_{kj},
\end{align*}
and therefore, the absolute difference $|T_n - 1|$ can be bounded as
\begin{align*}
|T_n - 1| &\quad=\quad \|A\|_F^{-2}\bigg|
\sum_{j=1}^n 2\left(\sum_{i=1}^{j-1} Z_i^2 a_{ij}^2 + 2 \sum_{i<k}^{j-1} Z_i Z_ka_{ij}a_{kj}\right) - \|A\|_F^2
\bigg|\\
&\quad\le\quad
\|A\|_F^{-2}\Bigg(
2 \bigg|\sum_{i<j}^n (Z_i^2-1)a_{ij}^2\bigg|
+ 4 \bigg|\sum_{j=1}^n\sum_{i<k}^{j-1}Z_iZ_ka_{ij}a_{kj}\bigg| 
+ \bigg|  \sum_{j=1}^n a_{jj}^2 \bigg|
\Bigg).
\end{align*}
To establish the convergence in \eqref{eq:CLTvar}, it remains to show convergence to zero in probability of the terms in absolute values on the last line of the preceding display multiplied by $\|A\|_F^{-2}$. 

First, note that $\|A\|_F^{-2}\sum_{j=1}^n a_{jj}^2$ converges to zero in probability by assumption and because of $\E[Z_j^4|\mathcal G_n]\ge 1$. Now, write $T_{n,1} = \sum_{i<j}^n (Z_i^2-1) a_{ij}^2$ and $T_{n,2} = \sqrt{2}\sum_{j=1}^n\sum_{i<k}^{j-1} Z_iZ_ka_{ij}a_{kj}$ and observe that
\begin{align*}
\E[T_{n,1}^2|\mathcal G_n] &= \sum_{i_1<j_1}^n \sum_{i_2<j_2}^n 
\E[(Z_{i_1}^2-1)(Z_{i_2}^2-1)|\mathcal G_n]a_{i_1j_1}^2a_{i_2j_2}^2\\
&= \sum_{j_1,j_2}^n \sum_{i=1}^{j_1\land j_2 - 1} (\E[Z_i^4|\mathcal G_n] - 1) a_{ij_1}^2 a_{ij_2}^2\\
&\le \left(\max_{j=1,\dots,n} \E[Z_j^4|\mathcal G_n]\right) \sum_{j_1,j_2}^n \sum_{i=1}^{j_1\land j_2 - 1} a_{ij_1}^2 a_{ij_2}^2,
\end{align*}
and
\begin{align*}
\E[T_{n,2}^2|\mathcal G_n] &= 
2\sum_{j_1,j_2}^n \sum_{i_1<k_1}^{j_1-1}\sum_{i_2<k_2}^{j_2-1} 
\E[Z_{i_1}Z_{k_1}Z_{i_2}Z_{k_2}|\mathcal G_n]a_{i_1j_1}a_{k_1j_1}a_{i_2j_2}a_{k_2j_2}\\
&= \sum_{j_1,j_2}^n 2 \sum_{i<k}^{j_1\land j_2-1}a_{i j_1}a_{k j_1}a_{i j_2}a_{k j_2}\\
&\le \left(\max_{j=1,\dots,n} \E[Z_j^4|\mathcal G_n]\right)
\sum_{j_1,j_2}^n  \sum_{i\ne k}^{j_1\land j_2-1}a_{i j_1}a_{k j_1}a_{i j_2}a_{k j_2}.
\end{align*}
Therefore, if we define the triangular truncation operator $\tilde{A}$ of the symmetric matrix $A$ by $\tilde{A} = \sum_{s>t} e_s a_{st} e_t'$, where $e_s\in\R^n$ is the $s$-th element of the standard basis in $\R^n$, we see that
\begin{align}
\trace{\left(\tilde{A}'\tilde{A}\right)^2} &= \trace{\left(\sum_{s_1>t_1}^n \sum_{s_2>t_2}^n e_{t_1}a_{s_1t_1}e_{s_1}'e_{s_2}a_{s_2t_2}e_{t_2}'\right)^2}\notag\\
&=\trace{\left(\sum_{s=1}^n\sum_{t_1,t_2=1}^{s-1}e_{t_1}a_{st_1}a_{st_2}e_{t_2}'\right)^2}\notag\\
&=\sum_{s_1,s_2=1}^n \trace{\sum_{t_1,t_2=1}^{s_1-1}e_{t_1}a_{s_1t_1}a_{s_1t_2}e_{t_2}'\sum_{u_1,u_2=1}^{s_2-1}e_{u_1}a_{s_2u_1}a_{s_2u_2}e_{u_2}'}\notag\\
&= \sum_{s_1,s_2=1}^n \sum_{t_1,t_2=1}^{s_1\land s_2 -1}a_{s_1t_1}a_{s_1t_2}a_{s_2t_2}a_{s_2t_1},\label{tracetilde}
\end{align}
and, in turn, that
\begin{align}
	\begin{split}\label{T1+T2}
	&\E[T_{n,1}^2|\mathcal G_n] + \E[T_{n,2}^2|\mathcal G_n] \\
	&\quad\le \left(\max_{j=1,\dots,n} \E[Z_j^4|\mathcal G_n]\right) 
	\sum_{j_1,j_2}^n  \sum_{i, k}^{j_1\land j_2-1}a_{i j_1}a_{k j_1}a_{i j_2}a_{k j_2}\\
	&\quad=\left(\max_{j=1, \dots, n} \E[Z_j^4|\mathcal G_n]\right)  \trace{(\tilde{A}'\tilde{A})^2}.
	\end{split}
\end{align}
Now, convergence to zero of $\|A\|_F^{-2}(|T_{n,1}|+|T_{n,2}|)$ in probability follows from the above considerations and Lemma~2.1 in \citet{Bhansali07}, which yields the inequality
\begin{align*}
\E\left[
\|A\|_F^{-4}(|T_{n,1}|+|T_{n,2}|)^2\Big|\mathcal G_n \right]
&\le 2 \left(\max_{j=1, \dots, n} \E[Z_j^4|\mathcal G_n]\right) 
\|A\|_F^{-4} \|\tilde{A}'\tilde{A}\|_F^2 \\
&\le
2C^2  \left(\max_{j=1, \dots, n} \E[Z_j^4|\mathcal G_n]\right) \frac{\|A\|_S^2}{\|A\|_F^2},
\end{align*}
where $C>0$ is a global constant, not depending on $n$. Thus, by assumption, the bound on the far right-hand-side of the preceding display converges to zero, in probability, which establishes the convergence in \eqref{eq:CLTvar}.

Finally, for \eqref{eq:CLTbound} we abbreviate $m_n = \max_j  \E[Z_j^4|\mathcal G_n]$ and use the upper bound $V_{n,j}^2\mathbf{1}_{|V_{n,j}|>\delta}\le \delta^{-2}V_{n,j}^4$. Now, 
\begin{align*}
&\E\left[ \sum_{j=1}^n\E[V_{n,j}^4|\mathcal F_{n,j-1}] \Bigg| \mathcal G_n \right] 
= 4 \|A\|_F^{-4} \sum_{j=1}^n  \E[Z_j^4|\mathcal G_n] 
	\E\left[\left(\sum_{i=1}^{j-1} Z_i a_{ij} \right)^4\Bigg|\mathcal G_n\right]\\
&\quad=
4\|A\|_F^{-4} 
\sum_{j=1}^n  \E[Z_j^4|\mathcal G_n] 
	\left(
		3\sum_{i_1\ne i_2}^{j-1} a_{i_1j}^2a_{i_2j}^2
		+ \sum_{i=1}^{j-1}\E[Z_i^4|\mathcal G_n]a_{ij}^4
	\right)\\
&\quad\le
4 m_n(m_n+ 3) 
\|A\|_F^{-4}\sum_{j=1}^n \sum_{i_1,i_2=1}^{j-1}a_{i_1j}^2a_{i_2j}^2,
\end{align*}
and furthermore
\begin{align*}
\|A\|_F^{-4} \sum_{j=1}^n \sum_{i_1,i_2=1}^{j-1}a_{i_1j}^2a_{i_2j}^2
&= \|A\|_F^{-4} \sum_{j=1}^n \left(\sum_{i=1}^{j-1}a_{ij}^2\right)^2\\
&\le
\|A\|_F^{-4} \left(\max_j\sum_{i=1}^na_{ij}^2\right) \sum_{i,j=1}^n a_{ij}^2\\
&= \frac{\max_j (A^2)_{jj}}{\|A\|_F^2}.
\end{align*}
Together with $m_n\ge 1$ and our assumption, this implies that the upper bound on the second-to-last display converges to zero in probability.
\end{proof}


\begin{proof}[Proof of Lemma~\ref{lemma:ProjDiag}]
For convenience, we drop the subscript $n$ that indexes the position in the array whenever there is no risk of confusion. Let $w_i' = (1,x_i')R$ denote the $i$-th row of the matrix $W$ and define $\tilde{w}_i = \Omega_W^{-1/2} w_i$, $\tilde{W} = W\Omega_W^{-1/2}$ and $S_1 = \tilde{W}'\tilde{W} - \tilde{w}_1\tilde{w}_1'=\sum_{i=2}^n \tilde{w}_i\tilde{w}_i' = \Omega_W^{-1/2}R'U_{-1}'U_{-1}R\Omega_W^{-1/2}$, where 
$$\Omega_W = \E[w_1w_1'] = R'\begin{bmatrix} 1 &\mu'\\ \mu &\Sigma+\mu\mu' \end{bmatrix}R$$ 
is positive definite and $U_{-1}$ is defined as in Assumption~\ref{a.design}.(\ref{a.invertibility}). This assumption also entails that $W'W$, $\tilde{W}'\tilde{W}$ and $S_1$ are invertible with probability one, where we denote the corresponding null set by $N$. For convenience, we redefine these quantities in an arbitrary invertible and measurable way on $N$. Moreover, we must also have $p_n+2\le n$ under \ref{a.design}.(\ref{a.invertibility}). 

Since $h_j = h_{j,n} = w_j'(W'W)^{-1}w_j$ on $N^c$, permuting the $h_1, \dots, h_n$ is equivalent to a permutation of $w_1, \dots, w_n$, which are i.i.d., and therefore their joint distribution is invariant under permutation. Hence, the $h_j$ are exchangeable random variables. In particular, the $h_j$ are identically distributed and therefore the fact that $\sum_{j=1}^n h_j = \trace{P_W} = k_n$, on $N^c$, entails that $\E[h_1] = k_n/n$. We also note for later use that $\Var[h_1] = \E[h_1^2] - \E[h_1]^2 \le \E[h_1] - \E[h_1]^2 = (1-k_n/n)k_n/n$, since $0\le h_1\le 1$. It only remains to show that the variance actually converges to zero.

As a preliminary consideration, we study $h_1$ in the case where $t_n := k_n/n \to t \in [0,1]$. The general case of possibly non-converging $t_n$ then follows from a standard subsequence argument (see the end of the proof). The case $t\in\{0,1\}$ is immediate, because here $\Var[h_1] \to 0$ as $n\to\infty$, and thus, $h_1 \to t$ in probability, by the arguments in the previous paragraph. Assume now that $t\in(0,1)$. Note that $P_W = P_{\tilde{W}}$ and use the Sherman-Morrison formula to obtain
\begin{align*}
h_1 = \tilde{w}_1'\left( S_1 + \tilde{w}_1\tilde{w}_1' \right)^{-1} \tilde{w}_1 
= \frac{\tilde{w}_1'S_1^{-1}\tilde{w}_1}{1 + \tilde{w}_1'S_1^{-1}\tilde{w}_1},
\end{align*}
at least on $N^c$. For $\alpha \ge 0$, set $J_\alpha = (S_1 + (n-1)\alpha I_{k_n})^{-1}$ and define the random function 
$$
\Psi_n(\alpha) = \frac{\tilde{w}_1'J_\alpha\tilde{w}_1}{1 + \tilde{w}_1'J_\alpha\tilde{w}_1},
$$
which satisfies $\Psi_n(0) = h_1$, almost surely. Since $y \mapsto y/(1+y)$ is non-decreasing on $[0,\infty)$, and $M \mapsto M^{-1}$ is non-increasing on invertible hermitian matrices \citep[cf.][p. 114]{Bhatia97}, the function $\Psi_n$ is non-increasing on $[0,\infty)$. We establish the convergence in probability of $\Psi_n(0)$ by first analyzing the limiting behavior of $\Psi_n(\alpha)$ as $n\to\infty$, for every $\alpha>0$. To this end, we consider the conditional mean and variance of $\tilde{w}_1'J_\alpha\tilde{w}_1$ given $S_1$.

Since $\tilde{w}_1$ and $S_1$ are independent and $\E[\tilde{w}_1\tilde{w}_1'] = I_{k_n}$, one easily calculates $\E[\tilde{w}_1'J_\alpha\tilde{w}_1|S_1] = \trace{J_\alpha}$. The conditional variance is slightly more involved. Abbreviate $\Sigma_W = \Var[w_1]$, $\bar{\mu} = \E[\tilde{w}_1] = \Omega_W^{-1/2}R'(1,\mu')'$ and use Assumption~\ref{a.design}.(\ref{a.factor}) to write 
$$
\tilde{w}_1 = \bar{\mu} + \Omega_W^{-1/2}R'\begin{pmatrix} 0 &0\\ 0 &\Gamma\end{pmatrix} \begin{pmatrix}0\\ z_1\end{pmatrix}.
$$
Also notice that $\Sigma_W = \Omega_W - R'(1,\mu')'(1,\mu')R$, and hence 
$\Omega_W^{-1/2}\Sigma_W\Omega_W^{-1/2} = I_{k_n} - \bar{\mu}\bar{\mu}'$. Since $\Omega_W^{-1/2}\Sigma_W\Omega_W^{-1/2}$ is positive semidefinite, this also implies that $\|\bar{\mu}\|\le 1$.
Now, decompose the quantity of interest
\begin{align}
\tilde{w}_1'J_\alpha\tilde{w}_1 \quad=\quad &\bar{\mu}' J_\alpha \bar{\mu} 
\quad+\quad 2\bar{\mu}' J_\alpha \Omega_W^{-1/2}R'\begin{pmatrix} 0 &0\\ 0 &\Gamma\end{pmatrix} \begin{pmatrix}0\\ z_1\end{pmatrix} \notag\\
&+ (0,z_1') \begin{pmatrix} 0 &0\\ 0 &\Gamma'\end{pmatrix} R \Omega_W^{-1/2} J_\alpha \Omega_W^{-1/2}R'\begin{pmatrix} 0 &0\\ 0 &\Gamma\end{pmatrix} \begin{pmatrix}0\\ z_1\end{pmatrix}. \label{eq:Jalpha}
\end{align}
Conditional on $S_1$, the variance of $\bar{\mu}' J_\alpha \bar{\mu}$ is zero, the variance of half of the mixed term is 
\begin{align*}
\bar{\mu}' J_\alpha \Omega_W^{-1/2}\Sigma_W\Omega_W^{-1/2}J_\alpha \bar{\mu} 
= \bar{\mu}'J_\alpha^2\bar{\mu} - (\bar{\mu}'J_\alpha \bar{\mu})^2
\end{align*}
and the variance of the last term in \eqref{eq:Jalpha} is
\begin{align*}
\Var[(0,z_1')M(0,z_1')'|S_1] &= \Var[z_1'M_{22}z_1|S_1] = O(\trace{M_{22}^2}) + (\trace{M_{22}})^2o(1)\notag \\
&\le 
O(\trace{M^2}) + (\trace{M})^2o(1), 
\end{align*}
by assumption, and where we have abbreviated the symmetric positive semidefinite matrix in between the vectors $(0,z_1')$ and $(0,z_1')'$ in \eqref{eq:Jalpha} by $M$ and used the notation $M_{22}$ to denote its bottom right sub matrix of order $m\times m$. Now, 
$\trace{M} = \trace{J_\alpha \Omega_W^{-1/2}\Sigma_W\Omega_W^{-1/2}} = \trace{J_\alpha} - \bar{\mu}'J_\alpha\bar{\mu}$ and
$\trace{M^2} = \trace{J_\alpha \Omega_W^{-1/2}\Sigma_W\Omega_W^{-1/2} J_\alpha \Omega_W^{-1/2}\Sigma_W\Omega_W^{-1/2}} 
= \trace{J_\alpha^2} - 2\bar{\mu}'J_\alpha^2 \bar{\mu} + (\bar{\mu}'J_\alpha\bar{\mu})^2$. For $\alpha>0$, $\|J_\alpha\|_S \le [\alpha(n-1)]^{-1}$ and $\|\bar{\mu}\|^2\le 1$. Therefore, $\bar{\mu}'J_\alpha\bar{\mu}$ and $\bar{\mu}'J_\alpha^2\bar{\mu}$ converge to zero, almost surely, for every $\alpha>0$. Thus, in order to show that $\Var[\tilde{w}_1'J_\alpha \tilde{w}_1|S_1]$ converges to zero, almost surely, for every $\alpha>0$, it suffices to show that $\trace{J_\alpha^2}\to 0$ as $n\to\infty$, almost surely, and that $\trace{J_\alpha}$ is almost surely convergent. Moreover, if we can even show that $\trace{J_\alpha} \to \psi_\alpha\in[0,\infty)$, almost surely, for every $\alpha>0$, then we also have $\tilde{w}_1'J_\alpha\tilde{w}_1 \to \psi_\alpha$, in probability (since $\E[\tilde{w}_1'J_\alpha\tilde{w}_1|S_1] = \trace{J_\alpha}$), and thus $\Psi_n(\alpha) \to \Psi(\alpha) := \psi_\alpha/(1+\psi_\alpha)$, in probability, for every $\alpha>0$.

Therefore, we need to study the limiting behavior of 
\begin{align}
\trace{(S_1+(n-1)\alpha I_{k_n})^{-\ell}} &
= \frac{k_n}{(n-1)^\ell} \frac{1}{k_n}\sum_{j=1}^{k_n} \frac{1}{(\lambda_j+\alpha)^\ell} \notag\\
&= \frac{k_n}{(n-1)^\ell}\int_0^\infty (y+\alpha)^{-\ell}\, d F^{S_1/(n-1)}(y), \label{eq:ESDconv}
\end{align}
for $\ell = 1, 2$, where $0 \le \lambda_1 \le \dots \le \lambda_{k_n}$ are the ordered eigenvalues of $S_1/(n-1) = \sum_{j=2}^n \tilde{w}_j\tilde{w}_j'/(n-1)$ and $F^{S_1/(n-1)}$ denotes the corresponding empirical spectral distribution function. 
Now one easily verifies the assumptions of Theorem~1.1 in \citet{Bai08}. First note that the $\tilde{w}_j$ are i.i.d. and $\E[\tilde{w}_1\tilde{w}_1'] = I_{k_n}$. Second, by the same argument following \eqref{eq:Jalpha} and for an arbitrary non-random $k_n\times k_n$ matrix $B$ with bounded spectral norm, we have
\begin{align*}
&\E[|\tilde{w}_1'B\tilde{w}_1 - \trace{B}|^2]  = \Var[\tilde{w}_1'(B/2+B'/2)\tilde{w}_1] \\
&\quad= O( \trace{(B/2+B'/2)^2}) + (\trace{B} + O(1))^2o(1) + O(1) \\
&\quad= O(k_n \|B\|_S^2) + k_n^2\|B\|_S^2 o(1) + O(1)= o(n^2).
\end{align*}
Recall also that for now $t_n=k_n/n \to t\in (0,1)$.
Therefore, $F^{S_1/(n-1)}$ converges weakly, almost surely to the Mar\v{c}enko-Pastur distribution with Lebesgue density $f^{MP}(y) = \sqrt{(y-a)(b-y)}/(2\pi t y)$ on $[a,b]$, where $a = (1-\sqrt{t})^2$ and $b = (1+\sqrt{t})^2$. Now we see that we can not use the same strategy to establish the convergence of \eqref{eq:ESDconv} in the case where $\alpha=0$, because the function $h_\alpha(y)=(y+\alpha)^{-1}$ is not bounded on $(0,\infty)$ in that case. However, for $\alpha>0$, $h_\alpha$ is bounded and continuous on $[0,\infty)$ and therefore the integral in \eqref{eq:ESDconv} converges almost surely,
\begin{align*}
\int_0^\infty h_\alpha^{\ell}(y)\, d F^{S_1/(n-1)}(y)
\quad\xrightarrow[]{a.s.}\quad \int_a^b h_\alpha^{\ell}(y) f^{MP}(y)\,dy \;\in\; (0,\infty).
\end{align*}
Since $k_n/(n-1) \to t\in (0,1)$ and $k_n/(n-1)^2 \to 0$, this means that $\trace{J_\alpha} \to \psi_\alpha := t \int_a^b h_\alpha(y) f^{MP}(y)\,dy$ and $\trace{J_\alpha^2}\to 0$, almost surely, for every $\alpha>0$. As discussed at the end of the previous paragraph, this entails that $\Psi_n(\alpha) \to \Psi(\alpha)$, in probability, for every $\alpha>0$, where the function $\Psi$ is given by
$$
\Psi(\alpha) = \frac{t\int_a^b h_\alpha(y) f^{MP}(y)\,dy}{1+ t \int_a^b h_\alpha(y) f^{MP}(y)\,dy}.
$$
Now it is easy to see (e.g., by the dominated convergence theorem) that the function $\Psi$ is continuous and non-increasing on $[0,\infty)$ (recall that here $t<1$ and $a>0$). Moreover, the limiting integral for $\alpha=0$ can be evaluated as $\int_a^b h_0(y) f^{MP}(y)\,dy = 1/(1-t)$ \citep[cf.][Lemma B.1]{Huber13} resulting in $\Psi(0) = t$.

Let us briefly recapitulate what we have found so far. First of all, we have seen that $h_1 = h_{1,n} \to t$ in probability, as $n\to\infty$, if $t_n\to t\in\{0,1\}$. For $t_n \to t\in(0,1)$, we know that $\E[h_{1,n}] = k_n/n \to t$ and $0\le \Psi_n(\alpha) \le \Psi_n(0) = h_1 \le 1$ almost surely. Moreover, $\Psi_n(\alpha) \to \Psi(\alpha)$ in probability, for every $\alpha>0$, and $\Psi(\alpha) \to \Psi(0) = t$ as $\alpha\to 0$. Thus, Lemma~\ref{lemma:XnYn} applies and we obtain that $h_{1,n} \to t$ in probability, also in the case $t\in(0,1)$.

Finally, consider $\Delta_n := | h_1 - k_n/n|$ with arbitrary $t_n = k_n/n \in [0,1]$. Suppose that $c := \limsup \E[\Delta_n] > 0$. Then there exists a subsequence $n'$, such that $\E[\Delta_{n'}] \to c$, as $n'\to \infty$. By compactness, there exists a further subsequence $n''$, such that $t_{n''} \to t\in[0,1]$, as $n'' \to \infty$. But in this case, our previous arguments have shown that $\Delta_{n''} \to 0$ in probability, which also entails that $\E[\Delta_{n''}] \to 0$, by boundedness, contradicting the assertion that $\E[\Delta_{n'}] \to c >0$.
\end{proof}


\section{Other technical results}

\begin{lemma}\label{lemma:sigma}
Under the model \eqref{eq:linmod}, suppose that $\frac{1}{n}\sum_{i=1}^n \E[(\eps_i/\sigma_n)^4|x_i] = O_\P(1)$ and $\rank(U)=p_n+1$, almost surely, for all $n\in\N$. If $\limsup_{n\to\infty} p_n/n<1$, then $\sqrt{n}|\hat{\sigma}_n^2/\sigma_n^2 - 1| = O_\P(1)$. In particular, we have $\P(\hat{\sigma}_n^2 =0) \to 0$ as $n\to\infty$.
\end{lemma}

\begin{remark*}\normalfont
The assumption that $\frac{1}{n}\sum_{i=1}^n \E[(\eps_i/\sigma_n)^4|x_i] = O_\P(1)$ is clearly weaker than a uniform bound on $\E[(\eps_1/\sigma_n)^4]$ or a uniform bound on $\E[\eps_1^4]$ together with $\liminf_n \sigma_n^2 >0$. Clearly, also Assumption~\ref{a.error} implies $\frac{1}{n}\sum_{i=1}^n \E[(\eps_i/\sigma_n)^4|x_i] = O_\P(1)$.
\end{remark*}

\begin{proof}[Proof of Lemma~\ref{lemma:sigma}]
Recall that $\hat{\sigma}_n^2 = \eps'M\eps$, almost surely, where $M := M(X) := (I_n-P_U)/(n-p-1)$ is a function of the design matrix $X$. Note that $y_1,\dots, y_n$ are conditionally independent given $X$, and hence, also $\eps_i =  y_i - \E[y_i|x_i]$, for $i=1,\dots, n$, are conditionally independent given $X$. Therefore, one easily obtains the almost sure identities
\begin{align*}
\E[\hat{\sigma}_n^2/\sigma_n^2|X] &= \trace{M}  \hspace{1cm} \text{and}\\
\Var[\hat{\sigma}_n^2/\sigma_n^2|X] &= 2 \trace{M^2} + \sum_{i=1}^n(\E[(\eps_i/\sigma_n)^4|x_i]-3)M_{ii}^2.
\end{align*}
By our assumption on $U$, with probability one, $\trace{M} = 1$ and $\trace{M^2} = 1/(n-p-1)$, whereas $M_{ii}^2 \le 1/(n-p-1)^2$ holds everywhere, since the diagonal entries of the projection matrix $I_n-P_U$ are always between $0$ and $1$. Taken together, we see that $\E[\hat{\sigma}_n^2/\sigma_n^2|X] = 1$ and $\Var[\hat{\sigma}_n^2/\sigma_n^2|X]\le 2/(n-p-1) + \sum_{i=1}^n\E[(\eps_i/\sigma_n)^4|x_i]/(n-p-1)^2$. Now, the conditional Markov inequality yields
\begin{align*}
\P(\sqrt{n}|\hat{\sigma}_n^2/\sigma_n^2 - 1|>\delta)
&= \E[\P(\sqrt{n}|\hat{\sigma}_n^2/\sigma_n^2 - 1|>\delta|X)\land 1] \\
&\le \E\left[
	\left(\frac{n}{\delta^2} \Var[\hat{\sigma}_n^2/\sigma_n^2|X]\right) \land 1
\right]\\
&\le
\P(n\Var[\hat{\sigma}_n^2/\sigma_n^2|X] > \delta) + \left( \frac{1}{\delta}\land 1\right).
\end{align*}
Since $n\Var[\hat{\sigma}_n^2/\sigma_n^2|X] = O_\P(1)$, in view of the previous considerations and the assumptions $\limsup_{n\to\infty} p_n/n<1$ and $\frac{1}{n}\sum_{i=1}^n \E[(\eps_i/\sigma_n)^4|x_i] = O_\P(1)$, this finishes the proof of the first claim.
The second assertion follows immediately, because of
$\P(\hat{\sigma}_n^2=0) \le \P(|\hat{\sigma}_n^2/\sigma_n^2 - 1| > 1/2)$.
\end{proof}


\begin{lemma}
\label{lemma:XnYn}
For $n\in\N$ and $\alpha>0$, let $h_n$ and $\Psi_n(\alpha)$ be real random variables such that $0\le \Psi_n(\alpha) \le h_n \le 1$ almost surely, and, for $t\in[0,1]$, let $\Psi : [0,\infty)\to [0,1]$ be such that $\Psi(\alpha) \to t$ as $\alpha\to 0$. If for every $\alpha>0$, $\Psi_n(\alpha) \xrightarrow[]{} \Psi(\alpha)$ in probability, and $\E[h_n] \xrightarrow[]{} t$ as $n\to\infty$, then $h_n  \xrightarrow[]{} t$ in probability, as $n\to\infty$.
\end{lemma}

\begin{remark*}\normalfont
Lemma~\ref{lemma:XnYn} is an asymptotic version of the well known fact that a random variable $h$ that satisfies $h\ge t\in\R$ and $\E[h] = t$ must be equal to $t$, almost surely.
\end{remark*}

\begin{proof}[Proof of Lemma~\ref{lemma:XnYn}]
Fix $\delta>0$, choose $\alpha=\alpha(\delta)>0$ such that $|\Psi(\alpha) - t| < \delta/2$ and do the following standard bound,
\begin{align*}
\P(h_n < \E[h_n] - \delta) \quad&\le\quad \P(|\Psi_n(\alpha) - \E[h_n]|> \delta) \\
&\quad\quad+ \P(h_n < \E[h_n] - \delta, \;|\Psi_n(\alpha) - \E[h_n]|\le \delta).
\end{align*}
But $|\Psi_n(\alpha) - \E[h_n]| \le |\Psi_n(\alpha)-\Psi(\alpha)| + |\Psi(\alpha) - t| + |t - \E[h_n]| \le \delta/2 + o_\P(1)$, whereas $|\Psi_n(\alpha) - \E[h_n]|\le \delta$ and $h_n < \E[h_n] - \delta$ together imply that $\Psi_n(\alpha) \ge \E[h_n] - \delta > h_n$, which, by assumption, happens only on a set of probability zero. Therefore, the upper bound in the previous display converges to zero. Now, by boundedness of $h_n$ we have
\begin{align*}
\E[|h_n - \E[h_n]|] \quad&\le\quad \E[|h_n - \E[h_n] + \delta|] + \delta \\
&\le\quad \E[(h_n - (\E[h_n] - \delta))\mathbf{1}_{\left\{ h_n \ge \E[h_n] - \delta\right\}}] \\
&\quad\quad+ \P(h_n < \E[h_n] - \delta) +  \delta\\
&\le\quad \E[h_n\mathbf{1}_{\left\{ h_n \ge \E[h_n] - \delta\right\}}] - \E[h_n]\P(h_n \ge \E[h_n] - \delta) \\
&\quad\quad + 2\delta + o(1),
\end{align*}
and we also see that both $\E[h_n\mathbf{1}_{\left\{ h_n \ge \E[h_n] - \delta\right\}}]$ and $\E[h_n]\P(h_n \ge \E[h_n] - \delta)$ converge to $t$. Since $\delta>0$ was arbitrary, we must have $\limsup \E[|h_n - \E[h_n]|] = 0$ and thus, convergence in probability of $h_n$ to $t$.
\end{proof}


\begin{lemma}
\label{lemma:moments}
For every $n\in\N$, let $x_{1,n},\dots,x_{n,n}$ be i.i.d. random $p_n$-vectors satisfying $x_{i,n} = \mu_n + \Gamma_n z_{i,n}$ as in Assumption~\ref{a.design}.(\ref{a.factor}) with positive semidefinite covariance matrix $\Sigma_n = \Gamma_n\Gamma_n'$. Set $X_n = [x_{1,n},\dots, x_{n,n}]'$, $\hat{\Sigma}_n = X_n'(I_n-P_\iota)X_n/n$ and 
$$
S_n = \E\left[\begin{pmatrix} 1 \\ x_1\end{pmatrix}\begin{pmatrix} 1 &x_1'\end{pmatrix}\right]
= \begin{pmatrix} 1 &\mu_n'\\ \mu_n &\Sigma_n + \mu_n\mu_n'\end{pmatrix}
=
\begin{pmatrix} 0 &0 \\ 0 &\Sigma_n\end{pmatrix} + \begin{pmatrix} 1 \\ \mu_n\end{pmatrix}\begin{pmatrix} 1 &\mu_n'\end{pmatrix}.
$$
Moreover, let $R_n$ be a $k_n \times (p_n+1)$ matrix such that $R_nR_n' = I_{k_n}$ (i.e., $k_n\le p_n+1$) and set $\Omega_n = R_nS_nR_n'$.

\begin{enumerate}
	\setlength\leftmargin{-20pt}
	\renewcommand{\theenumi}{(\roman{enumi})}
	\renewcommand{\labelenumi}{{\theenumi}} 

\item \label{l:mom:lin} Let $u_n\in\R^{p_n+1}$. If $\sup_{\|w\|=1}\E[|w'z_{1,n}|^\ell] = O(1)$ as $n\to\infty$, for some fixed $\ell\in\N$, not depending on $n$, then $\E[|u_n'(1,x_{1,n}')'|^\ell] = O(|u_n'S_nu_n|^{\ell/2})$ as $n\to\infty$.

\item \label{l:mom:cov} Let $v_{n,1}, v_{n,2}\in\R^{p_n}$. If $\sup_{\|w\|=1}\E[|w'z_{1,n}|^4] = O(1)$ as $n\to\infty$, then $\Var[\sqrt{n}v_{n,1}'\hat{\Sigma}_n v_{n,2}] =O(v_{n,1}'\Sigma v_{n,1}v_{n,2}'\Sigma v_{n,2})$.

\item \label{l:mom:consistency} If $\Sigma_n$ is positive definite, $z_{1,n}$ satisfies \ref{a.design}.(\ref{a.Srivastava}) and $\sup_{\|w\|=1}\E[|w'z_{1,n}|^4] = O(1)$, then the design matrix of the transformed data $W_n = [\iota,X]R_n'$ satisfies
$$
\left\|\Omega_n^{-1/2}(W_n'W_n/n)\Omega_n^{-1/2} - I_{k_n}\right\|_S =
\begin{cases}
o_\P(1), &\text{if } k_n/n \to 0,\\
O_\P(1), &\text{if } k_n = O(n).
\end{cases}
$$

\item \label{l:mom:quad1} If $\Sigma_n$ is positive definite and $\Var[z_{1,n}'M z_{1,n}] = O(\trace{M^2}) + (\trace{M})^2o(1)$, as $n\to \infty$, for every symmetric matrix $M\in\R^{m_n\times m_n}$, then we have $\E[|(1,x_{1,n}')R_n'\Omega_n^{-1}R_n(1,x_{1,n}')'|^2] = O(k_n^2)$.

\item \label{l:mom:quad2} Suppose that $\Sigma_n$ is positive definite and that $\sup_{\|w\|=1}\E[|w'z_{1,n}|^8] = O(1)$ and $(\E[|z_{1,n}'Pz_{1,n}|^4])^{1/4} = O(\|P\|_F^2)$, as $n\to\infty$, for every projection matrix $P$ in $\R^{m_n}$, and partition $R_n = [t_1, T_1]$ with $t_1\in\R^{k_n}$. If for every $n\in\N$ either one of $(a)$ $\mu_n=0$, or $(b)$ $\rank{T_1}=k_n$ holds, then $\E[|(1,x_{1,n}')R_n'\Omega_n^{-1}R_n(1,x_{1,n}')'|^4] = O(k_n^4)$.

\end{enumerate}
\end{lemma}

\begin{proof}
For ease of notation we will drop the subscript $n$ whenever there is no risk of confusion. A simple calculation involving the elementary inequality $|a+b|^\ell \le 2^{\ell-1}(|a|^\ell+|b|^\ell)$ and the notation $u_n = u = (u_0,u_{-1}')'$, with $u_{-1}\in\R^{p_n}$, yields
\begin{align*}
&\E[|u'(1,x_1')'|^\ell] = \E[|u'(1,\mu')' + u'(0, z_1'\Gamma')'|^\ell] \\
&\le
2^{\ell-1}\E\left[|u'(1,\mu')'|^\ell + |u_{-1}'\Gamma z_1|^\ell\right]\\
&=
2^{\ell-1} \left( |u'(1,\mu')'(1,\mu')u|^{\ell/2} + |u_{-1}'\Sigma u_{-1}|^{\ell/2}\E\left[|w' z_1|^\ell \right]
\right),
\end{align*}
where $w = \Gamma' u_{-1}/\|\Gamma' u_{-1}\|$, if $\|\Gamma' u_{-1}\|>0$ and $w=0$, else. In the sum $u'(1,\mu')'(1,\mu')u + u_{-1}'\Sigma u_{-1} = u'S_nu$ both summands are non-negative and thus both summands are bounded by $u'S_nu$. Therefore, the upper bound in the previous display is itself bounded by a constant multiple of $|u' S_n u|^{\ell/2}$. This was the claim of part~\ref{l:mom:lin}.

For part~\ref{l:mom:cov}, first note that because the distribution of $\hat{\Sigma}_n$ does not depend on $\mu$, we may assume that $\mu = 0$, without loss of generality. By the same argument as above but with $\mu=0$, $u_0=0$ and $u_{-1}$ is either $v_{n,1}$ or $v_{n,2}$, we see that $\E[|v_{n,s}'x_1|^4] =O(|v_{n,s}'\Sigma v_{n,s}|^2)$, $s=1,2$. Now
\begin{align*}
\Var[\sqrt{n}&v_{n,1}'\hat{\Sigma}_n v_{n,2}] = n \Var[v_{n,1}'X'Xv_{n,2}/n - v_{n,1}'X'\iota\iota'Xv_{n,2}/n^2]\\
&\le 2n\left(
\frac{1}{n^2}\sum_{i=1}^n\Var[v_{n,1}'x_i v_{n,2}'x_i] + \frac{1}{n^4}\Var\left[ \sum_{i,j=1}^n v_{n,1}'x_i v_{n,2}'x_j\right] 
\right)\\
&\le
2\sqrt{\E[|v_{n,1}'x_1|^4]\E[|v_{n,2}'x_1|^4]} + \frac{2}{n^3} \sum_{i,j,k,l=1}^n \E[v_{n,1}'x_i v_{n,2}'x_j v_{n,1}'x_k v_{n,2}'x_l]\\
&=
O(v_{n,1}'\Sigma v_{n,1}v_{n,2}'\Sigma v_{n,2}) + \frac{2}{n^2}  \E[|v_{n,1}'x_1|^2|v_{n,2}'x_1|^2] \\
&\quad+ \frac{4}{n^3} \sum_{i\ne j} \E[v_{n,1}'x_iv_{n,2}'x_i] \E[v_{n,1}'x_jv_{n,2}'x_j] \\
&\quad+
\frac{2}{n^3} \sum_{i\ne j} \E[|v_{n,1}'x_i|^2] \E[|v_{n,2}'x_j|^2],
\end{align*}
which is of order $O(v_{n,1}'\Sigma v_{n,1}v_{n,2}'\Sigma v_{n,2})$ because $\E[|v_{n,s}'x_1|^2] = v_{n,s}'\Sigma v_{n,s}$ and $\E[v_{n,1}'x_1v_{n,2}'x_1] \le \sqrt{\E[|v_{n,1}'x_1|^2] \E[|v_{n,2}'x_1|^2]}$.

For parts~\ref{l:mom:consistency}, \ref{l:mom:quad1} and \ref{l:mom:quad2} we make the following preliminary considerations. First, note that in all three of these statements $\Sigma$ is assumed to be positive definite and thus $\Omega$ is regular. Abbreviate $\bar{\mu}':= (1,\mu')R'\Omega^{-1/2}$ and 
$$
\Sigma_W = R\begin{pmatrix}
0 &0\\ 0 &\Sigma
\end{pmatrix} R'.
$$
Since $\Sigma_W =  \Omega - R(1,\mu')'(1,\mu')R'$ we have 
\begin{align}\label{eq:muDiag}
\Omega^{-1/2}  \Sigma_W \Omega^{-1/2}
= 
I_{k_n} - \bar{\mu}\bar{\mu}' 
= A \begin{pmatrix}
1-\|\bar{\mu}\|^2 & 0\\
0 &I_{k_n-1}
\end{pmatrix}A',
\end{align}
for some orthogonal matrix $A$ whose first column is $\bar{\mu}/\|\bar{\mu}\|$ if $\|\bar{\mu}\|>0$, and $A=I_{k_n}$ if $\bar{\mu}=0$. Here, quantities of dimension $k_n-1$ have to be removed in case $k_n=1$. The matrix $\Omega^{-1/2}  \Sigma_W \Omega^{-1/2}$ in the previous display is positive semidefinite, which means that $0\le \|\bar{\mu}\|\le 1$. For later use, we partition the matrix $B := \Omega^{-1/2} A$ as $B = [b_1, B_1]$ where $b_1 \in\R^{k_n}$ and note that $B$ satisfies
\begin{align}
B'\Sigma_W B = \begin{pmatrix}
1-\|\bar{\mu}\|^2 & 0\\
0 &I_{k_n-1}
\end{pmatrix},\label{eq:SigmaW}
\end{align}
and
$I_{k_n} = B'\Omega B = B'\Sigma_W B + B'R(1,\mu')'(1,\mu')R'B$, which entails that
\begin{align}
B'R(1,\mu')'(1,\mu')R'B = \begin{pmatrix}
\|\bar{\mu}\|^2 & 0\\
0 &0
\end{pmatrix}.\label{eq:Rmu}
\end{align}
This finishes the preliminary considerations.

Now, for the proof of part~\ref{l:mom:consistency}, write the quantity of interest as
\begin{align}
&\left\|
\Omega^{-1/2}(W'W/n)\Omega^{-1/2} - I_{k_n}
\right\|_S
=
\left\|
A'\Omega^{-1/2}(W'W/n)\Omega^{-1/2}A - A'A
\right\|_S\notag\\
&\quad=\left\|
B'(W'W/n)B - I_{k_n}
\right\|_S
\le
\left\|
B'\hat{\Sigma}_WB  - \begin{pmatrix}
1-\|\bar{\mu}\|^2 & 0\notag\\
0 &I_{k_n-1}
\end{pmatrix}
\right\|_S\\
&\quad\quad+
\left\|
B'R(1,\hat{\mu}')'(1,\hat{\mu}')R'B - 
\begin{pmatrix}
\|\bar{\mu}\|^2 & 0\notag\\
0 &0
\end{pmatrix}
\right\|_S\notag\\
&\quad=
\left\|
	\begin{pmatrix}
	b_1'\hat{\Sigma}_Wb_1 - (1-\|\bar{\mu}\|^2) &b_1'\hat{\Sigma}_WB_1\\
	B_1'\hat{\Sigma}_Wb_1 & B_1'\hat{\Sigma}_WB_1 - I_{k_n-1}
	\end{pmatrix}
\right\|_S \label{eq:SigmaWNorm}\\
&\quad\quad +
\left\|
	\begin{pmatrix}
	b_1'R(1,\hat{\mu}')'(1,\hat{\mu}')R'b_1 - \|\bar{\mu}\|^2 &b_1'R(1,\hat{\mu}')'(1,\hat{\mu}')R'B_1\\
	B_1'R(1,\hat{\mu}')'(1,\hat{\mu}')R'b_1 &B_1'R(1,\hat{\mu}')'(1,\hat{\mu}')R'B_1
	\end{pmatrix}
\right\|_S. \label{eq:RmuNorm}
\end{align}
where 
$$\hat{\Sigma}_W = R\begin{pmatrix} 0 &0\\ 0 &\hat{\Sigma}_n\end{pmatrix}R' \quad\text{and}\quad \hat{\mu}=X'\iota/n.$$
For a partitioned matrix as above we have
\begin{align*}
&\left\|\begin{pmatrix}
c_{11} &c_{12}'\\
c_{21} &C_{22}
\end{pmatrix}
\right\|_S^2 
= \sup_{\|w\|=1}
\left\|\begin{pmatrix}
c_{11}w_1 + c_{12}w_{-1} \\
c_{21} w_1 +C_{22}w_{-1}
\end{pmatrix}
\right\|^2 \\
&\quad\le
(|c_{11}|+\|c_{12}\|)^2 + (\|c_{21}\|+\|C_{22}\|_S)^2.
\end{align*}
Therefore, it suffices to show that the norms of the respective blocks are $O_\P(1)$, if $k_n = O(n)$, and converge to zero in probability, if $k_n/n\to 0$. 

We begin with the terms involving $\hat{\mu}$ in \eqref{eq:RmuNorm}. First,
\begin{align*}
&\E\left[ b_1'R(1,\hat{\mu}')'(1,\hat{\mu}')R'b_1\right] -\|\bar{\mu}\|^2\\
&\quad= 
b_1'R\left[\begin{pmatrix}
0 &0\\
0 &\Sigma/n
\end{pmatrix} +
\begin{pmatrix}1 \\ \mu  \end{pmatrix}\begin{pmatrix}1 &\mu'  \end{pmatrix}\right]R'b_1 - \|\bar{\mu}\|^2\\
&\quad= (1-\|\bar{\mu}\|^2)/n \xrightarrow[n\to\infty]{} 0,
\end{align*}
in view of $\E[\hat{\mu}\hat{\mu}'] = \Sigma/n + \mu\mu'$, \eqref{eq:SigmaW} and \eqref{eq:Rmu}. Moreover, the variance satisfies
\begin{align*}
&\Var\left[\left(b_1'R(1,\hat{\mu}')'\right)^2\right] 
=
\frac{1}{n^4}\Var\left[
\sum_{i,j=1}^n b_1'R(1,x_i')' b_1'R(1,x_j')'
\right] \\
&\quad=
\frac{1}{n^4} \Var\left[
\sum_{i\ne j}^n b_1'R[(1,x_i')'(1,x_j') - (1,\mu')'(1,\mu')]R'b_1 \right.\\
&\hspace{4cm}\left.+
\sum_{i=1}^n b_1'R[(1,x_i')'(1,x_i') - S]R'b_1
\right]\\
&\quad\le
\frac{2}{n^4} \left(
	\sum_{i\ne j}^n\sum_{r\ne s}^n \E[b_1'R[(1,x_i')'(1,x_j') - (1,\mu')'(1,\mu')]R'b_1 \times \right.\\
		&\hspace{4cm} b_1'R[(1,x_r')'(1,x_s') - (1,\mu')'(1,\mu')]R'b_1]  \\
		&\hspace{2cm}\left. +\sum_{i=1}^n \E[(b_1'R(1,x_i')'(1,x_i')R'b_1)^2]
\right).
\end{align*}
To work out the combinatorics of the quadruple sum above, abbreviate $F_{ij} = b_1'R[(1,x_i')'(1,x_j') - (1,\mu')'(1,\mu')]R'b_1$ and note that $\E[F_{ij}] = 0$ if $i\ne j$ and $\E[F_{ij}F_{rs}] = 0$ if all four indices are distinct. Moreover, there are only $O(n^3)$ summands in which not all four indices are distinct, i.e., there are only $O(n^3)$ non-zero summands. Moreover, the non-zero summands can always be bounded by 
\begin{align*}
|E[F_{ij}F_{rs}]| &\le \sqrt{\E[F_{ij}^2]\E[F_{rs}^2]} = \E[F_{ij}^2] = \Var[F_{ij}] = \Var[b_1'R(1,x_i')'(1,x_j')R'b_1] \\
&\le \E[((1,x_i')R'b_1)^2((1,x_j')R'b_1)^2] = (\E[((1,x_1')R'b_1)^2])^2,
\end{align*}
if $i\ne j$ and $r\ne s$. Since $\E[(b_1'R(1,x_1')')^2] = b_1'RSR'b_1 =b_1'\Omega b_1 = 1$, by definition of $B$, we see that the quadruple sum in the second-to-last display is of order $O(n^3)$. The remaining sum in the same display is of order $O(n)$, since $\E[(b_1'R(1,x_1')')^4] = O(1)$, by part~\ref{l:mom:lin} and the assumption $\sup_{\|w\|=1} \E[|w'z_1|^4] = O(1)$. Thus, we have shown that $b_1'R(1,\hat{\mu}')'(1,\hat{\mu}')R'b_1 - \|\bar{\mu}\|^2 \to 0$, in probability. 

Next, consider $\|b_1'R(1,\hat{\mu}')'(1,\hat{\mu}')R'B_1\|^2 \le |b_1'R(1,\hat{\mu}')'|^2 \|B_1'R(1,\hat{\mu}')'\|^2$. The first factor in the upper bound was just shown to be $O_\P(1)$. For the second factor note that $\E[\|B_1'R(1,\hat{\mu}')'\|^2] = \trace{(B_1'\Sigma_WB_1/n + B_1'R(1,\mu')'(1,\mu')R'B_1)} = (k_n-1)/n$, by \eqref{eq:SigmaW} and \eqref{eq:Rmu}. Since $\|B_1'R(1,\hat{\mu}')'(1,\hat{\mu}')R'B_1\|_S = \|B_1'R(1,\hat{\mu}')'\|^2$, we see that the spectral norm in \eqref{eq:RmuNorm} is $O_\P(1)$ if $k_n=O(n)$, and converges to zero in probability, if $k_n/n\to 0$.

For the spectral norm in \eqref{eq:SigmaWNorm}, we may restrict to $\mu=0$. First, write $R = [t_1,T_1]$ with $t_1\in\R^{k_n}$ and use \eqref{eq:SigmaW} to see that $\E[b_1'\hat{\Sigma}_Wb_1] - (1-\|\bar{\mu}\|^2) = b_1'\Sigma_Wb_1(n-1)/n - (1-\|\bar{\mu}\|^2) = (1-\|\bar{\mu}\|^2)/n \to 0$, whereas $\Var[b_1'\hat{\Sigma}_Wb_1] = \Var[b_1'T_1\hat{\Sigma}_nT_1'b_1] \to 0$, in view of the result in part~\ref{l:mom:cov} with $v_n = v_{n,1}=v_{n,2}= T_1'b_1/n^{1/4}$, which satisfies $v_n'\Sigma v_n = b_1'\Sigma_Wb_1/\sqrt{n} = (1-\|\bar{\mu}\|^2)/\sqrt{n}\to 0$. For the off-diagonal block $B_1'\hat{\Sigma}_Wb_1$, note that it has mean zero in view of \eqref{eq:SigmaW}. Therefore, $\E[\|B_1'\hat{\Sigma}_Wb_1\|^2] = \sum_{j=1}^{k_n-1} \E[(e_j'B_1'\hat{\Sigma}_Wb_1)^2] = \sum_{j=1}^{k_n-1} \Var[e_j'B_1'\hat{\Sigma}_Wb_1]$, where $e_1,\dots, e_{k_n-1}$ is the standard basis in $\R^{k_n-1}$. Now, $\Var[e_j'B_1'\hat{\Sigma}_Wb_1] = \Var[e_j'B_1'T_1\hat{\Sigma}_nT_1'b_1]$, and part~\ref{l:mom:cov} applies with $v_{n,1} = T_1'B_1e_j/n^{1/4}$ and $v_{n,2} = T_1'b_1/n^{1/4}$, which satisfy $v_{n,1}'\Sigma v_{n,1} = e_j'B_1'\Sigma_W B_1e_j /\sqrt{n} = 1/\sqrt{n}$ and $v_{n,2}'\Sigma v_{n,2} = b_1'\Sigma_W b_1/\sqrt{n} = (1-\|\bar{\mu}\|^2)/\sqrt{n}$, in view of \eqref{eq:SigmaW}. Therefore, $\E[\|B_1'\hat{\Sigma}_Wb_1\|^2] = \sum_{j=1}^{k_n-1} O(1/n) = O(k_n/n)$. Hence, the only remaining term is $\|B_1'\hat{\Sigma}_WB_1 - I_{k_n-1}\|_S \le \| \frac{1}{n}\sum_{i=1}^nB_1'T_1x_ix_i'T_1'B_1 - I_{k_n-1}\|_S + \|B_1'T_1\hat{\mu}\hat{\mu}'T_1'B_1\|_S$. For the second term in the upper bound, one easily finds its expected value to be $(k_n-1)/n$, as in the previous paragraph. For the spectral norm of the remaining covariance term we verify the strong regularity (SR) condition of \citet[Theorem 1.1]{Sriva13} for the random $(k_n-1)$-vectors $\bar{x}_i = B_1'T_1x_i = B_1'T_1\Gamma z_i$. First, note that the $\bar{x}_i$ are independent and isotropic, since $\mu=0$ and $\E[\bar{x}_i\bar{x}_i'] = B_1'T_1 \Sigma T_1' B_1 = B_1'\Sigma_W B_1 = I_{k_n-1}$. Fix a projection matrix $P$ in $\R^{k_n-1}$ and note that $\Gamma'T_1'B_1PB_1'T_1\Gamma$ is a projection matrix in $\R^{m_n}$ of the same rank as $P$. Since the $z_i$ satisfy Assumption~\ref{a.design}.(\ref{a.Srivastava}) and $\|P\bar{x}_1\|^2 = \|PB_1'T_1\Gamma z_1\|^2 = z_1'\Gamma'T_1'B_1PB_1'T_1\Gamma z_1 = \|\Gamma'T_1'B_1PB_1'T_1\Gamma z_1\|^2$, we see that the (SR) condition holds for $\bar{x}_1$ and with the same constants $c,C$ as in \ref{a.design}.(\ref{a.Srivastava}). Therefore, Corollary~1.4 of \citet{Sriva13} shows that $\| \frac{1}{n}\sum_{i=1}^nB_1'T_1x_ix_i'T_1'B_1 - I_{k_n-1}\|_S$ is $O_\P(1)$ if $k_n=O(n)$, and converges to zero, in probability, if $k_n/n\to 0$. This finishes part~\ref{l:mom:consistency}.

For the proof of parts~\ref{l:mom:quad1} and \ref{l:mom:quad2}, take $\ell \in\N$ and consider the elementary bound
\begin{align}
\E[|(1,x_1')&R'\Omega^{-1}R(1,x_1')'|^\ell] 
= \E\left[\left|(1,\mu')R'\Omega^{-1}R(1,\mu')' \right.\right.\notag\\
&\quad\left.\left.+ 2(1,\mu')R'\Omega^{-1}R(0,z_1'\Gamma')' 
+ (0,z_1'\Gamma')R'\Omega^{-1}R(0,z_1'\Gamma')' \right|^\ell\right] \notag\\
&\le
2^{\ell-1} \left(
\|\bar{\mu}\|^{2\ell} + 2^{\ell-1}
	\left\{
		2^{\ell}\E\left[\left|\bar{\mu}'\Omega^{-1/2}R(0,z_1'\Gamma')'\right|^\ell\right]\right.\right.\label{eq:a1ell}\\
		&\quad\left.\left.+
		\E\left[\left|(0,z_1'\Gamma')R'\Omega^{-1}R(0,z_1'\Gamma')'\right|^\ell\right]
	\right\}\notag
\right).
\end{align}
Partition $R=[t_1,T_1]$ as above and abbreviate $M = \Gamma' T_1'\Omega^{-1}T_1 \Gamma$, so that the expectation on the last line of the previous display can be written as $\E[|z_1'Mz_1|^\ell]$. Now, if $\ell=2$, this can be evaluated as $\E[|z_1'Mz_1|^2] = \Var[z_1'Mz_1] + (\E[z_1'Mz_1])^2 = O(\trace{M^2}) + (\trace{M})^2o(1) + (\trace{M})^2$ under the assumption of part~\ref{l:mom:quad1}. Since $\trace{M} = \trace{\Omega^{-1/2}\Sigma_W\Omega^{-1/2}} = k_n-\|\bar{\mu}\|^2$ and $\trace{M^2} = k_n-1 + (1-\|\bar{\mu}\|^2)$, by \eqref{eq:muDiag}, we see that $\E[|z_1'Mz_1|^2] = O(k_n^2)$. Furthermore, $\E[|\bar{\mu}'\Omega^{-1/2}R(0,z_1'\Gamma')|^2] = \bar{\mu}'\Omega^{-1/2}\Sigma_W\Omega^{-1/2}\bar{\mu} = \|\bar{\mu}\|^2 - \|\bar{\mu}\|^4\le 1/4$, which finishes part~\ref{l:mom:quad1}.

For part~\ref{l:mom:quad2} we begin with the expectation in \eqref{eq:a1ell} with $\ell=4$, which can be written as
\begin{align*}
\E[|\bar{\mu}'\Omega^{-1/2}T_1\Gamma z_1|^4] &\le \|\bar{\mu}'\Omega^{-1/2}T_1\Gamma\|^4 \sup_{\|w\|=1}\E[|w'z_1|^4]\\
&= O(|\bar{\mu}'\Omega^{-1/2}\Sigma_W\Omega^{-1/2}\bar{\mu}|^2) = O(1).
\end{align*}
For $\E[|z_1'Mz_1|^4]$, we begin with case $(a)$ $\mu=0$. Then
\begin{align*}
\E[|z_1'Mz_1|^4] = \E\left[\left|
(0,z_1')\begin{pmatrix} 1 &0 \\ 0 &\Gamma'\end{pmatrix}R'\left(R\begin{pmatrix} 1 & 0\\ 0 &\Sigma\end{pmatrix}R'\right)^{-1}R
\begin{pmatrix} 1 &0 \\ 0 &\Gamma\end{pmatrix}(0,z_1')'\right|^4
\right],
\end{align*}
and we denote the matrix corresponding to the quadratic form in the vector $(0,z_1')'$ on the right-hand-side of this display by $P$. Clearly, $P$ is a projection matrix which we partition as
$$
P=\begin{pmatrix}
p_{11} &p_{21}'\\ p_{21} & P_{22}
\end{pmatrix},
$$
with $p_{11}\in[0,1]$. Exploiting the idempotency and symmetry of $P$, one can show that the generalized Schur complement of $p_{11}$ in $P$, i.e, the matrix $P_{22}-p_{21}p_{11}^\dagger p_{21}'$, is again a projection matrix \citep[cf.][Corollary 2.1]{Baksalary04}, where $p_{11}^\dagger= p_{11}^{-1}$, if $p_{11}\ne0$, and $p_{11}^\dagger = 0$, else.\footnote{\citet{Baksalary04} actually prove a more general result. The special case we are interested in here can also be easily derived by direct calculation.} Moreover, since $|\|P_{22}\|_S-\|p_{21}p_{11}^\dagger p_{21}'\|_S|\le \|P_{22} - p_{21}p_{11}^\dagger p_{21}'\|_S \le 1$ and $\|P_{22}\|_S\le\|P\|_S= 1$, we see that $\|p_{21}(p_{11}^\dagger)^{1/2}\|^2 = \|p_{21}p_{11}^\dagger p_{21}'\|_S \le 2$ and that the Frobenius norm of the generalized Schur complement satisfies $\|P_{22}-p_{21}p_{11}^\dagger p_{21}'\|_F^2 = \trace{(P_{22}-p_{21}p_{11}^\dagger p_{21}')} \le \trace{P_{22}} \le \trace{P} = k_n$. Therefore, using our assumptions, we calculate
\begin{align*}
&\E[|z_1'Mz_1|^4] = \E[|z_1'P_{22}z_1|^4] \\
&\quad\le 2^3\left(\E[|z_1'(P_{22} - p_{21}p_{11}^\dagger p_{21}')z_1|^4] + \E[|z_1'p_{21}p_{11}^\dagger p_{21}'z_1|^4]\right)\\
&\quad\le
2^3\left( O(k_n^4) + \E[|(p_{11}^\dagger)^{1/2}p_{21}'z_1|^8]\right) = O(k_n^4).
\end{align*}
Finally, in the case $(b)$, where $\rank{T_1}=k_n$, the matrix $T_1\Sigma T_1'$ in the representation $\Omega = R S R' = T_1\Sigma T_1' + R(1,\mu')'(1,\mu')R'$, is regular and thus we can invert $\Omega$ by the Sherman-Morrison formula to get
\begin{align*}
M &= \Gamma'T_1'\Omega^{-1}T_1\Gamma \\
&= 
\Gamma'T_1'(T_1\Sigma T_1')^{-1}T_1\Gamma 
-
\frac{\Gamma'T_1'(T_1\Sigma T_1')^{-1}R(1,\mu')'(1,\mu')'R'(T_1\Sigma T_1')^{-1}T_1\Gamma }{1 + (1,\mu')R'(T_1\Sigma T_1')^{-1}R(1,\mu')'}.
\end{align*}
Therefore, we make use of the abbreviations $P = \Gamma'T_1'(T_1\Sigma T_1')^{-1}T_1\Gamma $ and $v = \Gamma'T_1'(T_1\Sigma T_1')^{-1}R(1,\mu')'$ to bound the fourth moment of the quadratic form $z_1'M z_1$ by
\begin{align*}
\E[|z_1'M z_1|^4]
&\le
2^3\left(
\E[|z_1'Pz_1|^4]
+
	\E\left[\left(
		\frac{|v'z_1|^2}{1+\|v\|^2}
	\right)^4\right]
\right)\\
&\le
2^3\left( O(\|P\|_F^8) + \left(\frac{\|v\|^2}{1+\|v\|^2}\right)^4 \sup_{\|w\|=1}\E[|w'z_1|^8] \right).
\end{align*}
Since $\|P\|_F^8 = (\trace{P})^4 = k_n^4$, the upper bound is of order $O(k_n^4)$, which finishes the proof of part~\ref{l:mom:quad2}.
\end{proof}


\begin{lemma}
\label{lemma:Tmatrix}
Let $1\le q\le p$ be positive integers. If $T$ is a $(p+1)\times (p+1)$ orthogonal matrix that is partitioned as 
$$
T = \begin{bmatrix}
t_0 &T_0\\
t_1 &T_1
\end{bmatrix},
$$
where $t_0\in\R^q$, $t_1\in\R^{p+1-q}$, $T_0 \in\R^{q\times p}$ and $T_1\in\R^{(p+1-q)\times p}$, then $\|t_0\| > 0$ if and only if, $\rank{T_1} = p+1-q$.
\end{lemma}
\begin{proof}
By orthogonality,
\begin{align*}
I_{p+1} = TT' = 
\begin{bmatrix}
t_0t_0' + T_0T_0' &t_0t_1'+T_0T_1'\\
t_1t_0'+T_1T_0' &t_1t_1'+T_1T_1'
\end{bmatrix},
\end{align*}
and $\|t_0\|^2+\|t_1\|^2 = 1$. Hence, 
$$
T_1T_1' = I_{p+1-q} - t_1t_1'
$$
has eigenvalues $1$, with multiplicity $p-q$, and a single eigenvalue $1-\|t_1\|^2 = \|t_0\|^2$, which is strictly positive if and only if, $\rank{T_1} = p+1-q$.
\end{proof}


\begin{lemma}
\label{lemma:invWish}
If the $n\times p$ random matrix $X$ has i.i.d. rows following the $\mathcal N(0,\Sigma)$-distribution with positive definite $\Sigma$, $v\in\R^{p}$, $v\ne 0$ and $T\in\R^{q\times p}$ has orthonormal rows, then $v'(T\hat{\Sigma}_n^{-1}T')^{-1}v \thicksim v'(T\Sigma^{-1}T')^{-1}v\,\chi_{n-1-(p-q)}^2/n$, where $\hat{\Sigma}_n = X'(I_n-P_\iota)X/n$ is the sample covariance matrix.
\end{lemma}

\begin{proof}
It is well known that $n\hat{\Sigma}_n\thicksim \mathcal W_p(\Sigma, n-1)$ has a Wishart distribution with scale matrix $\Sigma$ and $n-1$ degrees of freedom \citep[e.g.,][Theorem 3.4.4.(c)]{Mardia95}. If $q=p$, then $T$ is orthogonal and $nv'(T\hat{\Sigma}_n^{-1}T')^{-1}v = v'Tn\hat{\Sigma}_nT'v \thicksim v'T\Sigma T' v\, \chi_{n-1}^2 = v'(T\Sigma^{-1}T')^{-1}v\, \chi_{n-1}^2$ \citep[cf.][Theorem 3.4.2]{Mardia95}. So assume that $q<p$. Let $S\in\R^{(p-q)\times p}$ be such that $R = [S',T']'$ is an orthogonal matrix. Then, by block matrix inversion of 
$$
R\hat{\Sigma}_nR' = \begin{pmatrix} S\hat{\Sigma}_nS' & S\hat{\Sigma}_nT'\\ T\hat{\Sigma}_nS' & T\hat{\Sigma}_nT' \end{pmatrix} \thicksim \frac{1}{n}\mathcal W_p(R\Sigma R', n-1),
$$
we see that the matrix $(T\hat{\Sigma}_n^{-1}T')^{-1} = ([0,I_q](R\hat{\Sigma}_nR')^{-1}[0,I_q]')^{-1} = T\hat{\Sigma}_nT' - T\hat{\Sigma}_nS'(S\hat{\Sigma}_nS')^{-1}S\hat{\Sigma}_nT'$ is the Schur complement of $S\hat{\Sigma}_nS'$ in $R\hat{\Sigma}_nR'$, which follows the $\mathcal W_q(\Omega_{22\cdot1}, n-1-(p-q))$-distribution divided by $n$, where $\Omega_{22\cdot1} = T\Sigma T' - T\Sigma S'(S\Sigma S')^{-1} S\Sigma T'$ \citep[cf.][Theorem 3.4.6.(a)]{Mardia95}. Therefore, $nv'(T\hat{\Sigma}_n^{-1}T')^{-1}v \thicksim v'\Omega_{22\cdot 1}v\,\chi_{n-1-(p-q)}^2$, and $\Omega_{22\cdot1} = (T\Sigma^{-1}T')^{-1}$.
\end{proof}


\begin{lemma}
\label{lemma:SchurComp}
Let $\mu\in\R^p$ and $\Sigma$ be a symmetric, positive definite $p\times p$ matrix. Let $T = [R_0', R_1']'$ be a $(p+1)\times (p+1)$ orthogonal matrix such that $R_0\in \R^{q\times (p+1)}$ and set 
\begin{align*}
&S = \begin{pmatrix}
1 &\mu'\\
\mu &\Sigma + \mu\mu'
\end{pmatrix},
\quad
\Omega = TST' = \begin{pmatrix} R_0SR_0' & R_0SR_1'\\ R_1SR_0' &R_1SR_1' \end{pmatrix} = 
\begin{pmatrix}
\Omega_{00} &\Omega_{01}\\ \Omega_{10} &\Omega_{11}
\end{pmatrix},\\
&\Sigma_T = T\begin{pmatrix}0 &0\\ 0 &\Sigma\end{pmatrix}T' = 
\begin{pmatrix}R_0\\R_1\end{pmatrix}[0,I_p]'\Sigma[0,I_p]\begin{pmatrix}R_0' &R_1'\end{pmatrix}
=
\begin{pmatrix} \Sigma_{00} &\Sigma_{01}\\ \Sigma_{10} &\Sigma_{11}\end{pmatrix}.
\end{align*}
If $\Sigma_{11} := R_1[0,I_p]'\Sigma[0,I_p]R_1'$ is regular, then the Schur complement of $\Omega_{11}$ in $\Omega$ is related to the Schur complement of $\Sigma_{11}$ in $\Sigma_T$ by $\Omega_{00} - \Omega_{01}\Omega_{11}^{-1}\Omega_{10} = \Sigma_{00} - \Sigma_{01}\Sigma_{11}^{-1}\Sigma_{10} + \tilde{\mu}\tilde{\mu}'/(1+\nu)$, where $\tilde{\mu} = (R_0-\Sigma_{01}\Sigma_{11}^{-1}R_1)(1,\mu')'$ and $\nu = (1,\mu')R_1'\Sigma_{11}^{-1}R_1(1,\mu')'$.
\end{lemma}

\begin{proof}
First note that $\Omega_{ij} = \Sigma_{ij} + R_i(1,\mu')'(1,\mu')R_j$, for $i,j\in\{0,1\}$. Abbreviate $\tilde{\mu}_i = R_i(1,\mu')'$, for $i=0,1$ and $\nu=\tilde{\mu}_1'\Sigma_{11}^{-1}\tilde{\mu}_1$ and use the Sherman-Morrison formula to write
\begin{align*}
\Omega_{00}& - \Omega_{01}\Omega_{11}^{-1}\Omega_{10} 
=
\Sigma_{00} + \tilde{\mu}_0\tilde{\mu}_0' 
- 	(\Sigma_{01} + \tilde{\mu}_0\tilde{\mu}_1)
	(\Sigma_{11} + \tilde{\mu}_1\tilde{\mu}_1')^{-1}
	(\Sigma_{10} + \tilde{\mu}_1\tilde{\mu}_0')\\
&=
\Sigma_{00} + \tilde{\mu}_0\tilde{\mu}_0' 
- 	(\Sigma_{01} + \tilde{\mu}_0\tilde{\mu}_1')
	\left(\Sigma_{11}^{-1} - \frac{\Sigma_{11}^{-1}\tilde{\mu}_1\tilde{\mu}_1'\Sigma_{11}^{-1}}{1+ \nu}\right)
(\Sigma_{10} + \tilde{\mu}_1\tilde{\mu}_0')\\
&=
\Sigma_{00} - \Sigma_{01}\Sigma_{11}^{-1}\Sigma_{10} \\
&\quad
+ \tilde{\mu}_0\tilde{\mu}_0' 
- \tilde{\mu}_0\tilde{\mu}_1'\Sigma_{11}^{-1}\Sigma_{10} - \Sigma_{01}\Sigma_{11}^{-1} \tilde{\mu}_1\tilde{\mu}_0' 
- \tilde{\mu}_0\tilde{\mu}_1' \Sigma_{11}^{-1}\tilde{\mu}_1\tilde{\mu}_0' \\
&\quad
+ \Sigma_{01}\frac{\Sigma_{11}^{-1}\tilde{\mu}_1\tilde{\mu}_1'\Sigma_{11}^{-1}}{1+ \nu}\Sigma_{10}
+ \tilde{\mu}_0\tilde{\mu}_1'\frac{\Sigma_{11}^{-1}\tilde{\mu}_1\tilde{\mu}_1'\Sigma_{11}^{-1}}{1+ \nu}\Sigma_{10}
+ \Sigma_{01}\frac{\Sigma_{11}^{-1}\tilde{\mu}_1\tilde{\mu}_1'\Sigma_{11}^{-1}}{1+ \nu}\tilde{\mu}_1\tilde{\mu}_0' \\
&\quad+
 \tilde{\mu}_0\tilde{\mu}_1'\frac{\Sigma_{11}^{-1}\tilde{\mu}_1\tilde{\mu}_1'\Sigma_{11}^{-1}}{1+ \nu}\tilde{\mu}_1\tilde{\mu}_0' \\
 &= \Sigma_{00} - \Sigma_{01}\Sigma_{11}^{-1}\Sigma_{10} 
 \quad+\quad
 \tilde{\mu}_0 (1- \nu)\tilde{\mu}_0' + \tilde{\mu}_0\frac{\nu^2}{1+\nu}\tilde{\mu}_0'\\
 &\quad+
 \tilde{\mu}_0
 \left(
 	\frac{\nu}{1+ \nu} - 1
\right)   \tilde{\mu}_1'\Sigma_{11}^{-1}\Sigma_{10}
+
 \Sigma_{01}\Sigma_{11}^{-1}\tilde{\mu}_1
 \left(
 	\frac{\nu}{1+ \nu} - 1
 \right) \tilde{\mu}_0'\\
 &\quad+
  \Sigma_{01}\frac{\Sigma_{11}^{-1}\tilde{\mu}_1\tilde{\mu}_1'\Sigma_{11}^{-1}}{1+ \nu}\Sigma_{10}\\
&=
\Sigma_{00} - \Sigma_{01}\Sigma_{11}^{-1}\Sigma_{10} \\
&\quad+
\left(
	\tilde{\mu}_0\tilde{\mu}_0'
	-\tilde{\mu}_0\tilde{\mu}_1'\Sigma_{11}^{-1}\Sigma_{10}
	-\Sigma_{01}\Sigma_{11}^{-1}\tilde{\mu}_1\tilde{\mu}_0'
	+\Sigma_{01}\Sigma_{11}^{-1}\tilde{\mu}_1\tilde{\mu}_1'\Sigma_{11}^{-1}\Sigma_{10}
\right)/(1+\nu)\\
&=
\Sigma_{00} - \Sigma_{01}\Sigma_{11}^{-1}\Sigma_{10} + \frac{\tilde{\mu}\tilde{\mu}'}{1+\nu},
\end{align*}
where $\tilde{\mu} = \tilde{\mu}_0 - \Sigma_{01}\Sigma_{11}^{-1}\tilde{\mu}_1$.
\end{proof}


\begin{lemma}
\label{lemma:NormalProjMat}
Let $k$, $n$ be positive integers such that $k<n-1$. If $X$ is a random $n\times k$ matrix whose rows are i.i.d. distributed according to $\mathcal N(\mu, \Sigma)$, where $\mu\in\R^k$ and $\Sigma$ is positive definite, then 
\begin{align*}
\frac{1}{n}\iota'(I_n-P_X)\iota \quad\thicksim\quad \frac{\xi}{\xi+\zeta},
\end{align*}
where $\xi$ and $\zeta$ are independent and distributed according to $\xi\thicksim \chi_{n-k}^2$ and $\zeta\thicksim\chi_{k}^2(\lambda_n)$, with non-centrality parameter $\lambda_n=n\mu'\Sigma^{-1}\mu$.
\end{lemma}

\begin{remark*}\normalfont
Lemma~\ref{lemma:NormalProjMat} is a slight variation of Lemma~A.2 in \citet{Leeb09}.
\end{remark*}

\begin{proof}
Note that $P_X = P_{\bar{X}}$, where $\bar{X} = X\Sigma^{-1/2}$ has i.i.d. rows following the $\mathcal N(\Sigma^{-1/2}\mu,I_n)$ distribution. Writing $\hat{\bar{\mu}}_n = \bar{X}'\iota/n$ and $\hat{\bar{\Sigma}}_n = \bar{X}'(I_n-P_\iota)\bar{X}/n = \bar{X}'\bar{X}/n - \hat{\bar{\mu}}_n\hat{\bar{\mu}}_n'$, for the sample mean and sample covariance matrix of the transformed data, we have, at least on an event of probability one,
\begin{align*}
\frac{1}{n}\iota'(I_n-P_X)\iota \quad&=\quad 1-\hat{\bar{\mu}}_n'(\hat{\bar{\Sigma}}_n + \hat{\bar{\mu}}_n\hat{\bar{\mu}}_n')^{-1}\hat{\bar{\mu}}_n\\
&=\quad 1 - \left[
\hat{\bar{\mu}}_n'\hat{\bar{\Sigma}}_n^{-1}\hat{\bar{\mu}} - \frac{(\hat{\bar{\mu}}_n'\hat{\bar{\Sigma}}_n^{-1}\hat{\bar{\mu}}_n)^2}{1 + \hat{\bar{\mu}}_n'\hat{\bar{\Sigma}}_n^{-1}\hat{\bar{\mu}}_n}
\right]\\
&=\quad \frac{1}{1+\hat{\bar{\mu}}_n'\hat{\bar{\Sigma}}_n^{-1}\hat{\bar{\mu}}_n}.
\end{align*}
Since $n\hat{\bar{\Sigma}}_n$ has a standard Wishart distribution with $n-1$ degrees of freedom and is independent of $\hat{\bar{\mu}}_n$, we get from \citet[Theorem 3.4.7]{Mardia95} that, conditional on $\hat{\bar{\mu}}_n$, the quantity $\hat{\bar{\mu}}_n'\hat{\bar{\Sigma}}_n^{-1}\hat{\bar{\mu}}_n = n\|\hat{\bar{\mu}}_n\|^2 (\hat{\bar{\mu}}_n/\|\hat{\bar{\mu}}_n\|)'(n\hat{\bar{\Sigma}}_n)^{-1}(\hat{\bar{\mu}}_n/\|\hat{\bar{\mu}}_n\|)$ has the same distribution as $n\|\hat{\bar{\mu}}_n\|^2/\xi$, where $\xi\thicksim \chi_{n-k}^2$ is independent of $\hat{\bar{\mu}}_n$. The proof is finished upon noting that $\zeta := n\|\hat{\bar{\mu}}_n\|^2 = \|\bar{X}'\iota/\sqrt{n}\|^2 \thicksim \chi_{k}^2(n\mu'\Sigma^{-1}\mu)$ and that 
$1/(1+\zeta/\xi) = \xi/(\xi+\zeta)$.
\end{proof}

\begin{lemma}
\label{lemma:NCFdist}
Let $q_n\le p_n+1< n$ be positive integer sequences such that $\limsup_n p_n/n<1$, and let $\Delta_n$ be a non-negative real sequence of order $o(q_n/n)$. Moreover, let $s_n$ and $b_n$ be as in Theorem~\ref{thm:main}. Then a non-central $F$-distributed random variable $F_{q_n,n-p_n-1}(\lambda_n)$ with $q_n$ and $n-p_n-1$ degrees of freedom and non-centrality parameter $\lambda_n = \Delta_n (n-p_n-1+q_n)$ satisfies
\begin{align*}
s_n^{-1/2}&\left(F_{q_n,n-p_n-1}(\lambda_n) - 1\right) - \sqrt{n}\Delta_n b_n \\
&\hspace{1cm}\xrightarrow[n\to\infty]{w}\quad
\begin{cases}
\mathcal N(0,1), \hspace{0.5cm}&\text{if } q_n\to\infty,\\
(2q)^{-1/2}\chi_q^2 - \sqrt{q/2}, &\text{if } q_n = q,\quad\forall n\in\N.
\end{cases}
\end{align*}
\end{lemma}

\begin{remark*}\normalfont
This result is elementary and follows from basic properties of the non-central $\chi^2$ distribution \citep[cf.][Chapter 29.5]{Johnson94II}. For the convenience of the reader we include a proof nonetheless. The proof also nicely resembles the main steps of the much more involved argument needed to treat the non-Gaussian cases of Theorem~\ref{thm:main} and Theorem~\ref{thm:qfixed}.
\end{remark*}

\begin{proof}
Let $Y_1,\dots, Y_{q_n}, X_1,\dots, X_{n-p_n-1}$ be i.i.d. standard normal and let $\mu = (\mu_1,\dots, \mu_{q_n})'\in\R^{q_n}$ such that $\mu'\mu = \lambda_n$. Then $F_{q_n,n-p_n-1}(\lambda_n)$ can be represented as 
$$
F_{q_n,n-p_n-1}(\lambda_n) \quad\thicksim\quad \frac{\sum_{i=1}^{q_n}(Y_i+\mu_i)^2/q_n}{\sum_{j=1}^{n-p_n-1}X_j^2/(n-p_n-1)}.
$$
Therefore, 
\begin{align*}
&s_n^{-1/2} \left(F_{q_n,n-p_n-1}(\lambda_n) - 1\right) - \sqrt{n}\Delta_n b_n \\
&\;\thicksim\;
s_n^{-1/2} \frac{\frac{1}{q_n} \sum_{i=1}^{q_n} Y_i^2 - \frac{1}{n-p_n-1} \sum_{j=1}^{n-p_n-1}X_j^2 + \frac{2}{q_n}\sum_{i=1}^{q_n}Y_i\mu_i + \frac{\mu'\mu}{q_n} }
			{\frac{1}{n-p_n-1} \sum_{j=1}^{n-p_n-1}X_j^2}
	- \sqrt{n}\Delta_n b_n \\
&\;=	s_n^{-1/2}\left(\frac{1}{q_n} \sum_{i=1}^{q_n} Y_i^2 - \frac{1}{n-p_n-1} \sum_{j=1}^{n-p_n-1}X_j^2\right)\left(1+O_\P((n-p_n)^{-1/2})\right) \\
&\;\quad\quad +\quad \left(1+O_\P((n-p_n)^{-1/2})\right)\frac{2}{q_n\sqrt{s_n}}\sum_{i=1}^{q_n}Y_i\mu_i \\
&\;\quad\quad+ \quad
\frac{\lambda_n}{q_n\sqrt{s_n}}\left(1+O_\P((n-p_n)^{-1/2})\right) - \sqrt{n}\Delta_n b_n.
\end{align*}
For the mixed term $\frac{2}{q_n\sqrt{s_n}}\sum_{i=1}^{q_n}Y_i\mu_i$ we note that it has mean zero and variance equal to $4(q_n^2s_n)^{-1}\lambda_n = 4(q_ns_n)^{-1}\Delta_n (n/q_n) (n-p_n-1+q_n)/n = o(1)$, since $\Delta_n = o(q_n/n)$, by assumption, and $q_ns_n = 2(1+q_n/(n-p_n-1))\ge 2$. So the mixed term converges to zero in probability. For the non-centrality term $\frac{\lambda_n}{q_n\sqrt{s_n}}(1+O_\P((n-p_n)^{-1/2})) - \sqrt{n}\Delta_n b_n$ we first note that 
$$
\sqrt{n}\Delta_n b_n \;=\; \frac{n\Delta_n}{q_n\sqrt{s_n}} \frac{n-p_n-1+q_n}{n} \;=\; \frac{\lambda_n}{q_n\sqrt{s_n}},
$$ 
so that
\begin{align*}
\frac{\lambda_n}{q_n\sqrt{s_n}}&\left(1+O_\P((n-p_n)^{-1/2})\right) - \sqrt{n}\Delta_n b_n
\;=\;
\frac{\lambda_n}{q_n\sqrt{s_n}}  O_\P((n-p_n)^{-1/2})\\
&=\;
\frac{n}{q_n}\Delta_n  \frac{n-p_n-1+q_n}{n} (q_ns_n)^{-1/2} O_\P(\sqrt{q_n/(n-p_n)}) \;=\; o_\P(1),
\end{align*}
because $q_n/(n-p_n) \le n/(n-p_n) = (1-p_n/n)^{-1}$ is bounded in view of our assumption that $\limsup_n p_n/n<1$. The same assumption also guarantees that $n-p_n\to\infty$ as $n\to\infty$, so it remains to establish the appropriate convergence of
\begin{align}
s_n^{-1/2}&\left(\frac{1}{q_n} \sum_{i=1}^{q_n} Y_i^2 - \frac{1}{n-p_n-1} \sum_{j=1}^{n-p_n-1}X_j^2\right)\notag\\
&=\;
s_n^{-1/2}\left(\frac{1}{q_n} \sum_{i=1}^{q_n} (Y_i^2-1) - \frac{1}{n-p_n-1} \sum_{j=1}^{n-p_n-1}(X_j^2-1)\right)\notag\\
&=\;
\sqrt{\frac{n-p_n-1}{n-p_n-1+q_n}} \frac{1}{\sqrt{2q_n}}\sum_{i=1}^{q_n}(Y_i^2-1)\label{eq:chisqifqfixed}\\
&\quad-\;
\sqrt{\frac{q_n}{n-p_n-1+q_n}} \frac{1}{\sqrt{2(n-p_n-1)}}\sum_{j=1}^{n-p_n-1}(X_j^2-1).\label{eq:o1ifqfixed}
\end{align}
First, we consider the case where $q_n\to\infty$. Define the vectors
$$
Z_n := \begin{pmatrix}
\frac{1}{\sqrt{2q_n}}\sum_{i=1}^{q_n}(Y_i^2-1)\\
\frac{1}{\sqrt{2(n-p_n-1)}}\sum_{j=1}^{n-p_n-1}(X_j^2-1)
\end{pmatrix}\hspace{0.5cm} \text{and} \hspace{0.5cm}
v_n := \begin{pmatrix}
\sqrt{\frac{n-p_n-1}{n-p_n-1+q_n}}\\
-\sqrt{\frac{q_n}{n-p_n-1+q_n}}
\end{pmatrix}.
$$ 
By independence and the CLT, $Z_n$ converges weakly to a bivariate standard normal distribution as $n\to\infty$ and $\|v_n\|_2=1$. Thus, $v_n'Z_n\to\mathcal N(0,1)$, weakly, as $n\to\infty$, by compactness of the unit circle.
If $q_n=q\in\N$ does not depend on $n$, then the expression in \eqref{eq:o1ifqfixed} converges to zero in probability and the expression in \eqref{eq:chisqifqfixed} is distributed as $(1+o(1))[(2q)^{-1/2}\chi_q^2 - \sqrt{q/2}]$.
\end{proof}

\end{appendix}


{\small
\bibliographystyle{imsart-nameyear}
\bibliography{Ftest}
}

\end{document}